\documentclass{article}

\usepackage{amssymb}
\usepackage{amsthm}

\usepackage{amsmath, amsfonts, mathtools, subcaption} 
\usepackage[mathscr]{eucal}
\numberwithin{equation}{section}
\newtheorem{proposition}{Proposition}[section] 
\DeclareMathOperator{\opP}{P}
\DeclareMathOperator{\opPz}{P_0}
\DeclareMathOperator{\opPZ}{P^0}
\newcommand{\zr}[1]{\smash{\stackrel{\raisebox{-5pt}{$\scriptscriptstyle 0$}}{#1}}}

\usepackage{authblk}

\newcommand{\sep}{\unskip,\ }
\newcommand{\MSC}[1][2010]{\noindent\textit{{#1} MSC:}}
\newenvironment{keyword}{\noindent\textit{Keywords:}}{}

\begin{document}




\title{Galerkin finite element methods for the Shallow Water equations over 
variable bottom}


\author[1,2]{G. Kounadis}
\author[1,2]{V.A. Dougalis} 

\affil[1]{Department of Mathematics, National and Kapodistrian 
University of Athens, 15784 Zografou, Greece} 
\affil[2]{Institude of Appled and Computational Mathematics, FORTH, 
70013 Heraklion, Greece} 

\date{\vspace{-5ex}}
\maketitle

\begin{abstract}
We consider the one-dimensional shallow water equations (SW) in a finite channel 
with variable bottom topography. We pose several initial-boundary-value problems 
for the SW system, including problems with transparent (characteristic) boundary 
conditions in the supercritical and the subcritical case. We discretize these 
problems in the spatial variable by standard Galerkin-finite element methods and 
prove $L^2$-error estimates for the resulting semidiscrete approximations. We 
couple the schemes with the 4\textsuperscript{th} order-accurate, explicit, 
classical Runge-Kutta time stepping procedure and use the resulting fully 
discrete methods in numerical experiments of shallow water wave propagation over 
variable bottom topographies with several kinds of boundary conditions. We 
discuss issues related to the attainment of a steady state of the simulated 
flows, including the good balance of the schemes. 
\end{abstract}




\section{Introduction}
\label{sec:Introduction}

In this paper we will consider standard Galerkin finite element approximations 
to the one-dimensional system of shallow water equations over a variable bottom 
that we write following \cite{P}, as 
\begin{align} \label{eq:SW} \tag{SW} 
\begin{aligned} 
 & \eta_t + (\eta u)_x + (\beta u)_x = 0, \\ 
 & u_t + \eta_x + uu_x = 0. 
\end{aligned} 
\end{align} 
The system \eqref{eq:SW} approximates the two-dimensional Euler equations of 
water wave theory and models two-way propagation of long waves of finite 
amplitude on the surface of an ideal fluid in a channel with a variable bottom. 
The variables in \eqref{eq:SW} are nondimensional and unscaled; $x\in\mathbb R$ 
and $t\geq 0$ are proportional to position along the channel and time, 
respectively. With the depth variable $z$ taken to be positive upwards, the 
function $\eta=\eta(x,t)$ is proportional to the elevation of the free surface 
from a level of rest corresponding to $z=0$ and $u=u(x,t)$ is proportional to 
the horizontal velocity of the fluid at the free surface. The bottom of the 
channel is defined by the function $z=-\beta(x)$; it will be assumed that 
$\beta(x)>0$, $x\in \mathbb R$, and that the water depth $\eta(x,t)+\beta(x)$ is 
positive for all $x$, $t$. It should be noted that there are several equivalent 
formulations of the system represented by \eqref{eq:SW}, some of which will be 
considered in section \ref{sec:3} of the paper. 

It is well known that given smooth initial conditions $\eta(x,0)=\eta^0(x)$, 
$u(x,0)=u^0(x)$, $x\in\mathbb R$, and smooth bottom topography, the Cauchy 
problem for \eqref{eq:SW} has smooth solutions, in general only locally in $t$. 
In this paper we will be concerned with numerical approximations of 
\eqref{eq:SW} and suppose that its solution is sufficiently smooth so that the 
error estimates of section \ref{sec:2} hold. We will specifically consider three 
initial-boundary-value problems (ibvp's) for \eqref{eq:SW}, posed on the 
spatial interval $[0,1]$: A simple ibvp with vanishing fluid velocity at the 
endpoints and two ibvp's with transparent (characteristic) boundary conditions, 
in the supercitical and subcritical flow cases, respectively. For these types of 
ibvp's there exists a well-posedness theory locally in $t$, cf.\ e.g.\ 
\cite{PT1}, \cite{HPT}, \cite{PT2}. For the formulation and numerical solution 
of ibvp's with transparent boundary conditions see also \cite{Shiue}, 
\cite{Nyc}. In section \ref{sec:2} we will specify in detail these ibvp's and 
summarize their well-posedness theory. 

The literature on the numerical solution of the shallow water equations is vast. 
We will just mention that in recent years there has been considerable interest 
in solving them numerically by Discontinuous Galerkin finite element methods and 
refer the reader to \cite{XZS} and the recent surveys \cite{QZ}, \cite{X}, for an 
overview of issues related to the implementation of such methods in the presence 
of discontinuities and also in two space dimensions. 

In section \ref{sec:2} of the paper we consider the ibvp's previously mentioned, 
discretize them in space by the standard Galerkin finite element method, and 
prove $L^2$-error estimates for the semidiscrete approximations assuming smooth 
solutions of the equations and extending results of \cite{AD1}, \cite{AD2}, to the 
variable bottom case. In section \ref{sec:3} we discretize the semidiscrete 
problem in the temporal variable using he classical fourth-order accurate, 
four-stage explicit Runge-Kutta method. The resulting fully discrete scheme is 
stable under a Courant number stability condition and its convergence has been 
analyzed for \eqref{eq:SW} in the case of a horizontal bottom in \cite{ADK}. We 
use this scheme in a series of numerical experiments simulating shallow water 
wave propagation over variable bottom topography and in the presence of 
absorbing (characteristic) boundary conditions up to the attainment of 
steady-state solutions. We also discuss issues of good balance, cf.\ \cite{BV}, 
\cite{XZS}, of the standard Galerkin method applied to the shallow water 
equations written in balance-law form. 

In the sequel we denote, for integer $m\geq 0$, by $H^m=H^m(0,1)$ the usual 
$L^2$-based real Sobolev spaces of order $m$, and by $\|\cdot\|_m$ their norm. 
The space $H_0^1=H_0^1(0,1)$ will consist of the $H^1$ functions that vanish at 
$x=0,1$. The inner product and norm on $L^2=L^2(0,1)$ will be denoted by 
$(\cdot,\cdot)$, $\|\cdot\|$, respectively, while $C^m$ will be the $m$ times 
continuously differentiable functions on $[0,1]$ The norms of $L^\infty$ and of 
the $L^\infty$-based Sobolev space $W^{1,\infty}$ on $(0,1)$ will be denoted by 
$\|\cdot\|_\infty$, $\|\cdot\|_{1,\infty}$, respectively. $\mathbb P_r$ will be 
the space of polynomials of degree at most $r$.

\section{Initial-boundary-value problems and error estimates} 
\label{sec:2} 

In this section we will specify the initial-boundary-value problems (ibvp's) 
for the shallow water equations to be analyzed numerically, their 
Galerkin-finite element space discretizations and the properties of the 
attendant finite element spaces. We will then prove $L^2$-error estimates for 
these discretizations assuming that the data and the solutions of the ibvp's are 
smooth enough for the purposes  of the error estimation.

\subsection{Semidiscretization of a simple ibvp with vanishing fluid velocity at 
the endpoints} 
\label{sec:2p1} 

We consider first a simple ibvp for \eqref{eq:SW} posed in the finite channel 
$[0,1]$. let $T>0$ be given. We seek $\eta=\eta(x,t)$, $u=u(x,t)$, for $0\leq 
x\leq 1$, $0\leq t\leq T$, satisfying 
\begin{align} 
& \begin{aligned} 
& \eta_t + (\eta u)_x + (\beta u)_x = 0, \\ 
& u_t + \eta_x + uu_x = 0, 
\end{aligned}\quad 
0\leq x\leq 1,\ \ 0\leq t\leq T, \label{eq:2p1} \\ 
& \eta(x,0) = \eta^0(x),\ \ u(x,0)=u^0(x),\ \ 0\leq x\leq 1 \notag \\ 
& u(0,t) = u(1,t) = 0,\ \ 0\leq t\leq T. \notag 
\end{align} 

In \cite{PT1} Petcu and Temam, using an equivalent form of \eqref{eq:2p1}, 
established the existence-uniqueness of solutions $(\eta,u)$ of \eqref{eq:2p1} 
in $H^2\times H^2\cap H_0^1$ for some $T=T(\|\eta^0\|_2$, $\|u^0\|_2$) under the 
hypotheses that $\eta^0\in H^2$, and, say, $\beta\in H^2$, such that $\eta^0(x) 
+ \beta(x)>0$, $x\in [0,1]$, and $u^0\in H^2\cap H_0^1$. Moreover, it holds that 
$\eta(x,t)+\beta(x)>0$ for $(x,t)\in [0,1]\times [0,T]$, i.e.\ the water depth 
is always positive. (This property will be assumed in all the error estimates to 
follow in addition to the sufficient smoothness of $\eta$ and $u$.) 

In order to solve \eqref{eq:2p1} numerically let $0=x_1< x_2<\ldots <x_{N+1}=1$ 
be a quasiuniform partition of $[0,1]$ with $h:= \max_i(x_{i+1}-x_i)$, and for 
integers $k$, $r$ such that $r\geq 2$, $0\leq k\leq r-2$, consider the finite 
element spaces $S_h=\{\varphi\in C^k:\ \varphi\big|_{[x_j,x_{j+1}]}\in \mathbb 
P_{r-1},\ 1\leq j\leq N\}$ and $S_{h,0}=\left\{\varphi\in S_h: 
\varphi(0)=\varphi(1)=0\right\}$. It is well known that given $w\in H^r$, there 
exists $\chi\in S_h$ such that 
\begin{subequations} \label{eq:2p2} 
\begin{equation} \label{eq:2p2a} 
\|w-\chi\| + h \|w'-\chi'\| \leq C h^r \|w^{(r)}\|, 
\end{equation} 
and, in addition, if $r\geq 3$, such that 
\begin{equation} \label{eq:2p2b} 
\|w-\chi\|_2 \leq C h^{r-2} \|w^{(r)}\|, 
\end{equation} 
\end{subequations} 
where $C$ is a constant independent of $h$ and $w$; a similar property holds in 
$S_{h,0}$ provided $w\in H^r\cap H_0^1$. It follows from \eqref{eq:2p2a}, cf.\ 
\cite{DDW}, that if $\opP$ is the $L^2$-projection operator onto $S_h$, then 
\begin{subequations} \label{eq:2p3}
\begin{align} 
& \|\opP w\|_1 \leq C \|w\|_1,\quad \forall w\in H^1, \label{eq:2p3a} \\ 
& \|\opP w\|_\infty \leq C \|w\|_\infty,\quad \forall w\in C^0, 
\label{eq:2p3b} \\ 
& \|\opP w - w\|_{L^\infty} \leq C h^r \|w^{(r)}\|_\infty,\quad \forall w\in 
C^r, \label{eq:2p3c} 
\end{align} 
\end{subequations} 
and that the analogous properties also hold for $\opPz$, the $L^2$-projection 
operator onto $S_{h,0}$. In addition, as a consequence of the quasiuniformity of 
the mesh, the inverse properties 
\begin{equation} \label{eq:2p4} 
\|\chi\|_1\leq C h^{-1}\|\chi\|, \quad \|\chi\|_{j,\infty}\leq C 
h^{-(j+1/2)}\|\chi\|,\quad j=0,1, 
\end{equation} 
hold for $\chi\in S_h$ or $\chi\in S_{h,0}$. 

The standard Galerkin semidiscretization of \eqref{eq:2p1} is defined as 
follows: Seek $\eta_h: [0,T]\to S_h$, $u_h: [0,T]\to S_{h,0}$, such that for 
$t\in [0,T]$ 
\begin{align} \label{eq:2p5} 
\begin{aligned} 
& (\eta_{ht},\varphi) + ((\eta_hu_h)_x,\varphi) + ((\beta u_h)_x,\varphi) = 
0,\quad \forall\varphi\in S_h, \\ 
& (u_{ht},\chi) + (\eta_{hx},\chi) + (u_hu_{hx},\chi) = 0,\quad \forall\chi\in 
S_{h,0}, 
\end{aligned} 
\end{align} 
with initial conditions 
\begin{equation} \label{eq:2p6} 
\eta_h(0) = \opP\eta_0,\quad u_h(0) = \opPz u_0. 
\end{equation} 

We will prove below that the semidiscrete approximations $(\eta_h,u_h)$ 
satisfy an $L^2$-error bound of $\mathcal O(h^{r-1})$. Is is well known that 
this order of accuracy cannot be improved in the case of the standard Galerkin 
finite element method for first-order hyperbolic problems in the presence of 
general nonuniform meshes, \cite{D}, \cite{AD1}; for uniform meshes better 
results are possible, cf.\ \cite{AD1} and the numerical experiments of section 
\ref{sec:3}. 

\begin{proposition} \label{prop:2p1} 
Let $(\eta,u)$ be the solution of \eqref{eq:2p1}, assumed to be sufficiently smooth 
and satisfying $\beta+\eta>0$ in $[0,1]\times[0,T]$, where $\beta\in C^1$, 
$\beta>0$. Let $r\geq 3$ and $h$ be sufficiently small. Then, the semidiscrete 
ivp \eqref{eq:2p5}--\eqref{eq:2p6} has a unique solution $(\eta_h,u_h)$ for 
$t\in [0,T]$, such that 
\begin{equation} \label{eq:2p7} 
\max_{0\leq t\leq T} \left(\|\eta-\eta_h\| + \|u-u_h\|\right) \leq C h^{r-1}, 
\end{equation} 
where, here and in the sequel, $C$ will denote a generic constant independent of 
$h$. 
\end{proposition}

\begin{proof} 
As the proof is similar to that of Proposition 2.2 in \cite{AD1}, which is valid 
in the case of horizontal bottom ($\beta(x)=1$), we will only indicate the steps 
where the two proofs differ. We let $\rho := \eta-\opP \eta$, $\theta := \opP 
\eta-\eta_h$, $\sigma := u-\opPz u$, $\xi:= \opPz u-u_h$. While the solution 
exists we have 
\begin{gather} 
\begin{multlined} \hspace{-.5em}
(\theta_t,\phi) + (\beta(\xi_x+\sigma_x),\phi) + (\beta_x(\xi+\sigma),\phi) +  
((\eta u)_x-(\eta_hu_h)_x,\phi) = 0,\\ \forall\phi\in S_h, 
\end{multlined} \label{eq:2p8} \\ 
(\xi_t,\chi) + (\theta_x+\rho_x,\chi) + (uu_x-u_hu_{hx},\chi) = 0, \quad 
\forall\chi\in S_{h,0}. \label{eq:2p9} 
\end{gather}
Taking $\phi=\theta$ in \eqref{eq:2p8} and integrating by parts we have 
\begin{multline} \label{eq:2p10} 
\tfrac{1}{2}\tfrac{\mathrm d}{\mathrm dt}\|\theta\|^2 + 
([(\beta+\eta)\xi]_x,\theta) = 
- (\beta\sigma_x,\theta) - (\beta_x\sigma,\theta) - ((\eta\sigma)_x,\theta) - 
((u\rho)_x,\theta)\\ - ((u\theta)_x,\theta) + ((\rho\sigma)_x,\theta) + 
((\theta\sigma)_x,\theta) + ((\rho\xi)_x,\theta) + ((\theta\xi)_x,\theta).  
\end{multline} 
In view of \eqref{eq:2p6}, we conclude by continuity that there exists a maximal 
temporal instance $t_h>0$ such that $(\eta_h,u_h)$ exist and 
$\|\xi_x\|_\infty\leq 1$ for $t\leq t_h$. Suppose that $t_h<T$. Using the 
approximation and inverse properties of $S_h$ and $S_{h,0}$, we may then 
estimate the various terms in the r.h.s.\ of \eqref{eq:2p10} for $t\in [0,t_h]$ 
in a similar way as in \cite{AD1}, since $\beta\in C^1$, and conclude that for 
$t\in [0,t_h]$ 
\begin{equation} \label{eq:2p11} 
\tfrac{1}{2}\tfrac{\mathrm d}{\mathrm dt}\|\theta\|^2 - (\gamma,\theta_x) \leq 
C(h^{r-1}\|\theta\| + \|\theta\|^2 + \|\xi\|^2), 
\end{equation} 
where we have put $\gamma:= (\beta+\eta)\xi$. 

We turn now to \eqref{eq:2p9} in which we take $\chi = \opPz\gamma = 
\opPz[(\beta+\eta)\xi]$. For $0\leq t\leq t_h$ it follows that 
\begin{multline} \label{eq:2p12} 
(\xi_t,\gamma) + (\theta_x,\opPz\gamma) = -(\rho_x,\opPz\gamma) - 
((u\xi)_x,\opPz\gamma) - ((u\sigma)_x,\opPz\gamma) \\ 
+ ((\sigma\xi)_x,\opPz\gamma) + (\sigma\sigma_x,\opPz\gamma) + 
(\xi\xi_x,\opPz\gamma). 
\end{multline} 
Arguing now as in \cite{AD1}, since $\beta\in C^1$, noting that 
\[ 
((u\xi)_x,\opPz\gamma) = ((u\xi)_x,\opPz\gamma-\gamma) + (u_x(\beta+\eta),\xi^2) 
- \tfrac 1 2 ([\beta+\eta)u_x],\xi^2), 
\] 
and using a well-known \emph{superapproximation} property of $S_{h,0}$ to 
estimate the term $\opPz\gamma-\gamma$: 
\[
\|\opPz\gamma-\gamma\| = \|\opPz[(\beta+\eta)\xi] - (\beta+\eta)\xi\| \leq C h 
\|\xi\|, 
\]
we get 
\[ 
|((u\xi)_x,\opPz\gamma)| \leq C h \|\xi\|_1 \|\xi\| + C \|\xi\|^2 \leq C 
\|\xi\|^2. 
\] 
With similar estimates as in \cite{AD1}, using the hypothesis that 
$\|\xi_x\|_\infty\leq 1$ for $0\leq t\leq t_h$, we conclude from this inequality 
and \eqref{eq:2p12} that for $0\leq t\leq t_h$ 
\begin{equation} \label{eq:2p13} 
(\xi_t,(\beta+\eta)\xi) + (\theta_x,\opPz\gamma) \leq C(h^{r-1}\|\xi\| + 
\|\xi\|^2). 
\end{equation} 
Adding now \eqref{eq:2p12} and \eqref{eq:2p13} we obtain 
\[ 
\tfrac{1}{2}\tfrac{\mathrm d}{\mathrm dt}\|\theta\|^2 + (\xi_t,(\beta+\eta)\xi) 
+ (\theta_x,\opPz\gamma-\gamma) \leq C[h^{r-1}(\|\theta\|+\|\xi\|) + 
\|\theta\|^2 + \|\xi\|^2]. 
\] 
But, since $\beta=\beta(x)$, we have $(\xi_t,(\beta+\eta)\xi) = 
\tfrac{1}{2}\tfrac{\mathrm d}{\mathrm dt}((\beta+\eta)\xi,\xi) - 
\tfrac{1}{2}(\eta_t\xi,\xi)$. Therefore, for $0\leq t\leq t_h$ 
\[ 
\tfrac{1}{2}\tfrac{\mathrm d}{\mathrm dt}[\|\theta\|^2 + ((\beta+\eta)\xi,\xi)] 
\leq C[h^{r-1}(\|\theta\| + \|\xi\|) + \|\theta\|^2 + \|\xi\|^2], 
\] 
for a constant $C$ independent of $h$ and $t_h$. Since $\beta+\eta > 0$, the 
norm $((\beta+\eta)\,\cdot,\cdot)^{1/2}$ is equivalent to that of $L^2$ 
uniformly for $t\in [0,T]$. Hence, Gronwall's inequality and \eqref{eq:2p6} 
yield for a constant $C=C(T)$ 
\begin{equation} \label{eq:2p14} 
\|\theta\| + \|\xi\| \leq C h^{r-1}\quad \text{for}\quad 0\leq t\leq t_h. 
\end{equation} 
We conclude from \eqref{eq:2p14}, using inverse properties, that 
$\|\xi_x\|_\infty \leq C h^{r-5/2}$ for $0\leq t\leq t_h$, and, since $r\geq 3$, 
if $h$ is taken sufficiently small, we see that $t_h$ is not maximal. Hence we 
may take $t_h=T$ and \eqref{eq:2p7} follows from \eqref{eq:2p14}. 
\end{proof} 
The hypothesis 
that $r\geq 3$ seems to be technical, as numerical experiments indicate that 
\eqref{eq:2p7} apparently holds for $r=2$ as well, cf.\ \cite{AD1}.

\subsection{Semidiscretization of an ibvp with absorbing (characteristic) 
boundary conditions in the supercritical case} 
\label{sec:2p2} 

We consider now the shallow water equations with variable bottom with 
transparent (characteristic) boundary conditions. First we examine the 
\emph{supercritical} case: For $(x,t)\in [0,1]\times [0,T]$ we seek 
$\eta=\eta(x,t)$ and $u=u(x,t)$ satisfying the ibvp
\begin{align} \label{eq:2p15} 
& \begin{aligned} 
& \eta_t + (\beta u)_x + (\eta u)_x = 0,\\ 
& u_t + \eta_x + uu_x = 0, 
\end{aligned}\quad  
0\leq x\leq 1,\quad 0\leq t\leq T, \\ 
& \eta(x,0) = \eta^0(x),\quad u(x,0)=u^0(x),\quad 0\leq x\leq 1, \notag \\ 
& \eta(0,t)=\eta_0,\quad u(0,t)=u_0,\quad 0\leq t\leq T, \notag 
\end{align} 
where $\beta\in C^1$, $\eta^0$, $u^0$ are given functions on $[0,1]$ and 
$\eta_0$, $u_0$ constants such that $\beta(x)+\eta_0>0$, $u_0>0$, 
$u_0>\sqrt{\beta(x)+\eta_0}$, $x\in [0,1]$. 

The ibvp \eqref{eq:2p15} was studied by Huag et al., \cite{HPT}, in the more 
general case of the presence of a lateral component of the horizontal velocity 
depending on $x$ only (nonzero Coriolis parameter). In the simpler case of 
\eqref{eq:2p15}, we assume that $(\eta_0,u_0)$ is a suitable constant solution 
of \eqref{eq:2p15} and that $\eta^0(x)$, $u^0(x)$ are sufficiently smooth 
initial conditions close to $(\eta_0,u_0)$ and satisfying appropriate 
compatibility relations at $x=0$. Then, as is proved in \cite{HPT}, given 
positive constants $c_0$, $\alpha_0$, $\underline{\zeta}_0$, and 
$\overline{\zeta}_0$, there exists a $T>0$ and a sufficiently smooth solution 
$(\eta,u)$ of \eqref{eq:2p15} satisfying for $(x,t)\in [0,1]\times [0,T]$ the 
strong supercriticality properties 
\begin{subequations} 
\begin{align} 
& u^2 - (\beta+\eta) \geq c_0^2, \label{eq:2p16a} \\ 
& u \geq \alpha_0, \label{eq:2p16b} \\ 
& \underline{\zeta}_0 \leq (\beta+\eta)\leq \overline{\zeta}_0. \label{eq:2p16c}
\end{align} 
\end{subequations}

For the purposes of the error estimation to follow we will assume in addition 
that the solution of \eqref{eq:2p15} satisfies a strengthened supercriticality 
condition of the following form: There exist positive constants $a$, and $b$, 
such that for $(x,t)\in [0,1]\times [0,T]$ 
\begin{subequations}
\begin{align} 
& \beta + \eta \geq b, \label{eq:2p17a} \\ 
& u \geq 2a, \label{eq:2p17b} \\ 
& \beta+\eta \leq (u-a)(u-\tfrac{2a}{3}). \label{eq:2p17c} 
\end{align} 
\end{subequations} 

Obviously \eqref{eq:2p17a} and \eqref{eq:2p17b} imply that $u\geq 
\sqrt{\beta+\eta}$. It is not hard to see that \eqref{eq:2p17c} follows from 
\eqref{eq:2p16a}--\eqref{eq:2p16c} if e.g.\ $\alpha_0$ is taken sufficiently 
small and $c_0$ sufficiently large. We also remark here that in the error 
estimates to follow \eqref{eq:2p17c} will be needed only at $x=1$ for $t\in 
[0,T]$. 

We will approximate the solution of \eqref{eq:2p15} in a slightly transformed 
form. We let $\tilde \eta = \eta-\eta_0$, $\tilde u = u-u_0$ and rewrite 
\eqref{eq:2p15} as an ibvp for $\tilde \eta$ and $\tilde u$ with homogeneous 
boundary conditions. Dropping the tildes we obtain the system 
\begin{align} \label{eq:2p18}
& \begin{aligned} 
& \eta_t + u_0\eta_x + (\beta+\eta_0)u_x + (\eta u)_x + (u+u_0)\beta_x = 0, \\ 
& u_t + \eta_x + u_0u_x + uu_x = 0, 
\end{aligned}\quad  
0\leq x\leq 1,\quad 0\leq t\leq T, \\ 
& \eta(x,0) = \eta^0(x)-\eta_0,\quad u(x,0)=u^0(x)-u_0,\quad 0\leq x\leq 1, 
\notag \\ 
& \eta(0,t)=0,\quad u(0,t)=0,\quad 0\leq t\leq T. \notag 
\end{align} 
In terms of the new variables \eqref{eq:2p17a}--\eqref{eq:2p17c} become 
\begin{subequations} 
\begin{align} 
& \beta + \eta + \eta_0 \geq b, \label{eq:2p19a} \\ 
& u + u_0 \geq 2a, \label{eq:2p19b} \\ 
& \beta + \eta + \eta_0 \leq (u+u_0-a)(u+u_0-\tfrac{2a}{3}). \label{eq:2p19c}
\end{align} 
\end{subequations} 

In the rest of this subsection, for integer $k\geq 0$, let $\zr{C}^k = \{ v\in 
C^k[0,1]: v(0)=0\}$, and $\zr{H}^{k+1} = \{v\in H^{k+1}(0,1): v(0)=0\}$. Using 
the hypotheses of section \ref{sec:2p1} on the finite element space 
discretization we define $\zr{S}_h = \{\phi\in\zr{C}^{r-2}: 
\phi\big|_{[x_j,x_{j+1}]}\in\mathbb P_{r-1}, 1\leq j\leq N\}$ and $\opPZ$ the 
$L^2$ projection operator onto $\zr{S}_h$. Note that 
\eqref{eq:2p2}--\eqref{eq:2p4} also hold on $\zr{S}_h$ \emph{mutatis mutandis}. 

The standard Galerkin semidiscretization of \eqref{eq:2p18} is defined as 
follows: We seek $\eta_h,\ u_h,\ : [0,T]\to \zr{S}_h$ such that for $0\leq t\leq 
T$ 
\begin{align} 
\begin{multlined} 
(\eta_{ht},\phi) + (u_0\eta_{hx},\phi) + ((\beta+\eta_0)u_{hx},\phi) + ((\eta_h 
u_h)_x,\phi) + ((u_h+u_0)\beta_x,\phi) = 0,\hspace{-1em}\\ 
 \forall \phi\in\zr{S}_h, 
\end{multlined} \label{eq:2p20} \\ 
(u_{ht},\phi) + (\eta_{hx},\phi) + (u_0u_{hx},\phi) + (u_hu_{hx},\phi) = 0,\quad 
\forall \phi\in\zr{S}_h, \label{eq:2p21} 
\end{align}
with
\begin{equation} \label{eq:2p22} 
\eta_h(0) = \opPZ(\eta^0(\cdot) - \eta_0),\quad u_h(0) = \opPZ(u^0(\cdot)-u_0). 
\end{equation} 

The boundary conditions implied by the choice of $\zr{S}_h$ are no longer 
exactly transparent, but they are highly absorbing as will be seen in the 
numerical experiments of Section \ref{sec:3}. 

\begin{proposition} \label{prop:2p2} 
Let $(\eta,u)$ be the solution of \eqref{eq:2p18}, and assume that the 
hypotheses \eqref{eq:2p19a}--\eqref{eq:2p19c} hold, that $r\geq 3$, and $h$ is 
sufficiently small. Then the semidiscrete ivp \eqref{eq:2p20}--\eqref{eq:2p22} 
has a unique solution $(\eta_h,u_h)$ for $0\leq t\leq T$ satisfying 
\begin{equation} \label{eq:2p23} 
\max_{0\leq t\leq T}(\|\eta(t)-\eta_h(t)\| + \|u(t)-u_h(t)\|)\leq C h^{r-1}. 
\end{equation} 
\end{proposition}

\begin{proof} 
Let $\rho=\eta-\opPZ\eta$, $\theta=\opPZ\eta-\eta_h$, $\sigma=u-\opPZ u$, 
$\xi=\opPZ u-u_h$. After choosing a basis for $\zr{S}_h$, it is straightforward 
to see that the semidiscrete problem represents an ivp for an ode system which 
has a unique solution locally in time. While this solution exists, it follows 
from \eqref{eq:2p20}--\eqref{eq:2p22} and the pde's in \eqref{eq:2p18}, that 
\begin{align*} 
\begin{multlined}
(\theta_t,\phi) + (u_0(\rho_x+\theta_x),\phi) + 
((\beta+\eta_0)(\sigma_x+\xi_x),\phi) + ((\eta u-\eta_hu_h)_x,\phi) + \\ 
((\sigma+\xi)\beta_x,\phi) = 0,\quad \forall \phi\in \zr{S}_h, 
\end{multlined} \\ 
(\xi_t,\phi) + (\rho_x+\theta_x,\phi) + (u_0(\sigma_x+\xi_x),\phi) + 
(uu_x-u_hu_{hx},\phi) = 0,\quad \forall\phi\in \zr{S}_h 
\end{align*}
Proceeding as in the proof of Proposition 2.1 of \cite{AD2}, which is valid for 
a horizontal bottom, we obtain from the above in the case of variable bottom 
that 
\begin{align} 
(\theta_t,\phi) + (u_0\theta_x,\phi) + (\gamma_x,\phi) + ((u\theta)_x,\phi) - 
((\theta\xi)_x,\phi) &= -(R_1,\phi), \quad \forall\phi\in\zr{S}_h, 
\label{eq:2p24} \\
(\xi_t,\phi) + (\theta_x,\phi) + (u_0\xi_x,\phi) + ((u\xi)_x,\phi) - 
(\xi\xi_x,\phi) &= -(R_2,\phi),\quad \forall\phi\in\zr{S}_h, \label{eq:2p25}  
\end{align} 
where $\gamma = (\beta+\eta_0+\eta)\xi$ and  
\begin{align} 
& R_1 = u_0\rho_x + (\beta+\eta_0)\sigma_x + \sigma\beta_x +  (\eta\sigma)_x + 
(u\rho)_x- (\rho\sigma)_x - (\rho\xi)_x - (\theta\sigma)_x, \label{eq:2p26} \\ 
& R_2 = \rho_x + u_0\sigma_x + (u\sigma)_x - (\sigma\xi)_x - \sigma\sigma_x. 
\label{eq:2p27} 
\end{align} 
Putting $\phi=\theta$ in \eqref{eq:2p24}, using integration by parts, and 
suppressing the dependence on $t$ we have 
\begin{multline} \label{eq:2p28} 
\tfrac 1 2 \tfrac{\mathrm d}{\mathrm dt}\|\theta\|^2 - (\gamma,\theta_x) + 
\tfrac 1 2 (u_0 + u(1))\theta^2(1) + (\beta(1)+\eta_0+\eta(1))\xi(1)\theta(1) \\ 
- \tfrac 1 2 \xi(1)\theta^2(1) = -\tfrac 1 2 (u_x\theta,\theta) + \tfrac 1 2 
(\xi_x\theta,\theta) - (R_1,\theta) 
\end{multline} 
Take now $\phi = \opPZ \gamma = \opPZ[(\beta+\eta_0+\eta)\xi]$ in 
\eqref{eq:2p25} and get 
\begin{equation} \label{eq:2p29} 
(\xi_t,\gamma) + (\theta_x,\gamma) + (u_0\xi_x,\gamma) + ((u\xi)_x,\gamma) - 
(\xi\xi_x,\gamma) = -(R_3,\opPZ\gamma-\gamma) - (R_2,\opPZ\gamma), 
\end{equation} 
where 
\begin{equation} \label{eq:2p30} 
R_3 = \theta_x + u_0\xi_x + (u\xi)_x - \xi\xi_x. 
\end{equation} 
Integration by parts in various terms in \eqref{eq:2p29} gives 
\begin{multline} \label{eq:2p31} 
(\xi_t,\gamma) + (\theta_x,\gamma) + 
\tfrac{1}{2}(u_0+u(1))(\beta(1)+\eta_0+\eta(1))\xi^2(1) - \tfrac{1}{3}(\beta(1) 
+ \eta_0 + \eta(1))\xi^3(1) \\ 
= (R_4,\xi) - (R_3,\opPZ\gamma-\gamma) - (R_2,\opPZ\gamma), 
\end{multline} 
where 
\begin{equation} \label{eq:2p32}
R_4 = \tfrac{1}{2}u_0(\beta_x+\eta_x)\xi - \tfrac{1}{2}u_x(\beta+\eta_0+\eta)\xi 
+ \tfrac{1}{2}u(\beta_x+\eta_x)\xi - \tfrac{1}{3}(\beta_x+\eta_x)\xi^2. 
\end{equation} 
Adding now \eqref{eq:2p28} and \eqref{eq:2p31} we obtain
\begin{multline} \label{eq:2p33} 
\tfrac{1}{2}\tfrac{\mathrm d}{\mathrm dt}\left[\|\theta\|^2 + 
((\beta+\eta_0+\eta)\xi,\xi)\right] + \omega = \tfrac{1}{2}(\eta_t\xi,\xi) - 
\tfrac{1}{2}(u_x\theta,\theta) \\ 
+ \tfrac{1}{2}(\xi_x\theta,\theta) - (R_1,\theta) + (R_4,\xi) - 
(R_3,\opPZ\gamma-\gamma) - (R_2,\opPZ\gamma), 
\end{multline} 
where 
\begin{multline} \label{eq:2p34} 
\omega = \tfrac{1}{2}(u_0+u(1))\theta^2(1) + 
\tfrac{1}{2}(u_0+u(1))(\beta(1)+\eta_0+\eta(1))\xi^2(1) \\ 
+ (\beta(1)+\eta_0+\eta(1))\xi(1)\theta(1) - \tfrac{1}{2}\xi(1)\theta^2(1) - 
\tfrac{1}{3}(\beta(1)+\eta_0+\eta(1))\xi^3(1). 
\end{multline} 
In view of \eqref{eq:2p22}, by continuity we conclude that there exists a 
maximal temporal instance $t_h>0$ such that $(\eta_h,u_h)$ exist and 
$\|\xi_x\|_\infty \leq a$ for $t\leq t_h$. Suppose that $t_h<T$. Then, since 
$\|\xi\|_\infty\leq \|\xi_x\|_\infty$, it follows from \eqref{eq:2p34} that for 
$t\in[0,t_h]$ 
\begin{multline} \label{eq:2p35} 
\omega \geq \tfrac{1}{2}(u_0+u(1)-a)\theta^2(1) + 
\tfrac{1}{2}(\beta(1)+\eta_0+\eta(1))\left(u_0+u(1)-\tfrac{2a}{3}\right)\xi^2(1) 
\\ 
+ (\beta(1)+\eta_0+\eta(1))\xi(1)\theta(1) = 
\tfrac{1}{2}(\theta(1),\xi(1))^T\begin{pmatrix}\mu & \lambda \\ \lambda & 
\lambda\nu\end{pmatrix}\begin{pmatrix}\theta(1) \\ \xi(1)\end{pmatrix}, 
\end{multline}
where $\mu=u_0+u(1)-a$, $\lambda = \beta(1)+\eta_0+\eta(1)$, 
$\nu=u_0+u(1)-\tfrac{2a}{3}$. The hypotheses \eqref{eq:2p19a}--\eqref{eq:2p19b} 
give that $0<\mu<\nu$, $\lambda>0$. It is easy to see then that the matrix in 
\eqref{eq:2p35} will be positive semidefinite precisely when \eqref{eq:2p19c} 
holds. Hence, \eqref{eq:2p35} implies that $\omega \geq 0$.

We now estimate the various terms in the right-hand side of \eqref{eq:2p33} for 
$0\leq t\leq t_h$. As in the proof of Proposition 2.1 of \cite{AD2} adapted in 
the case of a variable $\beta(x)\in C^1$ and using an appropriate 
variable-$\beta$ superapproximation property to estimate 
$\|\opPZ\gamma-\gamma\|$. We finally obtain from \eqref{eq:2p33} and the fact 
that $\omega\geq 0$, that for $0\leq t\leq t_h$ it holds that 
\[ 
\tfrac{\mathrm d}{\mathrm dt}\left[\|\theta\|^2 + 
((\beta+\eta_0+\eta)\xi,\xi)\right] \leq Ch^{r-1}(\|\theta\|+\|\xi\|) + 
C(\|\theta\|^2 + \|\xi\|^2), 
\] 
where $C$ is a constant independent of $h$ and $t_h$. By \eqref{eq:2p19a} the 
norm $((\beta+\eta_0+\eta)\,\cdot,\cdot)^{1/2}$ is equivalent to that of $L^2$ 
uniformly for $t\in [0,T]$. Hence, Gronwall's inequality and the fact that 
$\theta(0)=\xi(0)=0$ yield for a constant $C=C(T)$ 
\begin{equation} \label{eq:2p36} 
\|\theta\|+\|\xi\|\leq Ch^{r-1}\quad \text{for}\quad 0\leq t\leq t_h. 
\end{equation} 
We conclude from the inverse properties that $\|\xi_x\|_\infty\leq Ch^{r-5/2}$ for 
$0\leq t\leq t_h$, and, since $r\geq 3$, if $h$ is taken sufficiently small, 
$t_h$ is not maximal. Hence we may take $t_h=T$ and \eqref{eq:2p23} follows from 
\eqref{eq:2p36}. 
\end{proof}

\subsection{Semidiscretization in the case of absorbing (characteristic) boundary 
conditions in the subcritical case} 
\label{sec:2p3} 

We finally consider the shallow water equations with variable bottom in the 
presence of transparent (characteristic) boundary conditions in the 
\emph{subcritical case}. In this case, instead of the variable $\eta$, we will 
use the \emph{total height} of the water, $H = \beta + \eta$. For 
$(x,t)\in[0,1]\times[0,T]$ we seek $H=H(x,t)$ and $u=u(x,t)$ satisfying the ibvp 
\begin{align} \label{eq:2p37} 
& \begin{aligned} 
& H_t + (Hu)_x = 0,\\ 
& u_t + H_x + uu_x = \beta_x, 
\end{aligned}\quad  
0\leq x\leq 1,\quad 0\leq t\leq T, \\ 
& H(x,0) = H^0(x),\quad u(x,0)=u^0(x),\quad 0\leq x\leq 1, \notag \\ 
& u(0,t) + 2\sqrt{H(0,t)} = u_0 + 2\sqrt{H_0},\quad 0\leq t\leq T, \notag \\ 
& u(1,t) - 2\sqrt{H(1,t)} = u_0 - 2\sqrt{H_0},\quad 0\leq t\leq T, \notag 
\end{align} 
where $H^0$, $u^0$ are given functions on $[0,1]$ and $H_0$, $u_0$ constants 
such that $H_0>0$ and $u_0^2 < H_0$. 

Implicit in the formulation of the boundary conditions in \eqref{eq:2p37} is 
that outside the spatial domain $[0,1]$ $u$ and $H$ are equal to constants 
$u_0$, $H_0$, respectively. The ibvp \eqref{eq:2p37} in a slightly different but 
equivalent form was studied by Petcu and Temam, \cite{PT2}, under the hypotheses 
that for some constant $c_0>0$ it holds that $u_0^2-H_0\leq -c_0^2$ and that the 
initial conditions $H^0(x)$ and $u^0(x)$ are sufficiently smooth and satisfy the 
condition $(u^0(x))^2-H^0(x) \leq -c_0^2$ and suitable compatibility relations 
at $x=0$ and $x=1$. Under these assumptions one may infer from the theory of 
\cite{PT2} that there exists a $T>0$ such that a sufficiently smooth solution 
$(H,u)$ of \eqref{eq:2p37} exists for $(x,t)\in[0,1]\times[0,T]$ with the 
properties that $H$ is positive and the strong supercriticality condition 
\begin{equation} \label{eq:2p38} 
u^2 - H \leq -c_0^2, 
\end{equation} 
holds for $(x,t)\in[0,1]\times[0,T]$. Here we will assume that the solution 
satisfies a stronger subcriticality solution; specifically that for some 
constant $c_0>0$ it holds that 
\begin{subequations} \label{eq:2p39} 
\begin{equation} \label{eq:2p39a} 
u_0 + \sqrt{H_0} \geq c_0,\quad u_0-\sqrt{H_0} \leq -c_0, 
\end{equation} 
and for $(x,t)\in[0,1]\times[0,T]$ that 
\begin{equation} \label{eq:2p39b} 
u + \sqrt{H} \geq c_0,\quad u-\sqrt{H} \leq -c_0.
\end{equation} 
\end{subequations} 
In this section we will approximate the solution of \eqref{eq:2p37} after 
transforming the system in diagonal form. We write the system of pde's in 
\eqref{eq:2p37} as 
\begin{equation} \label{eq:2p40} 
\begin{pmatrix} H_t \\ u_t \end{pmatrix} + A\begin{pmatrix} H_x \\ u_x 
\end{pmatrix} = \begin{pmatrix} 0 \\ \beta_x \end{pmatrix} 
\end{equation} 
where $A = \begin{pmatrix} u & H \\ 1 & u \end{pmatrix}$. The matrix $A$ has 
eigenvalues $\lambda_1=u+\sqrt{H}$, $\lambda_2=u-\sqrt{H}$, (note that 
\eqref{eq:2p39b} implies that $\lambda_1\geq c_0$ and $\lambda_2\leq -c_0$ in 
$[0,1]\times[0,T]$), with associated eigenvectors $X_1 = 
\big(\sqrt{H},1\big)^\mathrm T$, $X_2 = \big(-\sqrt{H},1\big)^\mathrm T$. If $S$ 
is the matrix with columns $X_1,\ X_2$ it follows from \eqref{eq:2p40} that 
\begin{equation} \label{eq:2p41} 
S^{-1} \begin{pmatrix} H_t \\ u_t \end{pmatrix} + \begin{pmatrix} \lambda_1 & 0 
\\ 0 & \lambda_2\end{pmatrix} S^{-1}\begin{pmatrix} H_x \\ u_x\end{pmatrix} = 
S^{-1}\begin{pmatrix} 0 \\ \beta_x \end{pmatrix}. 
\end{equation} 
If we try to define now functions $v,\ w$ on $[0,1]\times[0,T]$ by the equations 
$S^{-1}\begin{pmatrix} H_t \\ u_t \end{pmatrix} = \begin{pmatrix} v_t \\ w_t 
\end{pmatrix}$, $S^{-1}\begin{pmatrix} H_x \\ u_x \end{pmatrix} = 
\begin{pmatrix} v_x \\ w_x \end{pmatrix}$, we see that these equations are 
consistent and their solutions are given by $v = \frac 1 2 u + \sqrt{H} + c_v$, 
$w = \frac 1 2 u - \sqrt{H} + c_w$, for arbitrary constants $c_v,\ c_w$. 
Choosing the constants $c_v,\ c_w$ so that $v(0,t)=0$, $w(1,t)=0$, and using the 
boundary conditions in \eqref{eq:2p37} we get 
\begin{equation} \label{eq:2p42} 
v = \tfrac 1 2 \big[u-u_0 + 2\big(\sqrt{H}-\delta_0\big)\big],\quad w = \tfrac 1 
2 \big[u-u_0 - 2\big(\sqrt{H}-\delta_0\big)\big] 
\end{equation} 
where $\delta_0 = \sqrt{H_0}$. The original variables $H,\ u$ are given in terms 
of $v$ and $w$ by the formulas 
\begin{equation} \label{eq:2p43} 
H = (\tfrac 1 2 (v-w) + \delta_0)^2,\quad u = v + w + u_0 
\end{equation} 
Since 
\begin{equation} \label{eq:2p44} 
\lambda_1 = u+\sqrt{H} = u_0 + \delta_0 + \frac{3v+w}{2},\quad \lambda_2 = 
u-\sqrt{H} = u_0-\delta_0+\frac{v+3w}{2}  
\end{equation} 
we see that the ibvp \eqref{eq:2p37} becomes 
\begin{equation} \label{eq:2p45} 
\begin{gathered} 
 \begin{multlined}[t][.8\textwidth] \begin{pmatrix} v_t\\ w_t \end{pmatrix} + 
\begin{pmatrix} u_0+\delta_0+\tfrac{3v+w}{2} & 0 \\ 0 & 
u_0-\delta_0+\frac{v+3w}{2} \end{pmatrix} \begin{pmatrix} v_x \\ w_x 
\end{pmatrix} = \tfrac{1}{2}\beta_x\begin{pmatrix} 1 \\ 1 \end{pmatrix}, \\  
0\leq x\leq 1,\ \ 0\leq t\leq T. \end{multlined} \\ 
\begin{aligned} 
& v(x,0)=v^0(x),\quad w(x,0)=w^0(x), \quad 0\leq x\leq 1, 
\hspace{.1\linewidth}\\ 
& v(0,t)=0,\quad w(1,t)=0,\quad 0\leq t\leq T, 
\end{aligned}
\end{gathered}
\end{equation} 
where $v^0(x) = \tfrac1 2 [u^0(x)-u_0+2(\sqrt{H^0(x)}-\delta_0)]$, $w^0(x) = 
\tfrac1 2 [u^0(x)-u_0-2(\sqrt{H^0(x)}-\delta_0)]$. Under our hypotheses 
\eqref{eq:2p45} has a unique solution in $[0,1]\times[0,T]$ which will be 
assumed to be smooth enough for the purposes of the error estimation that 
follows. \\
Given a quasiuniform partition of $[0,1]$ as in section \ref{sec:2p1}, in 
addition to the spaces defined there, let for integer $k\geq 0$\ \ $\zr{\mathscr 
C}^k = \{ f\in C^k[0,1]: f(1)=0\}$, $\zr{\mathscr H}^{k+1} = \{ f\in 
H^{k+1}(0,1), f(1)=0 \}$, and, for integer $r\geq 2$, $\zr{\mathscr S}^0_h =\{ 
\phi\in \zr{\mathscr C}^{r-2}: \phi\big|_{[x_j,x_{j+1}]}\in \mathbb P_{r-1}, 
1\leq j\leq N\}$. Note that the analogs of the approximation and inverse 
properties \eqref{eq:2p2}, \eqref{eq:2p4} hold for $\zr{\mathscr S}_h$ as well, 
and that the estimates in \eqref{eq:2p3} are also valid for the $L^2$ projection 
$\operatorname{P^1}$ onto $\zr{\mathscr S}_h$, {\em mutatis mutandis}.
The (standard) Galerkin semidiscretization of \eqref{eq:2p45} is then defined as 
follows: Seek $v_h:[0,T]\to \zr{S}_h,\ w_h:[0,T]\to \zr{\mathscr S}_h$, such 
that for $t\in [0,T]$ 
\begin{gather} 
(v_{ht},\phi) + ((u_0+\delta_0)v_{hx},\phi) + \tfrac 3 2 (v_hv_{hx},\phi) + 
\tfrac 1 2 (w_hv_{hx},\phi) = \tfrac 1 2 (\beta_x,\phi),\ \forall \phi\in \zr 
S_h, \label{eq:2p46} \\ 
(w_{ht},\chi) + ((u_0{-}\delta_0)w_{hx},\chi) + \tfrac 3 2 (w_hw_{hx},\chi) + 
\tfrac 1 2 (v_hw_{hx},\chi) = \tfrac 1 2 (\beta_x,\chi),\; \forall \chi\in 
\zr{\mathscr S}_h,  \label{eq:2p47} 
\end{gather}
with 
\begin{equation} \label{eq:2p48} 
v_h(0) = \opPZ(v^0),\quad w_h(0) = \operatorname{P^1}(w^0). 
\end{equation}
The boundary conditions induced by the finite element spaces and the discrete 
variational formulation \eqref{eq:2p46}--\eqref{eq:2p48} are no longer exactly 
transparent; they are highly absorbent nevertheless as will be checked in 
numerical experiments in Section \ref{sec:3}. \\ 
The main result of this section is 
\begin{proposition} \label{prop:2p3} 
Let $(v,w)$ be the solution of \eqref{eq:2p45} and assume that the hypotheses 
\eqref{eq:2p39a}--\eqref{eq:2p39b} hold, that $r\geq 3$, and that $h$ is 
sufficiently small. Then the semidiscrete ivp \eqref{eq:2p46}--\eqref{eq:2p48} 
has a unique solution $(v_h, w_h)$ for $0\leq t\leq T$ that satisfies 
\begin{equation} \label{eq:2p49} 
\max_{0\leq t\leq T} \left(\|v-v_h\|+\|w-w_h\|\right) \leq C h^{r-1}. 
\end{equation} 
If $(H,u)$ is the solution of \eqref{eq:2p39} and we define 
\begin{equation} \label{eq:2p50} 
H_h = [\tfrac 1 2 (v_h-w_h)+\delta_0]^2,\quad u_h = v_h + w_h + u_0, 
\end{equation} 
then 
\[ 
\max_{0\leq t\leq T} \left(\|H-H_h\|+\|u-u_h\|\right) \leq C h^{r-1}. 
\]
\end{proposition} 

\begin{proof} 
With the notation that we have introduced incorporating the variable $\beta(x)$, 
it can be seen that the proof is entirely analogous to that of Proposition 3.1 
of \cite{AD2} ---\emph{mutatis mutandis}--- and will consequently be omitted. 
(Note that the source terms involving $\beta_x$ in the right-hand sides of 
\eqref{eq:2p45}, \eqref{eq:2p46}, \eqref{eq:2p47} will cancel in the variational 
error equations.) 

\end{proof}

\section{Numerical experiments}
\label{sec:3} 

In this section we present results of numerical experiments that we performed 
solving numerically the shallow water equations using standard Galerkin finite 
element space discretizations like the ones analyzed in the previous section. 
The semidiscrete schemes were discretized in the temporal variable by the 
`classical', explicit, 4-stage, 4\textsuperscript{th}-order Runge-Kutta scheme 
(RK4), unless otherwise indicated. The resulting fully discrete scheme is stable 
and fourth-order accurate in time provided a Courant-number stability condition 
of the form $\frac k h \leq \alpha$ is imposed; here $k$ denotes the (uniform) 
time step. In the case of a horizontal bottom the convergence of this scheme for 
the ibvp \eqref{eq:2p1} was analyzed in \cite{ADK} and used in numerical 
experiments for the absorbing b.c.\ ibvp's \eqref{eq:2p15} and \eqref{eq:2p37} 
in \cite{AD2}.  

In section \ref{sec:3p1} below we use this fully discrete scheme to study 
computationally various issues related to the discretization of the ibvp's with 
absorbing (characteristic) b.c.'s considered in sections \ref{sec:2p2} and 
\ref{sec:2p3}. In section \ref{sec:3p2} we write the shallow water equations in 
the form of a balance law and study various issues of the numerical solution of 
this model with Galerkin-finite element methods, including questions of `good 
balance' of the schemes. Since the numerical method simulates only smooth 
solutions, initial conditions and bottom topographies were taken to be of small 
amplitude to ensure that no discontinuities developed within the time frame of 
the experiments.

\subsection{Absorbing (characteristic) boundary conditions} 
\label{sec:3p1} 

In the numerical experiments of this section we use the Galerkin finite element 
method with continuous, piecewise linear functions for the space discretization 
of the numerical solution of the ibvp's with absorbing (characteristic) boundary 
conditions considered in sections \ref{sec:2p2} and \ref{sec:2p3}. The 
theoretical error estimates in Propositions \ref{prop:2p2} and \ref{prop:2p3} 
require at least piecewise quadratic elements, i.e.\ $r\geq 3$, and predict 
$L^2$-error bounds of $\mathcal O(h^{r-1})$ for quasiuniform meshes. The results 
of numerical experiments shown in the sequel suggest that the method works with 
piecewise linear functions (i.e.\ $r=2$) as well, and in this case the $L^2$ 
errors for a uniform mesh are of $\mathcal O(h^2)$. 

In the \emph{supercritical} case, in order to find the numerical convergence 
rates of the scheme \eqref{eq:2p20}--\eqref{eq:2p21} we consider an ibvp with 
$\eta_0=1$, $u_0=3$ and a bottom function and exact solution given for 
$x\in(0,1)$ by 
\begin{equation} \label{eq:3p1}
\begin{aligned}  
& \beta(x) = 1 - 0.04\exp(-100(x-0.5)^2), \\ 
& \eta(x,t) = x\exp(-xt) + \eta_0, \quad u(x,t) = (1-x-\cos(\pi x))\exp(2t) + 
u_0. 
\end{aligned} 
\end{equation} 
(The initial conditions and an appropriatee right-hand side were computed from 
these formulas.) The problem was solved with a uniform mesh with $h=1/N$ and 
$k=h/10$. The $L^2$ errors and rates of convergence at $T=1$ are shown in Table 
\ref{tab:1}. 
\begin{table}[htbp] \tt 
\centering
{\small 
\begin{tabular}[b]{|c|cc|cc|}
\hline
\multicolumn{1}{|c|}{~} &
\multicolumn{2}{|c|}{$\eta$} &
\multicolumn{2}{|c|}{$u$} \\ \hline
$N$ & $L_2$ {\rm error} & {\rm rate} & $L_2$ {\rm error} & {\rm rate} \\ \hline
 40 & 1.3202e-03 & - & 6.1375e-03 & - \\ \hline
 80 & 3.2932e-04 & 2.003 & 1.5334e-03 & 2.001 \\ \hline
160 & 8.2245e-05 & 2.001 & 3.8335e-04 & 2.000 \\ \hline
320 & 2.0550e-05 & 2.001 & 9.5918e-05 & 1.999 \\ \hline
640 & 5.1361e-06 & 2.000 & 2.4070e-05 & 1.995 \\ \hline
\end{tabular}
}
\caption{$L^2$ errors and rates of convergence at $T=1$, $r=2$, supercritical 
case, \eqref{eq:3p1}, $h=1/N$, $k/h=1/10$ \label{tab:1}} 
\end{table} 

In the case of a \emph{subcritical} flow we consider an ibvp with $\eta_0=1$, 
$u_0=1$, and bottom function and exact solution given for $x\in(0,1)$ by 
\begin{align} \label{eq:3p2}
& \begin{aligned} 
& b(x) = 1 - 0.04\exp(-100(x-0.5)^2), \\ 
& \eta(x,t) = (x+1)\exp(-xt), \\ 
& u(x,t) = (2x + \cos(\pi x) - 1)\exp(t) + x A(t) + (1-x)B(t), \\ 
\end{aligned} \\ 
& \text{where} \notag \\ 
& A(t) = 2\sqrt{1+\eta(1,t)}+u_0-2\sqrt{1+\eta_0},\notag \\ 
& B(t) = -2\sqrt{1+\eta(0,t)}+u_0+2\sqrt{1+\eta_0}. \notag 
\end{align} 
(The initial conditions and an appropriate right-hand side were computed by 
these formulas). The problem was solved by the scheme 
\eqref{eq:2p46}--\eqref{eq:2p48}, \eqref{eq:2p50}, with $h=1/N$ and $k=h/10$. 
The $L^2$ errors and rates of convergence for the variables $\eta$ and $u$ at 
$T=1$ are shown in Table \ref{tab:2}. 
\begin{table}[htbp] \tt 
\centering
{\small 
\begin{tabular}[b]{|c|cc|cc|}
\hline
\multicolumn{1}{|c|}{~} &
\multicolumn{2}{|c|}{$\eta$} &
\multicolumn{2}{|c|}{$u$} \\ \hline
$N$ & $L_2$ {\rm error} & {\rm rate} & $L_2$ {\rm error} & {\rm rate} \\ \hline
 40 & 7.8451e-03 & - & 4.7238e-03 & - \\ \hline
 80 & 1.9602e-03 & 2.001 & 1.2154e-03 & 1.959 \\ \hline
160 & 4.8955e-04 & 2.001 & 3.0717e-04 & 1.984 \\ \hline
320 & 1.2229e-04 & 2.001 & 7.7169e-05 & 1.993 \\ \hline
640 & 3.0560e-05 & 2.001 & 1.9349e-05 & 1.996 \\ \hline
\end{tabular}
}
\caption{$L^2$ errors and rates of convergence at $T=1$, $r=2$, subcritical 
case, \eqref{eq:3p2}, $h=1/N$, $k/h=1/10$ \label{tab:2}} 
\end{table}

It is clear that Tables \ref{tab:1} and \ref{tab:2} suggest that the $L^2$ 
convergence rates are optimal in the case of piecewise linear elements on a 
uniform mesh. 

In order to check further the accuracy of the numerical schemes we consider in 
the \emph{supercritical} case a problem with a variable bottom having a single 
hump, and constant initial conditions on $(0,1)$ given by 
\begin{equation} \label{eq:3p3} 
\begin{aligned} 
& \beta(x) = 1-0.4\exp(-100(x-0.5)^2), \\ 
& \eta^0(x) = \eta_0 = 1, \quad u^0(x) = u_0 = 3, 
\end{aligned} 
\end{equation} 
that we integrate numerically using $h=1/400$, $k=h/3$. In Figure \ref{fig:1} we 
show some profiles of the temporal evolution of the numerical solution up to 
$t=0.5$. The data given by \eqref{eq:3p3} and the  boundary conditions generate 
a wave moving to the right and sensing the effect of the variable bottom which 
is centered at $x=0.5$. There are no spurious oscillations reflected from the 
boundary $x=1$ as the wave exits. By $t=0.5$ the solution has attained a 
\emph{steady state} shown in \eqref{fig:1d}. 

\begin{figure}[htbp]
\centering 

\begin{subfigure}[b]{0.49\textwidth}
\includegraphics[width=\textwidth]{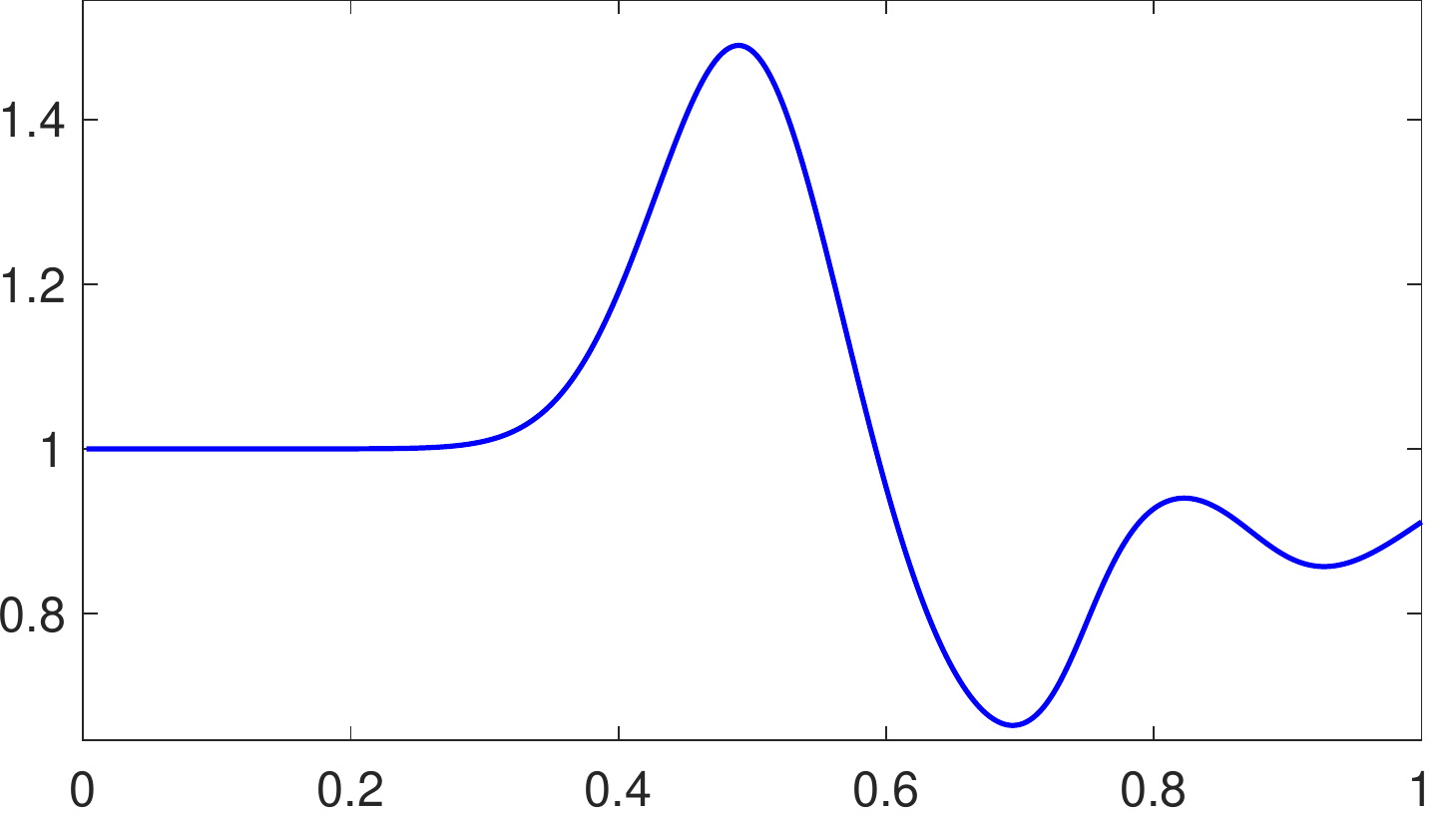}
\subcaption{$\eta$ for $t=0.1$}
\end{subfigure}\hspace{.01\textwidth}
\begin{subfigure}[b]{0.49\textwidth}
\includegraphics[width=\textwidth]{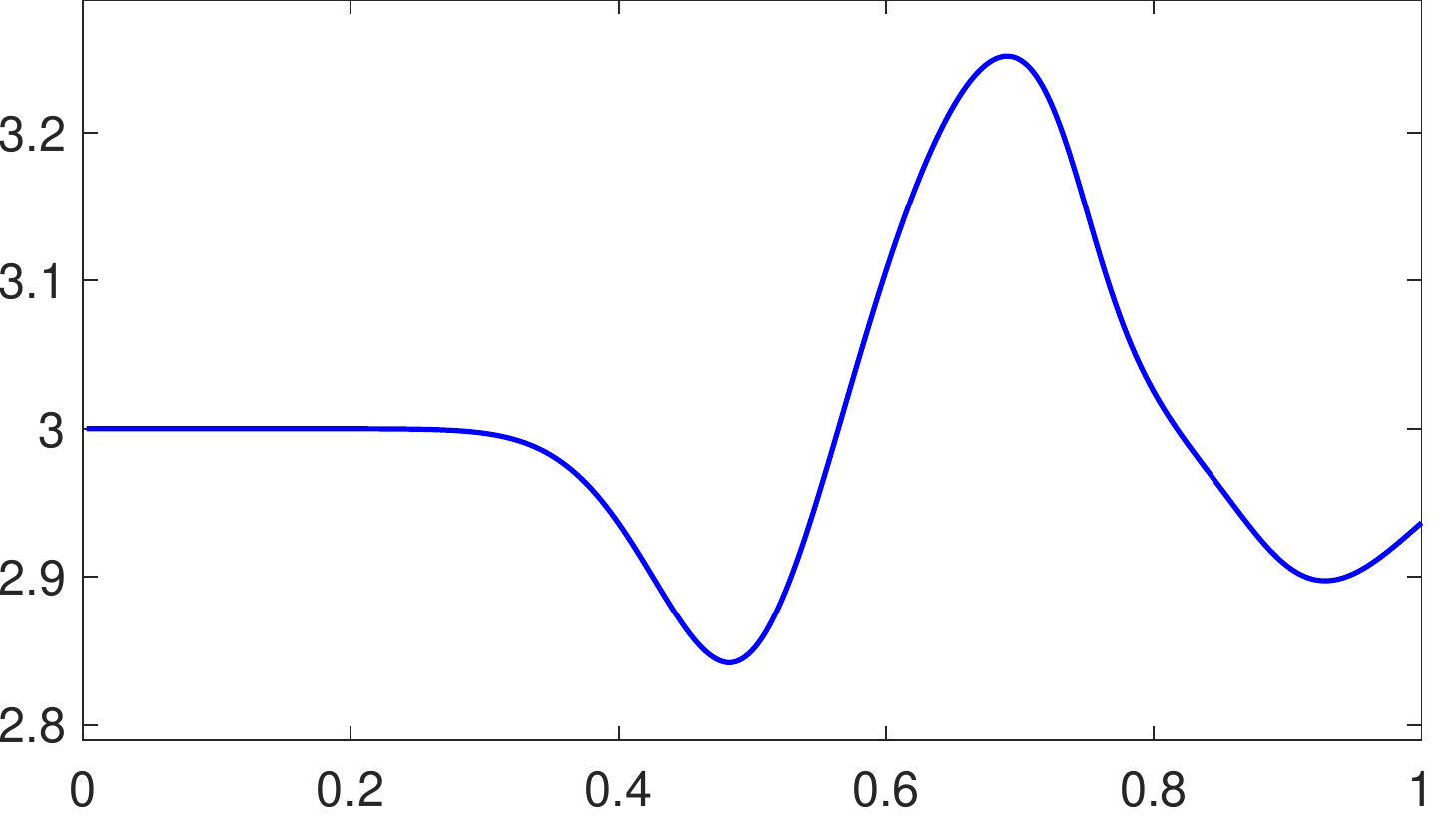}
\subcaption*{$u$ for $t=0.1$}
\end{subfigure}\\[1ex]
\begin{subfigure}[b]{0.49\textwidth}
\includegraphics[width=\textwidth]{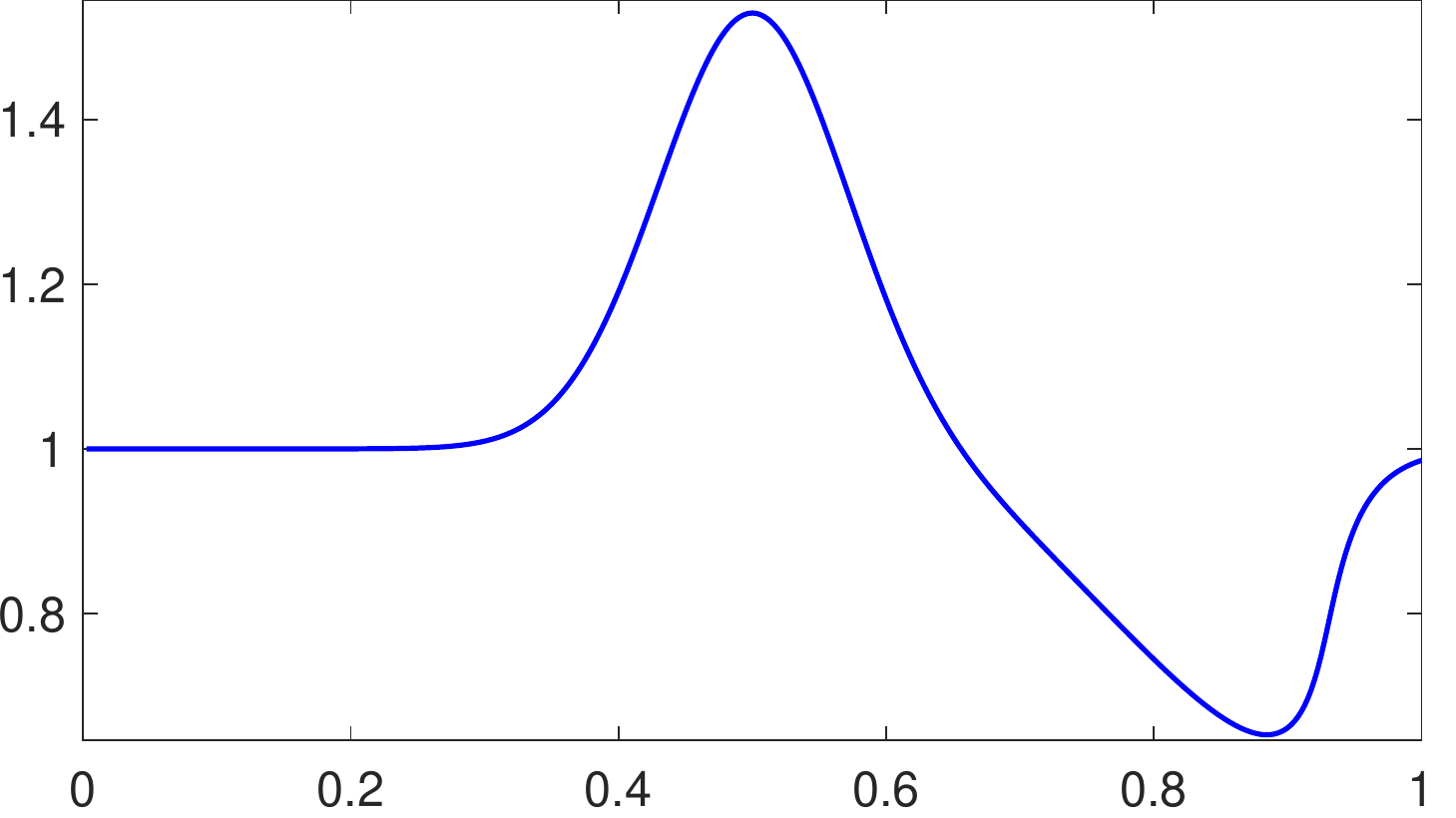}
\subcaption{$\eta$ for $t=0.2$}
\end{subfigure}\hspace{.01\textwidth}
\begin{subfigure}[b]{0.49\textwidth}
\includegraphics[width=\textwidth]{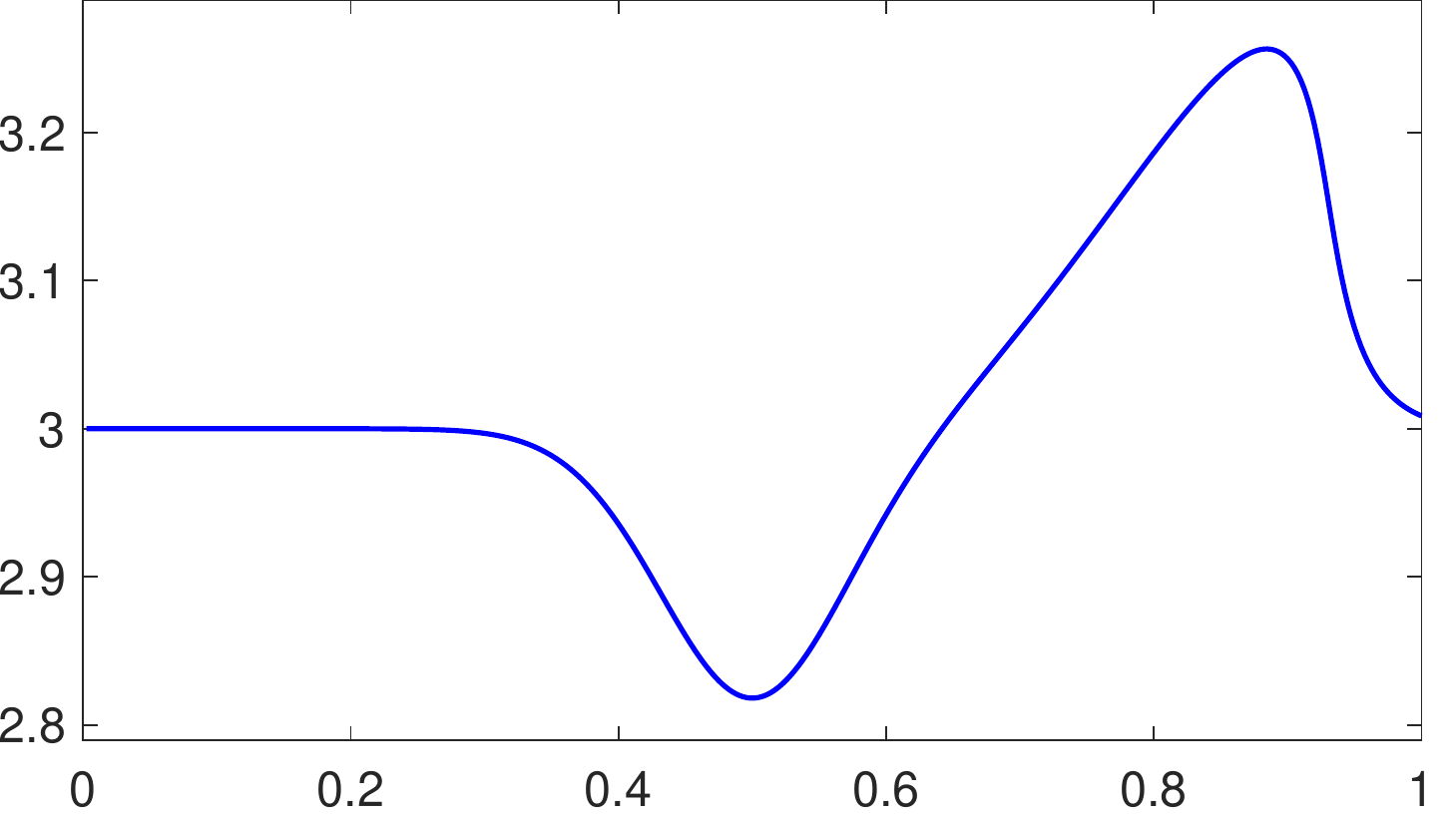}
\subcaption*{$u$ for $t=0.2$}
\end{subfigure}\\[1ex]
\begin{subfigure}[b]{0.49\textwidth}
\includegraphics[width=\textwidth]{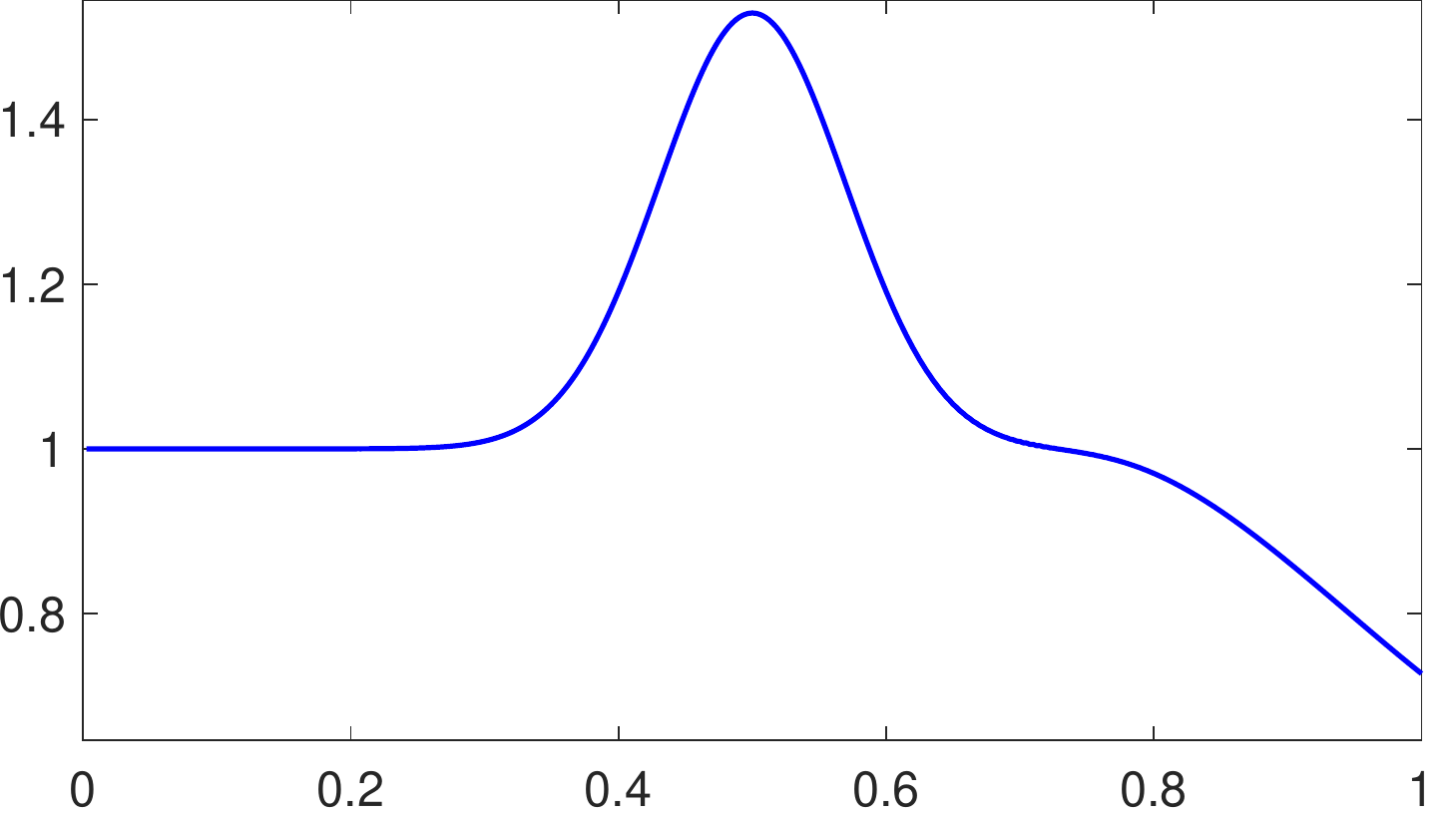}
\subcaption{$\eta$ for $t=0.3$}
\end{subfigure}\hspace{.01\textwidth}
\begin{subfigure}[b]{0.49\textwidth}
\includegraphics[width=\textwidth]{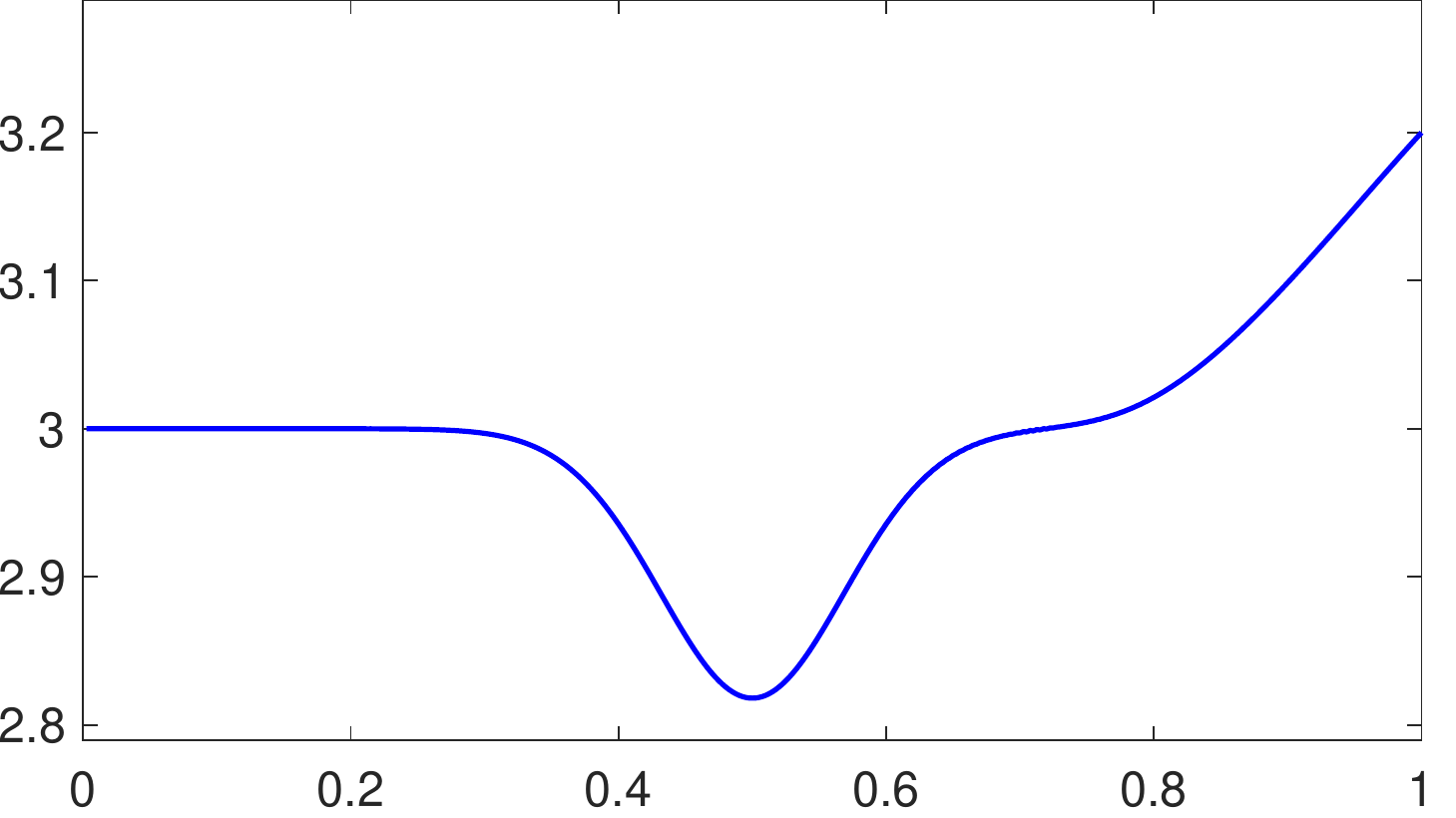}
\subcaption*{$u$ for $t=0.3$}
\end{subfigure}\\[1ex]
\begin{subfigure}[b]{0.49\textwidth}
\includegraphics[width=\textwidth]{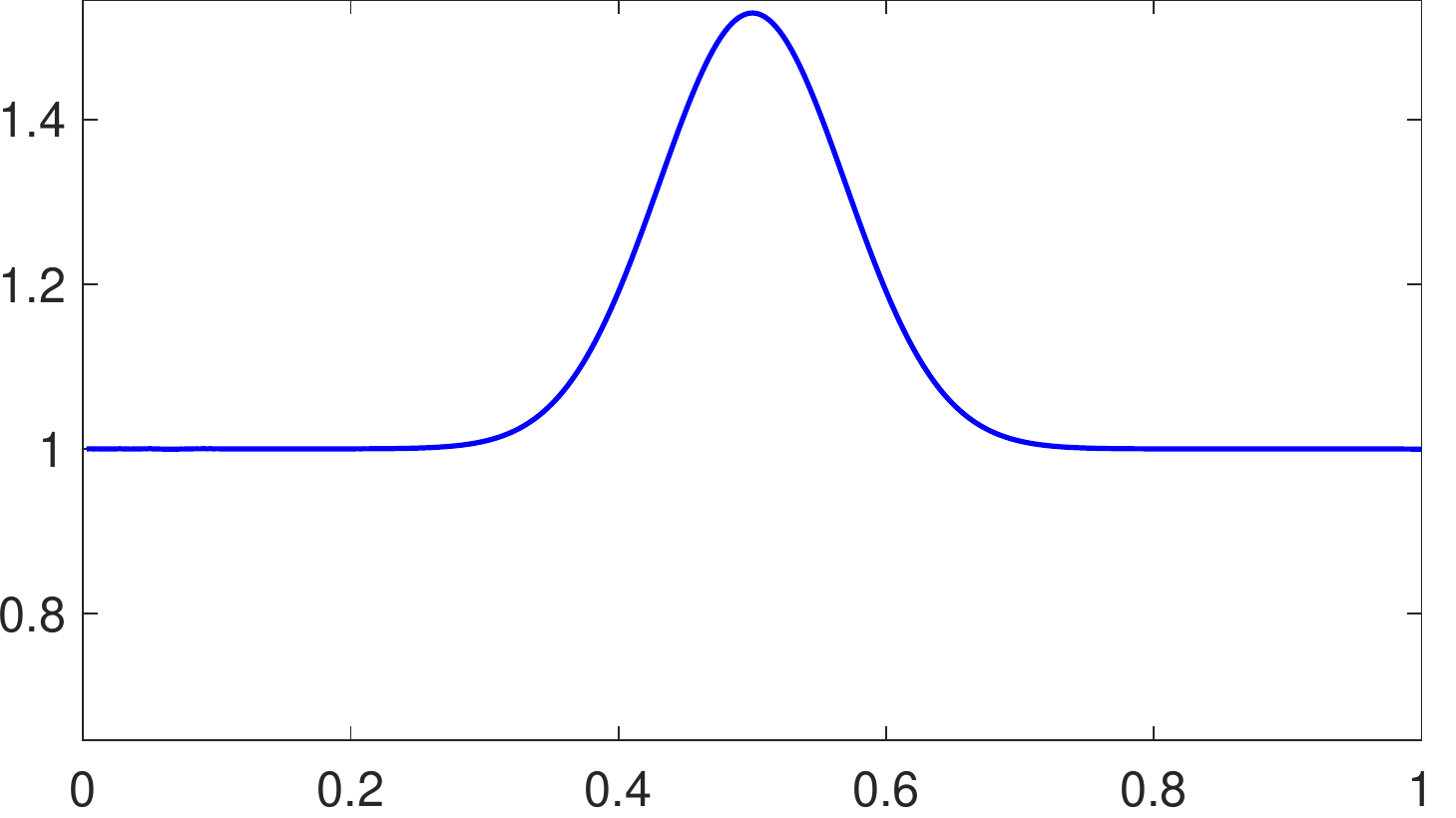}
\subcaption{$\eta$ for $t=0.5$ \label{fig:1d}}
\end{subfigure}\hspace{.01\textwidth}
\begin{subfigure}[b]{0.49\textwidth}
\includegraphics[width=\textwidth]{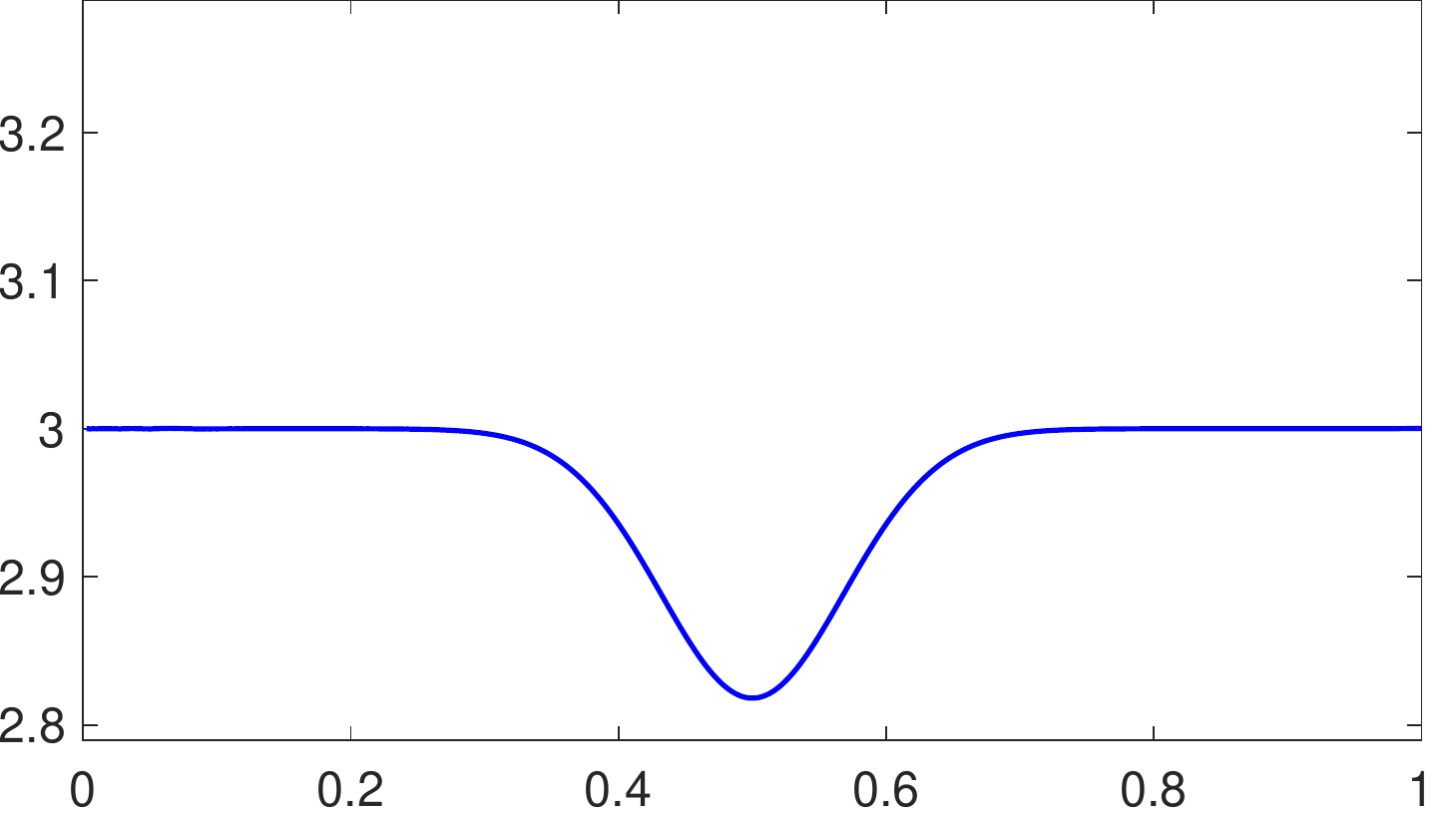}
\subcaption*{$u$ for $t=0.5$}
\end{subfigure}\\[1ex]
\caption{Evolution with data \eqref{eq:3p3}, supercritical case, $r=2$, $h=1/400$, 
$k=h/3$ \label{fig:1}} \vspace{-1ex} 
\end{figure}

The steady state of such flows is straightforward to determine analytically. Its 
profile $\eta=\eta(x)$, $u=u(x)$ satisfies the equations 
\begin{equation} \label{eq:3p4} 
\begin{aligned} 
& \left((\beta+\eta)u\right)_x = 0, \\ 
& \left(\eta + \tfrac 1 2 u^2\right)_x = 0, 
\end{aligned} 
\end{equation} 
from which using the boundary conditions at $x=0$, we see that $u$ is given in 
terms of $\eta$ by 
\begin{subequations} \label{eq:3p5}  
\begin{equation} \label{eq:3p5a} 
u = \frac{u_0\left(\eta_0+\beta(0)\right)}{\eta+\beta}, 
\end{equation} 
where $\eta$ is the physically acceptable solution of the cubic equation 
\begin{equation} \label{eq:3p5b} 
(\eta+\beta)^2\left(\eta - \eta_0 - \tfrac 1 2 u_0^2\right) + \tfrac 1 2 u_0^2 
\left(\eta_0+\beta(0)\right)^2 = 0. 
\end{equation} 
\end{subequations}
(For the analysis of the solutions of the steady-state problem, cf.\ \cite{HK}). 
We checked the ability of the code to preserve steady-state solutions by taking 
the profile computed analytically from \eqref{eq:3p5} for this problem as 
initial condition and integrating up to $t=0.6$. The difference between the final 
profile and the $L^2$ projection of the analytical initial condition was of 
$\mathcal O(10^{-9})$ in $L^2$ for both components when $h=1/400$, $k=h/10$. 

In Figure \ref{fig:2} we show instances of the temporal evolution up to the 
attainment of steady state (in \eqref{fig:2d}) of the supercritical flow generated 
with $h=1/400$, $k=h/3$, by $\eta_0=1$, $u_0=3$ and bottom topography and 
initial conditions given on $[0,1]$ by 
\begin{equation} \label{eq:3p6} 
\begin{aligned} 
& \beta(x) = 1-0.04\exp(-1000(x-0.75)^2), \\ 
& \eta^0(x) = 0.05\exp(-400(x-0.25)^2) + \eta_0, \\ 
& u^0(x) = 0.1\exp(-400(x-0.25)^2)+u_0. 
\end{aligned} 
\end{equation} 
The variable initial profile gives rise to a wavetrain that moves to the right, 
interacts with the bottom and exits without spurious oscillations leaving behind 
the steady state that depends only on $\eta_0$, $u_0$ and $\beta$. 

\begin{figure}[htbp]
\centering 

\begin{subfigure}[b]{0.49\textwidth}
\includegraphics[width=\textwidth]{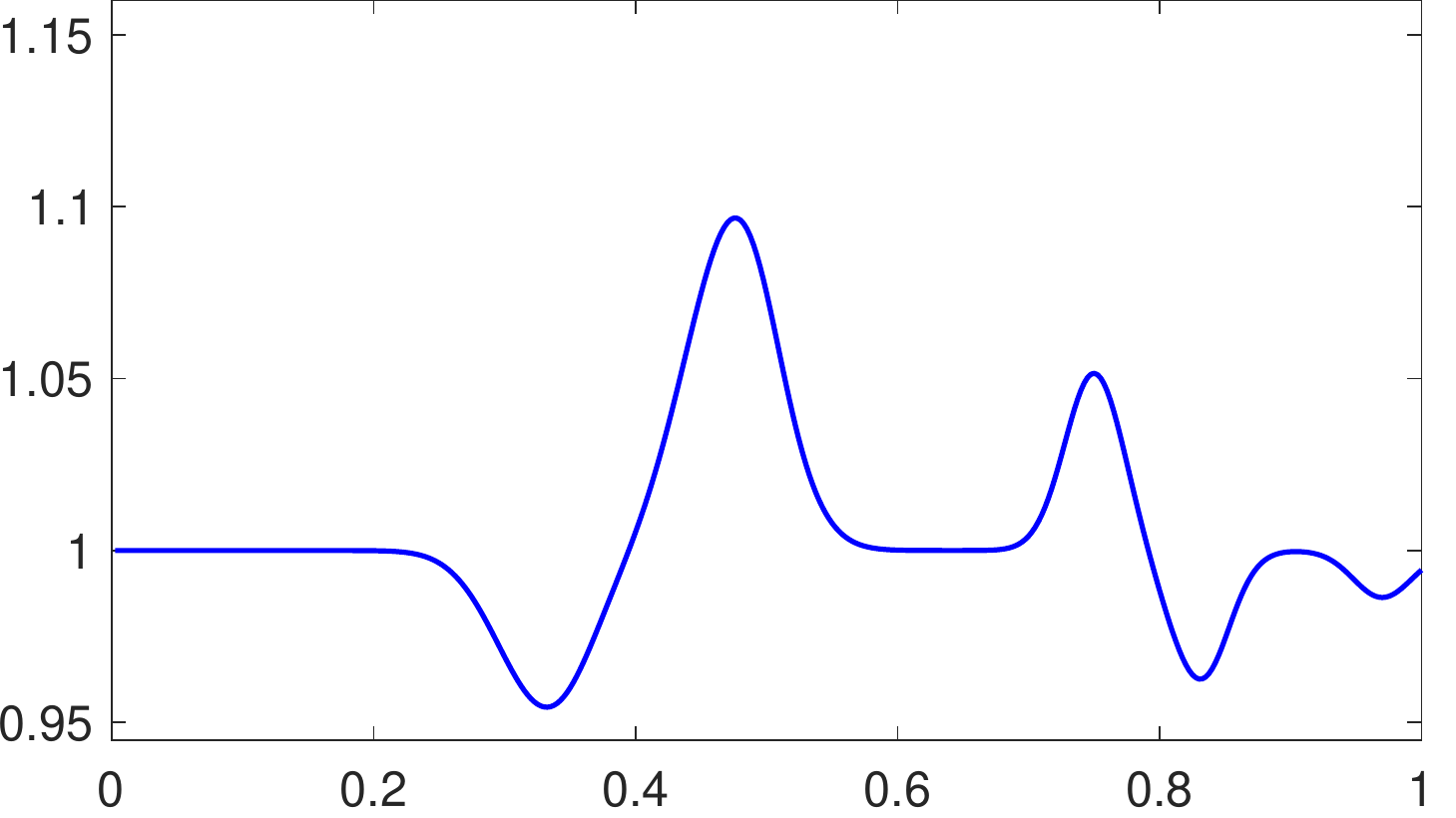}
\subcaption{$\eta$ for $t=0.5000$}
\end{subfigure}\hspace{.01\textwidth}
\begin{subfigure}[b]{0.49\textwidth}
\includegraphics[width=\textwidth]{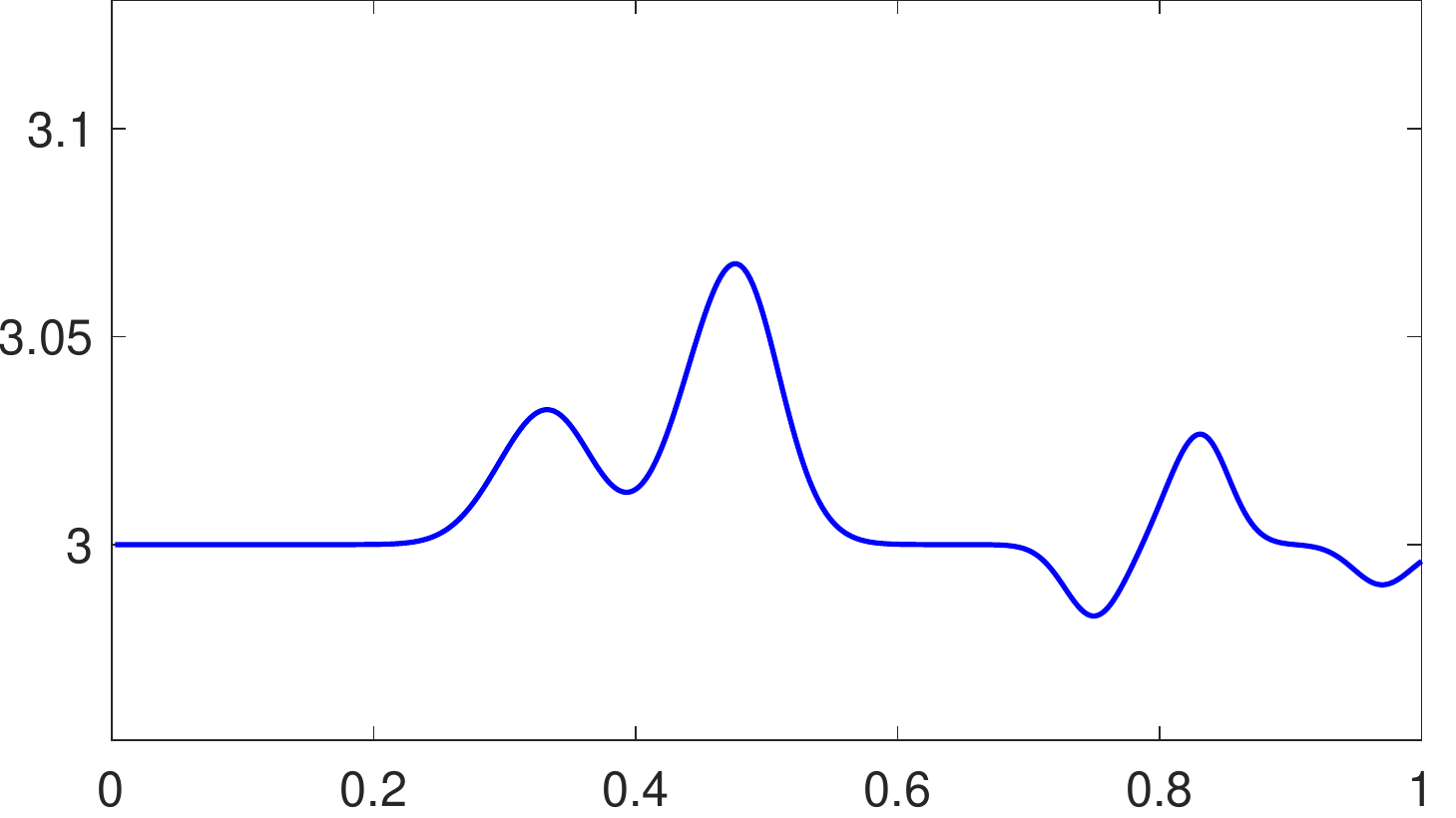}
\subcaption*{$u$ for $t=0.5000$}
\end{subfigure}\\[1ex]
\begin{subfigure}[b]{0.49\textwidth}
\includegraphics[width=\textwidth]{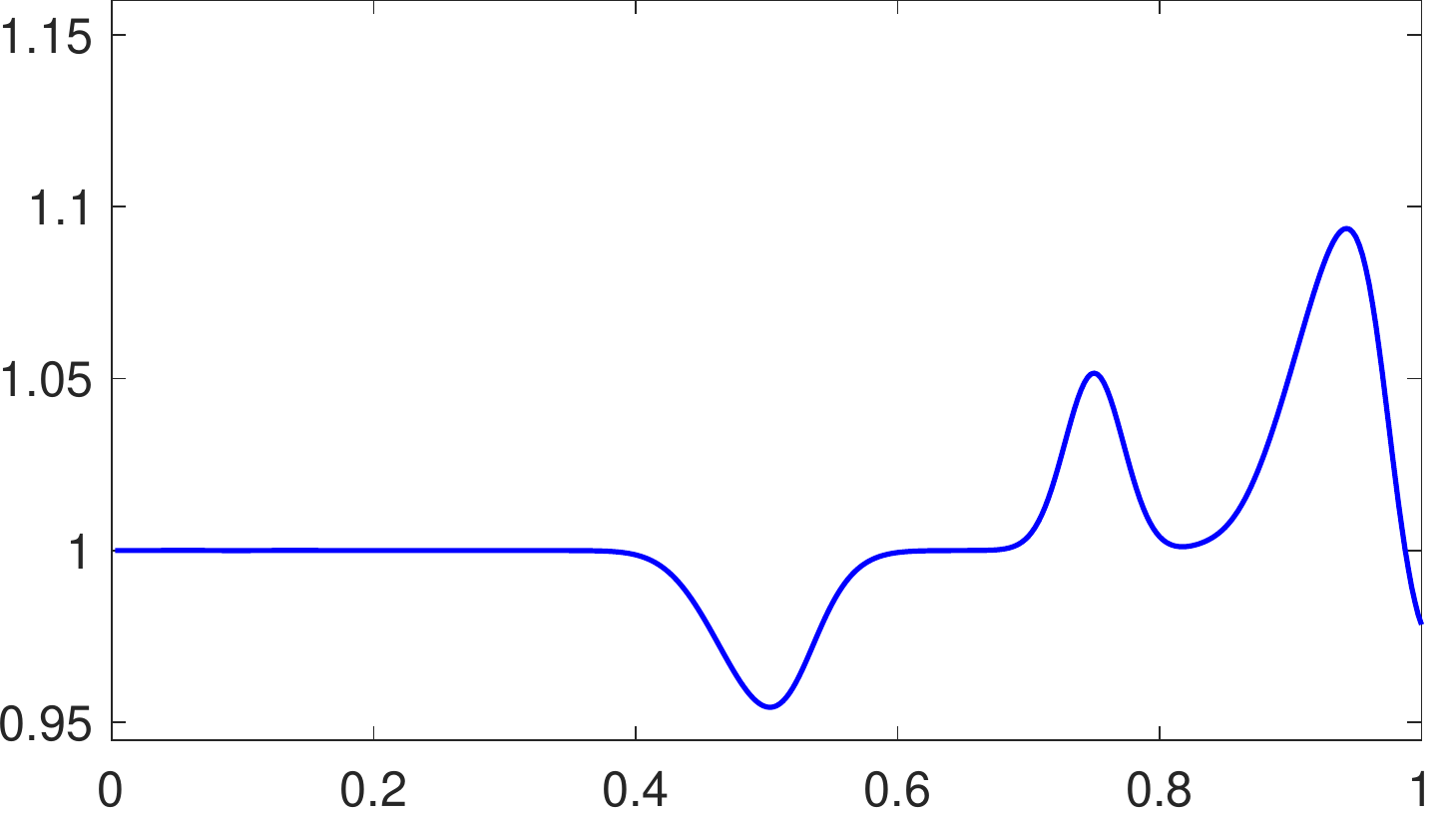}
\subcaption{$\eta$ for $t=0.1542$}
\end{subfigure}\hspace{.01\textwidth}
\begin{subfigure}[b]{0.49\textwidth}
\includegraphics[width=\textwidth]{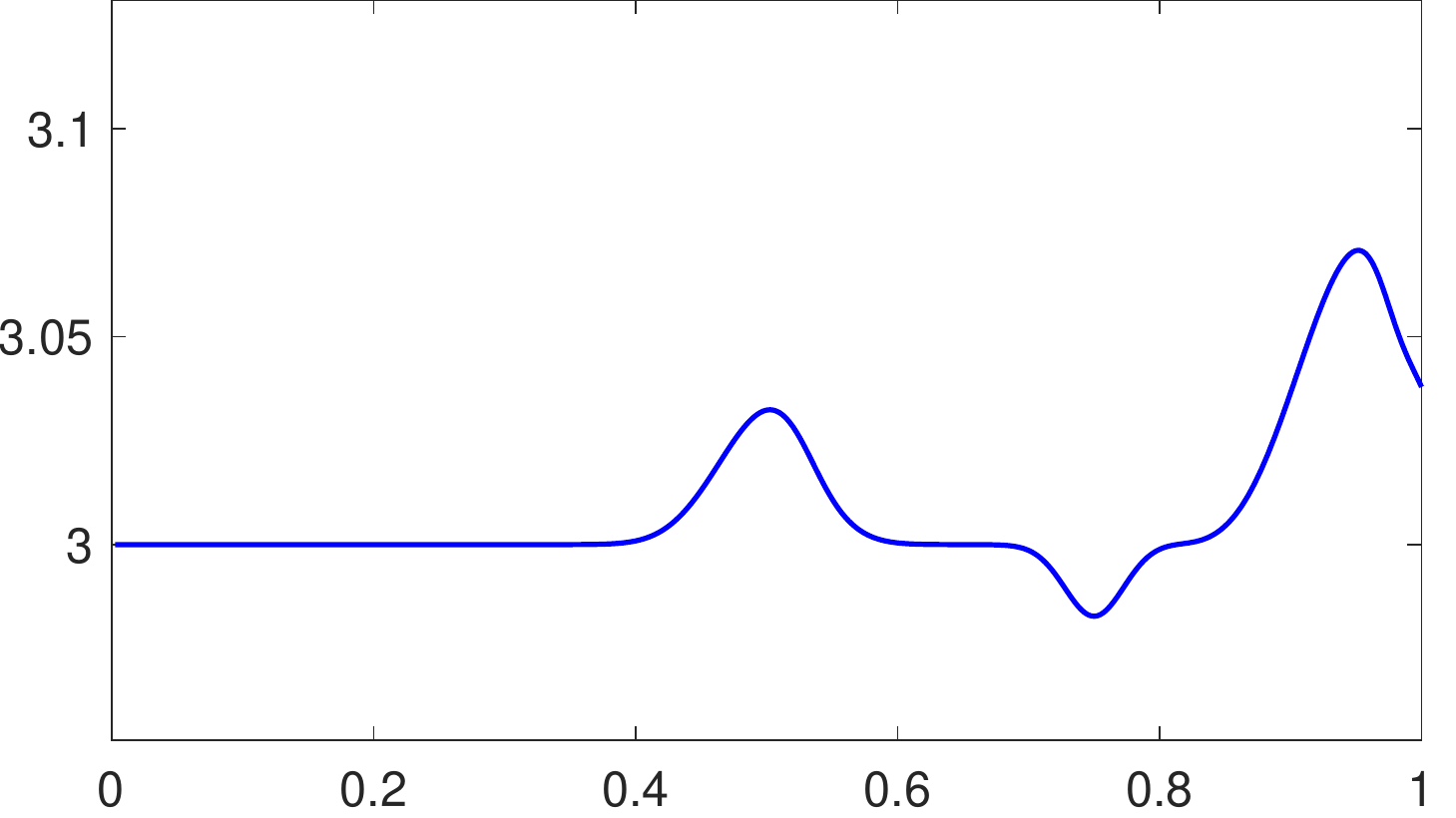}
\subcaption*{$u$ for $t=0.1542$}
\end{subfigure}\\[1ex]
\begin{subfigure}[b]{0.49\textwidth}
\includegraphics[width=\textwidth]{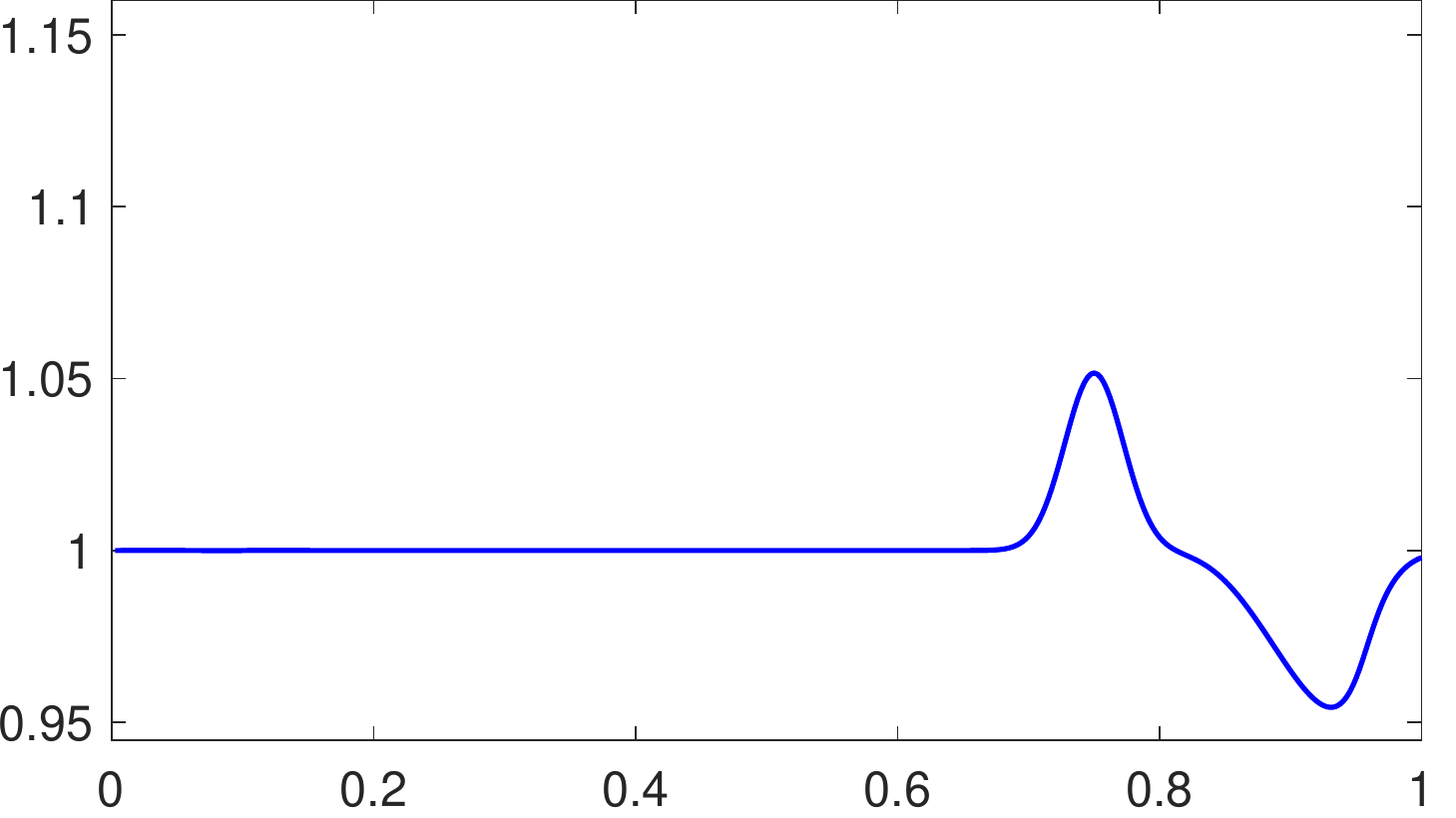}
\subcaption{$\eta$ for $t=0.4167$}
\end{subfigure}\hspace{.01\textwidth}
\begin{subfigure}[b]{0.49\textwidth}
\includegraphics[width=\textwidth]{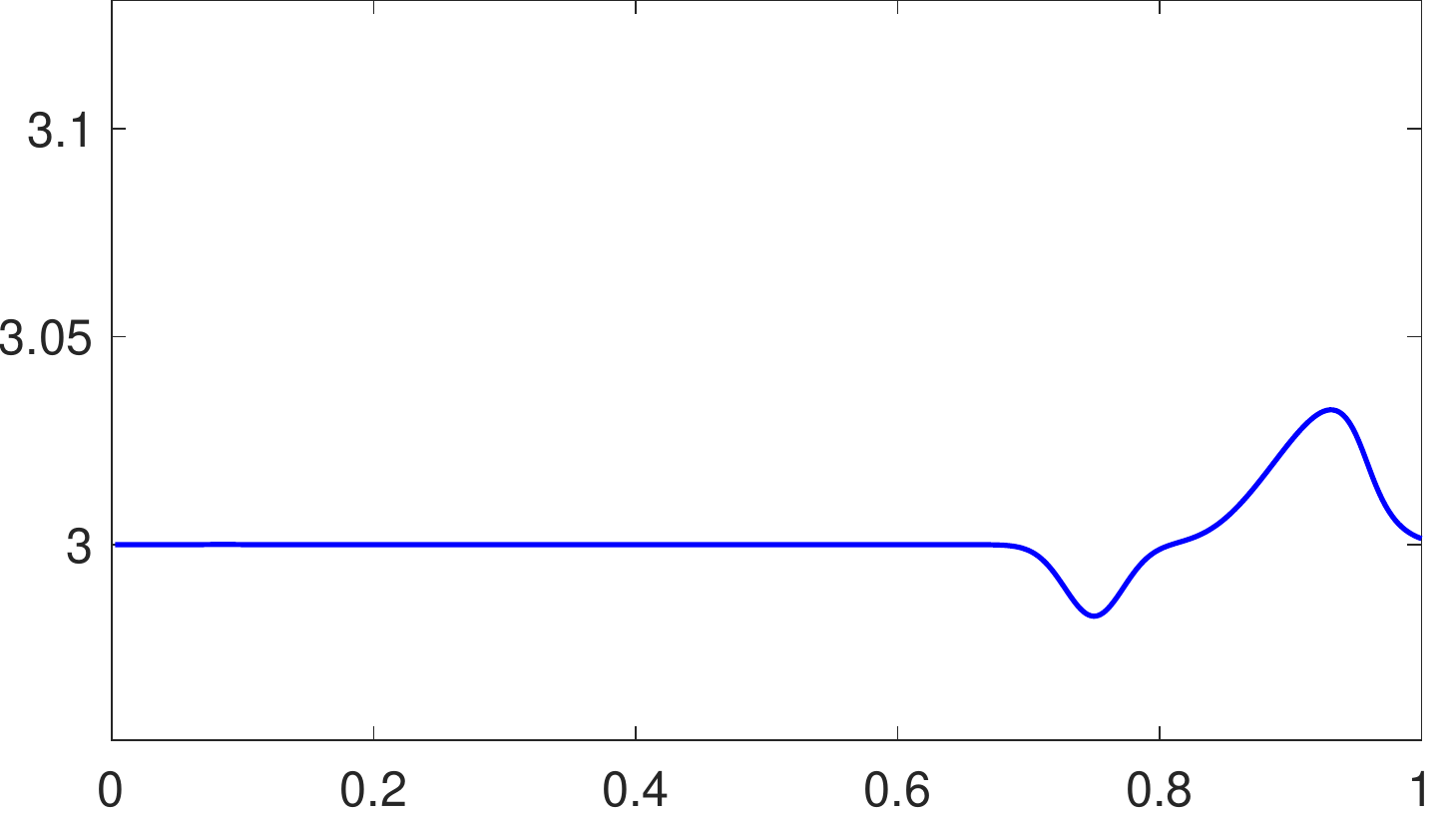}
\subcaption*{$u$ for $t=0.4167$}
\end{subfigure}\\[1ex]
\begin{subfigure}[b]{0.49\textwidth}
\includegraphics[width=\textwidth]{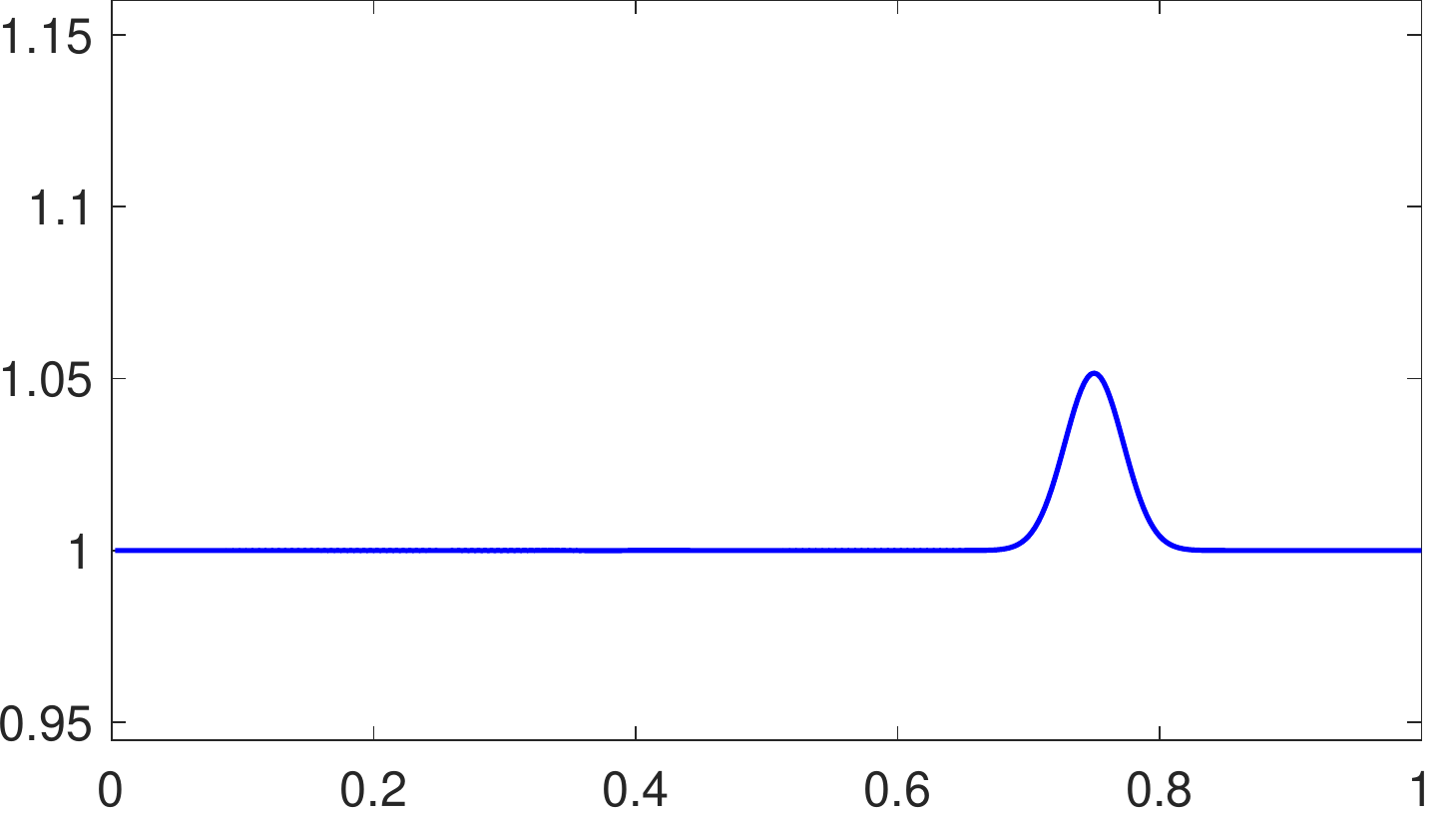}
\subcaption{$\eta$ for $t=0.6000$ \label{fig:2d}}
\end{subfigure}\hspace{.01\textwidth}
\begin{subfigure}[b]{0.49\textwidth}
\includegraphics[width=\textwidth]{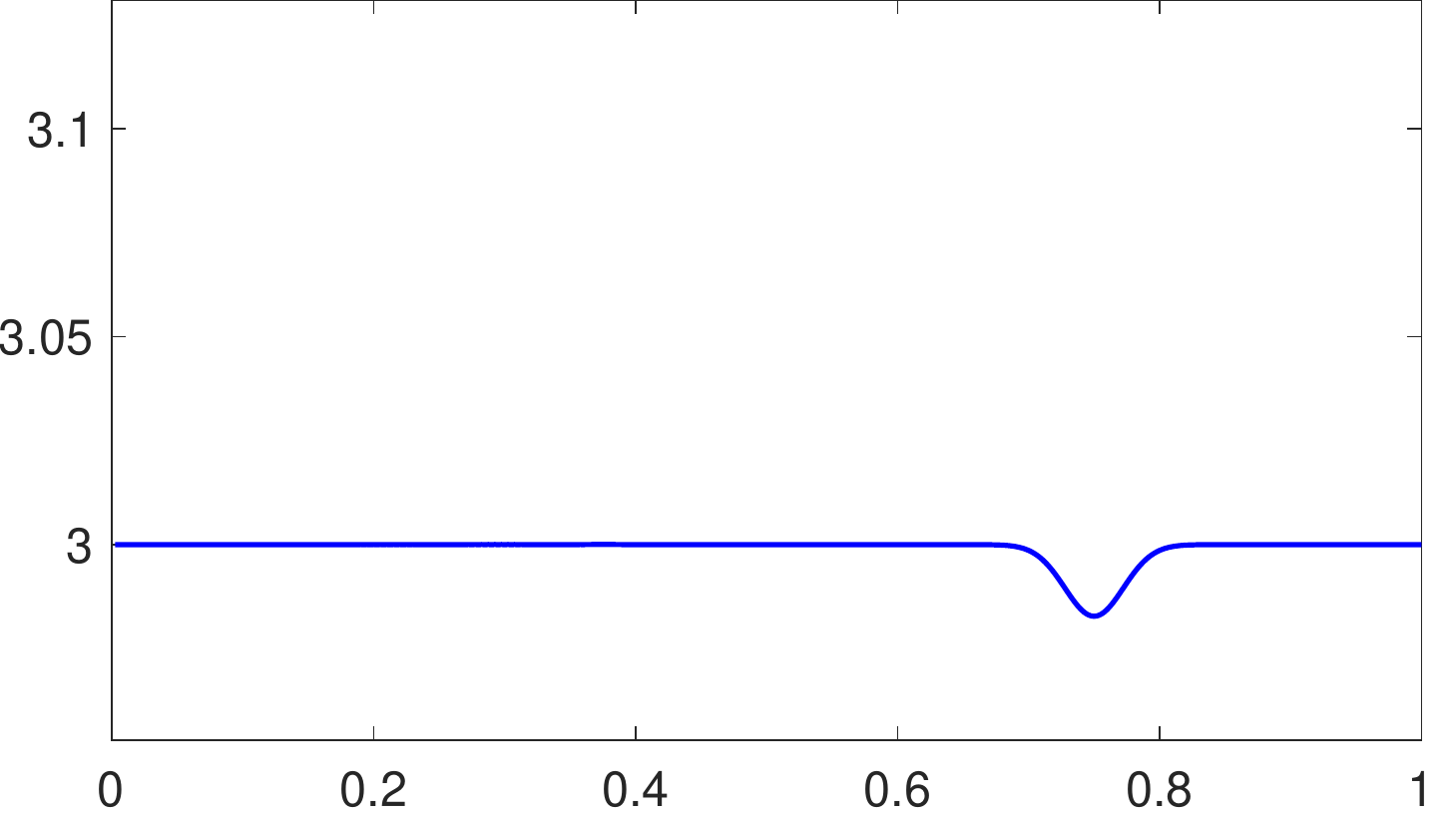}
\subcaption*{$u$ for $t=0.6000$}
\end{subfigure}\\[1ex]
\caption{Evolution with data \eqref{eq:3p6}, supercritical case, $r=2$, $h=1/400$, 
$k=h/3$ \label{fig:2}} \vspace{-1ex} 
\end{figure}

We now present some analogous results in the \emph{subcritical} case . We used 
the fully discrete scheme with spatial discretization given by 
\eqref{eq:2p46}--\eqref{eq:2p48}, \eqref{eq:2p50}; the variables depicted in the 
figures are the approximations of $\eta$ and $u$. The spatial discretization was 
effected on $[0,1]$ with piecewise linear functions on a uniform mesh with 
$h=1/2000$; the time-stepping procedure was RK4 as usual with $k=h/10$. In the 
first example we took $\eta_0=1$, $u_0=1$ and 
\begin{equation} \label{eq:3p7} 
\begin{aligned} 
& \beta(x) = 1-0.04\exp(-100(x-0.5)^2), \\ 
& \eta^0(x) = \eta_0,\quad u^0(x) = u_0. 
\end{aligned} 
\end{equation} 
The ensuing evolution of the solution is shown in Figure \ref{fig:3}. The 
generated wave interacts with the bottom and forms pulses that exit without 
artificial oscillations at both ends of the boundary; the steady-state solution 
may be found analytically as before. When used as initial condition, its $L^2$ 
projection differed from the numerical solution at $t=2$ by an $L^2$-error of 
$\mathcal O(10^{-8})$ for this example. 

\begin{figure}[htbp]
\centering 

\begin{subfigure}[b]{0.49\textwidth}
\includegraphics[width=\textwidth]{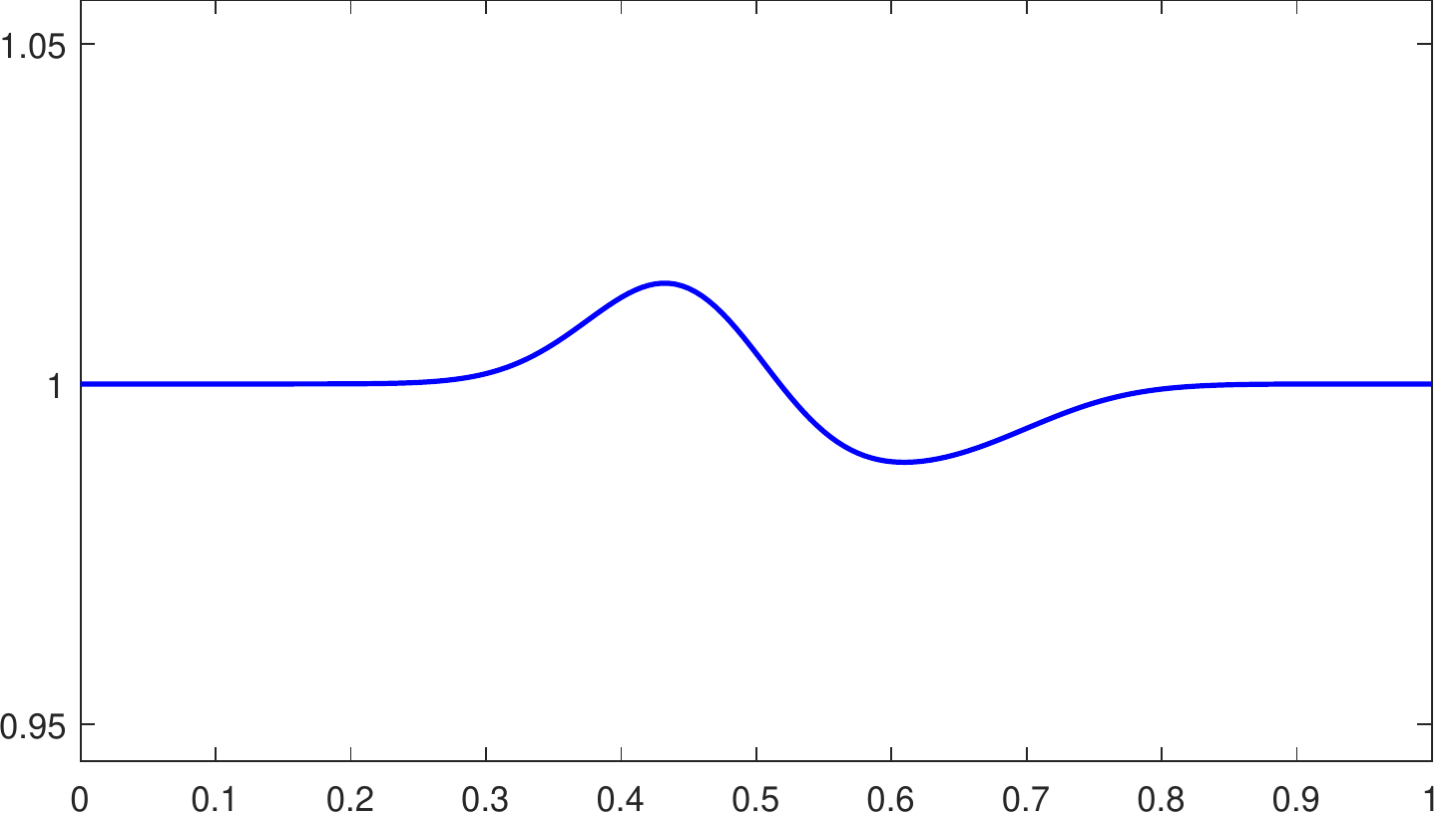}
\subcaption{$\eta$ for $t=0.06$}
\end{subfigure}\hspace{.01\textwidth}
\begin{subfigure}[b]{0.49\textwidth}
\includegraphics[width=\textwidth]{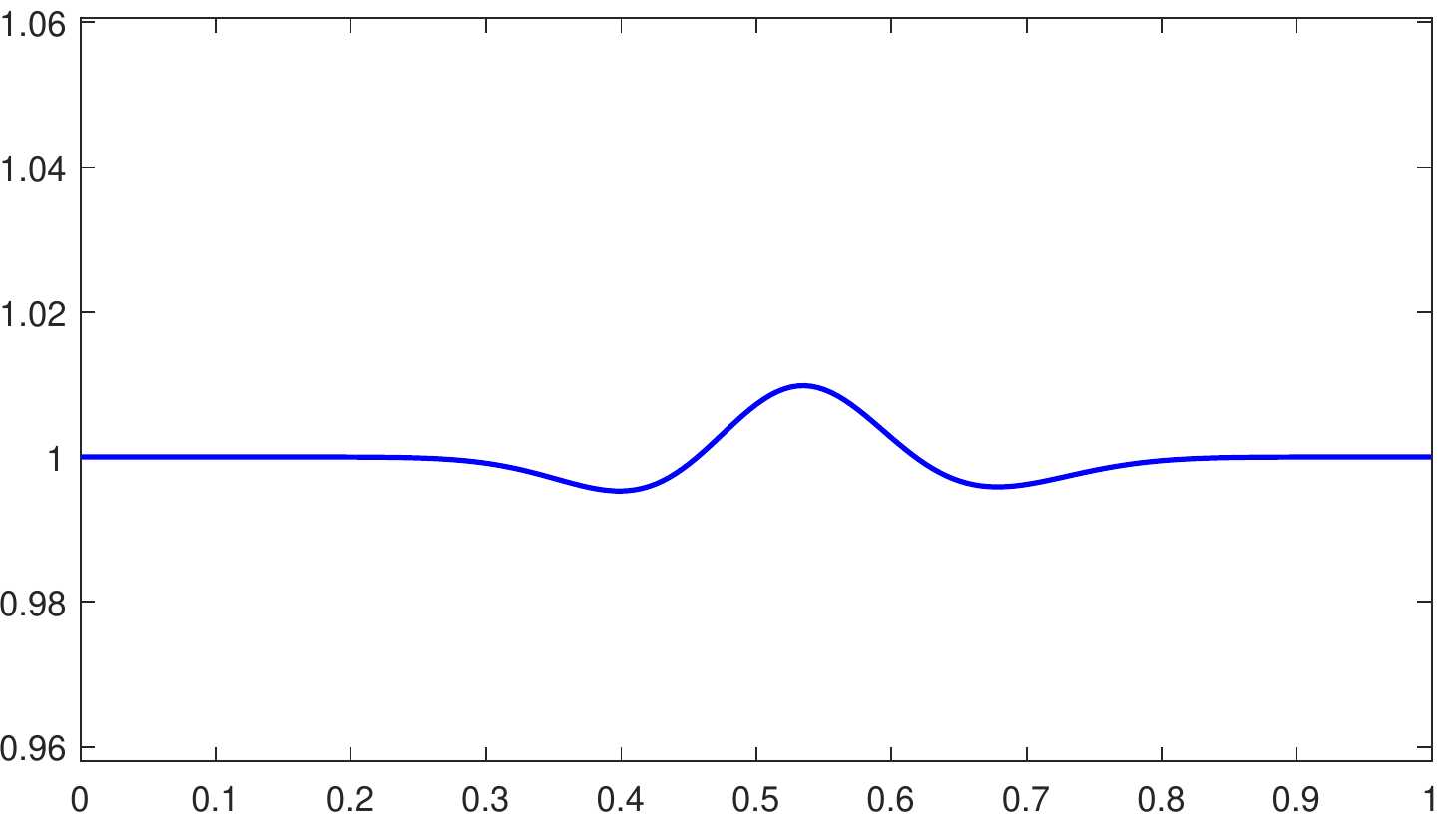}
\subcaption*{$u$ for $t=0.06$}
\end{subfigure}\\[1ex]
\begin{subfigure}[b]{0.49\textwidth}
\includegraphics[width=\textwidth]{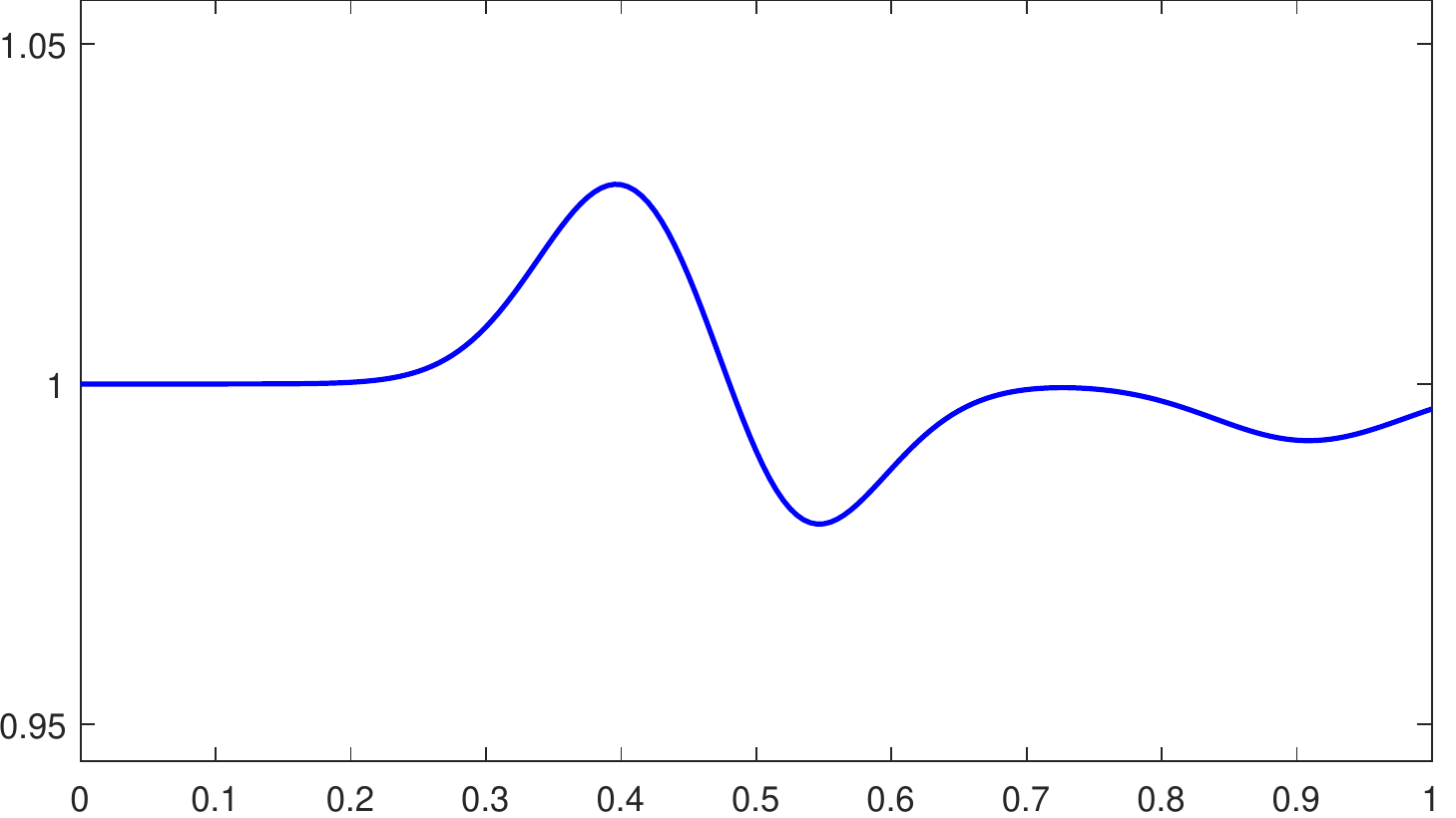}
\subcaption{$\eta$ for $t=0.17$}
\end{subfigure}\hspace{.01\textwidth}
\begin{subfigure}[b]{0.49\textwidth}
\includegraphics[width=\textwidth]{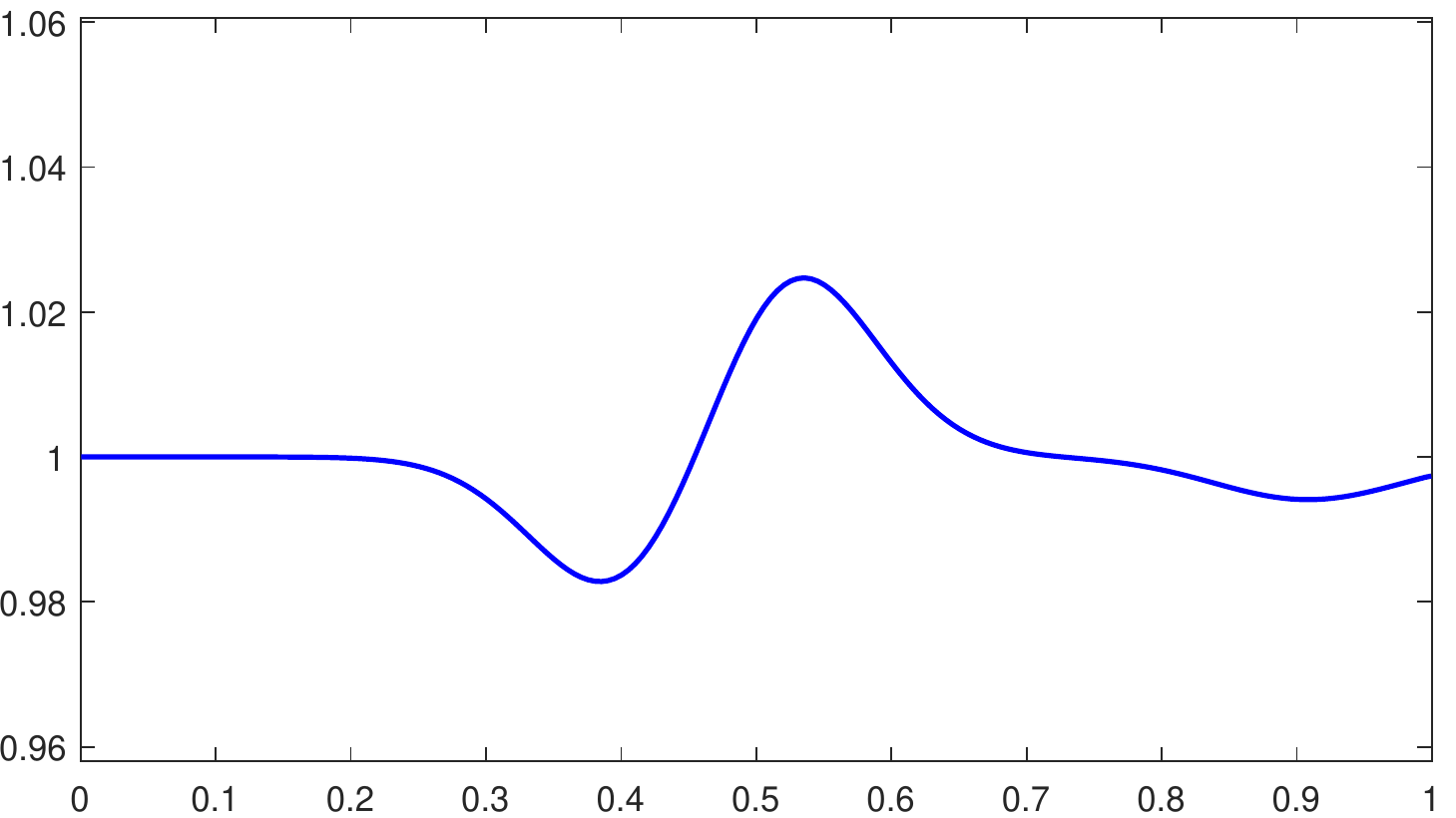}
\subcaption*{$u$ for $t=0.17$}
\end{subfigure}\\[1ex]
\begin{subfigure}[b]{0.49\textwidth}
\includegraphics[width=\textwidth]{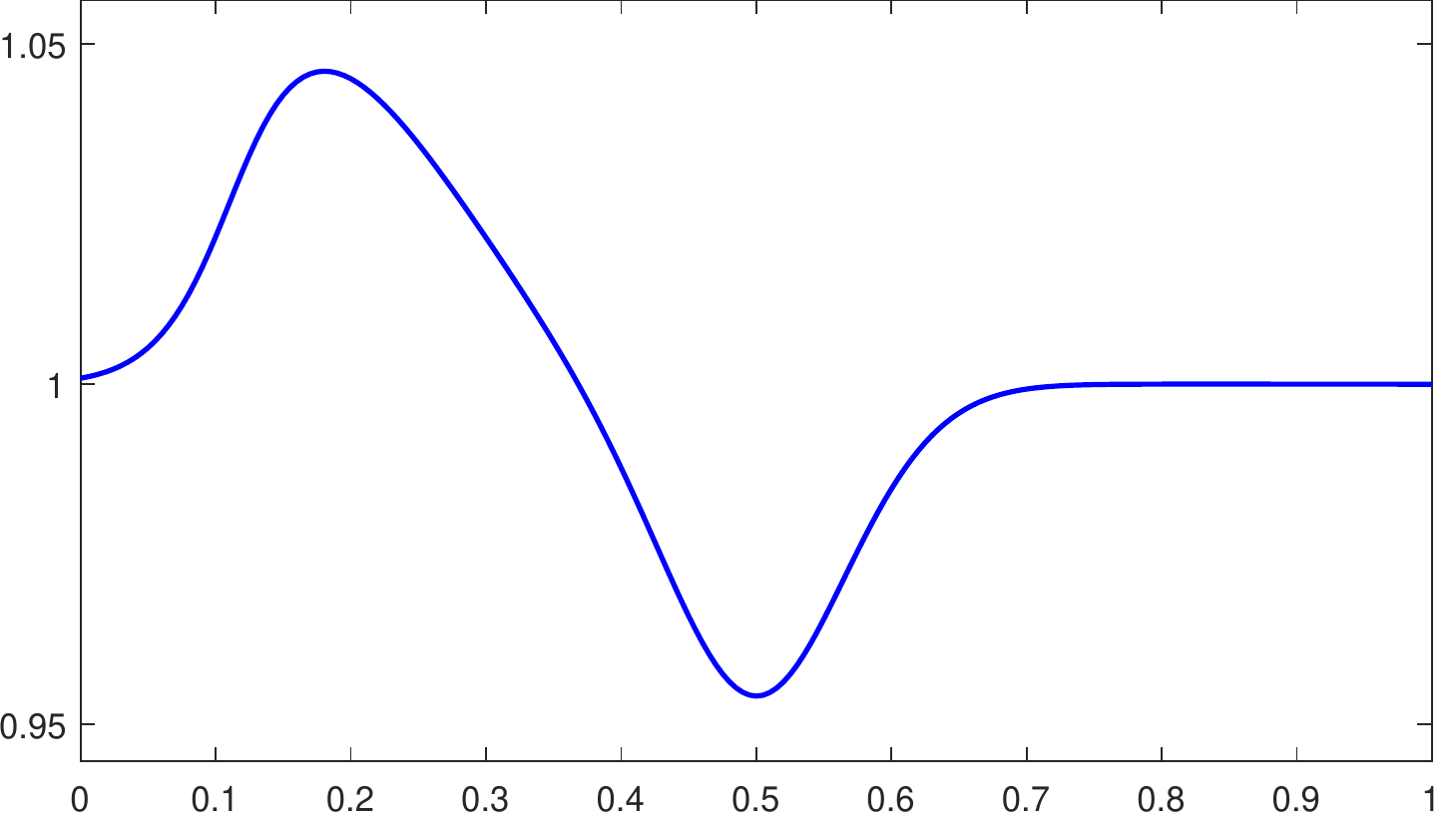}
\subcaption{$\eta$ for $t=0.72$}
\end{subfigure}\hspace{.01\textwidth}
\begin{subfigure}[b]{0.49\textwidth}
\includegraphics[width=\textwidth]{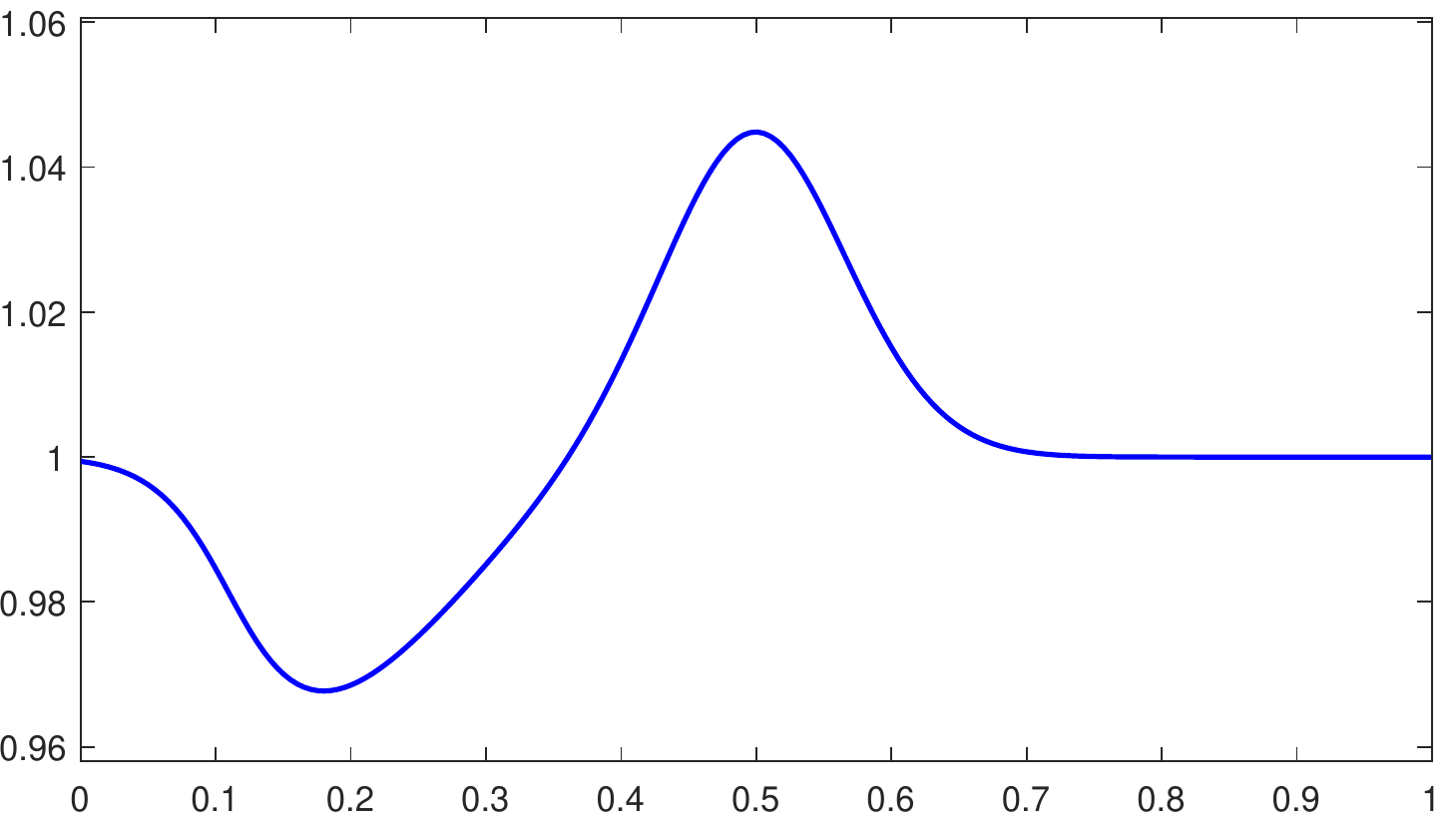}
\subcaption*{$u$ for $t=0.72$}
\end{subfigure}\\[1ex]
\begin{subfigure}[b]{0.49\textwidth}
\includegraphics[width=\textwidth]{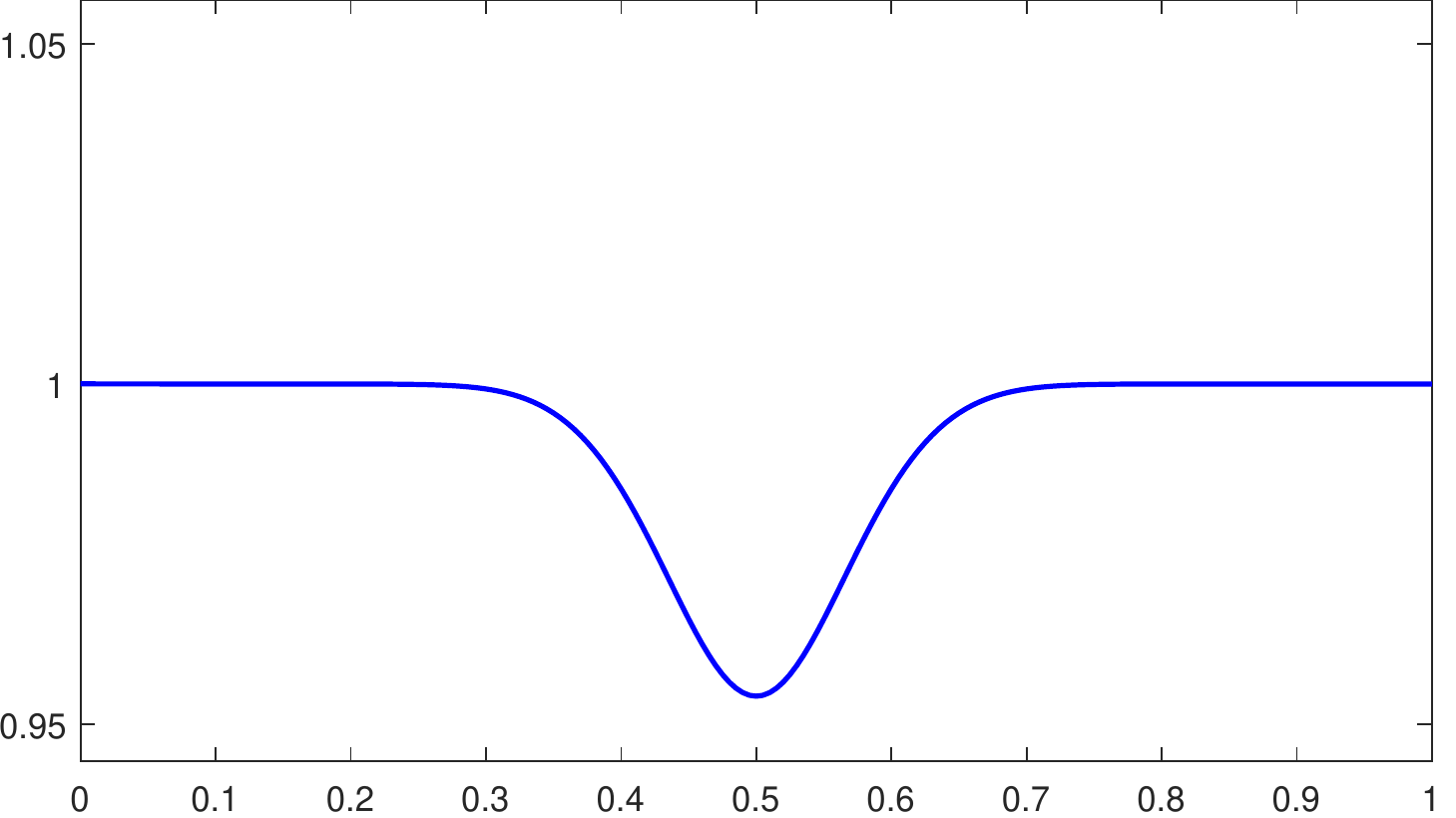}
\subcaption{$\eta$ for $t=2$}
\end{subfigure}\hspace{.01\textwidth}
\begin{subfigure}[b]{0.49\textwidth}
\includegraphics[width=\textwidth]{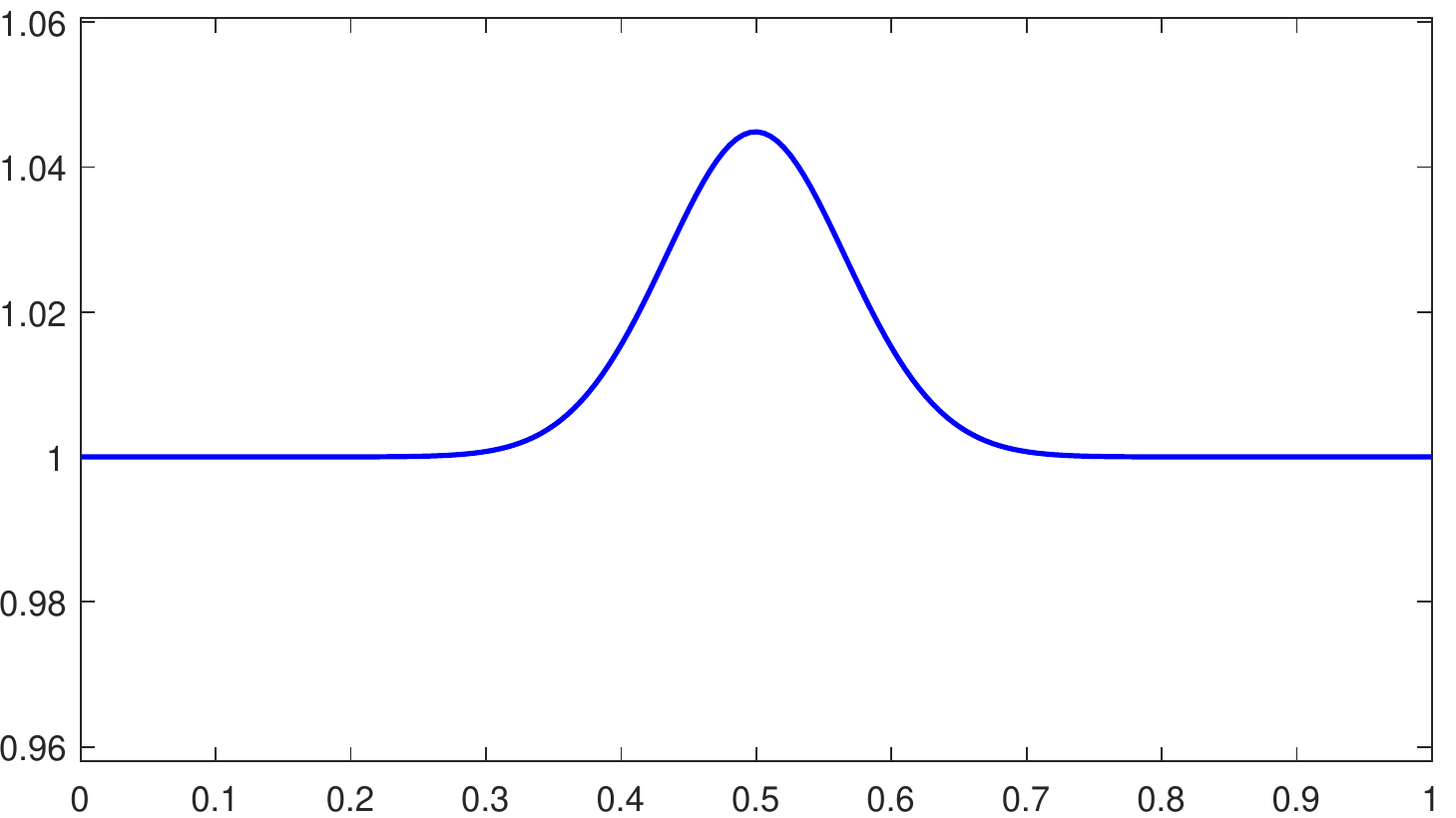}
\subcaption*{$u$ for $t=2$}
\end{subfigure}\\[1ex]
\caption{Evolution with data \eqref{eq:3p7}, subcritical case, $r=2$, 
$h=1/2000$, 
$k=h/10$ \label{fig:3}} \vspace{-1ex} 
\end{figure}

\begin{figure}[htbp]
\centering 

\begin{subfigure}[b]{0.49\textwidth}
\includegraphics[width=\textwidth]{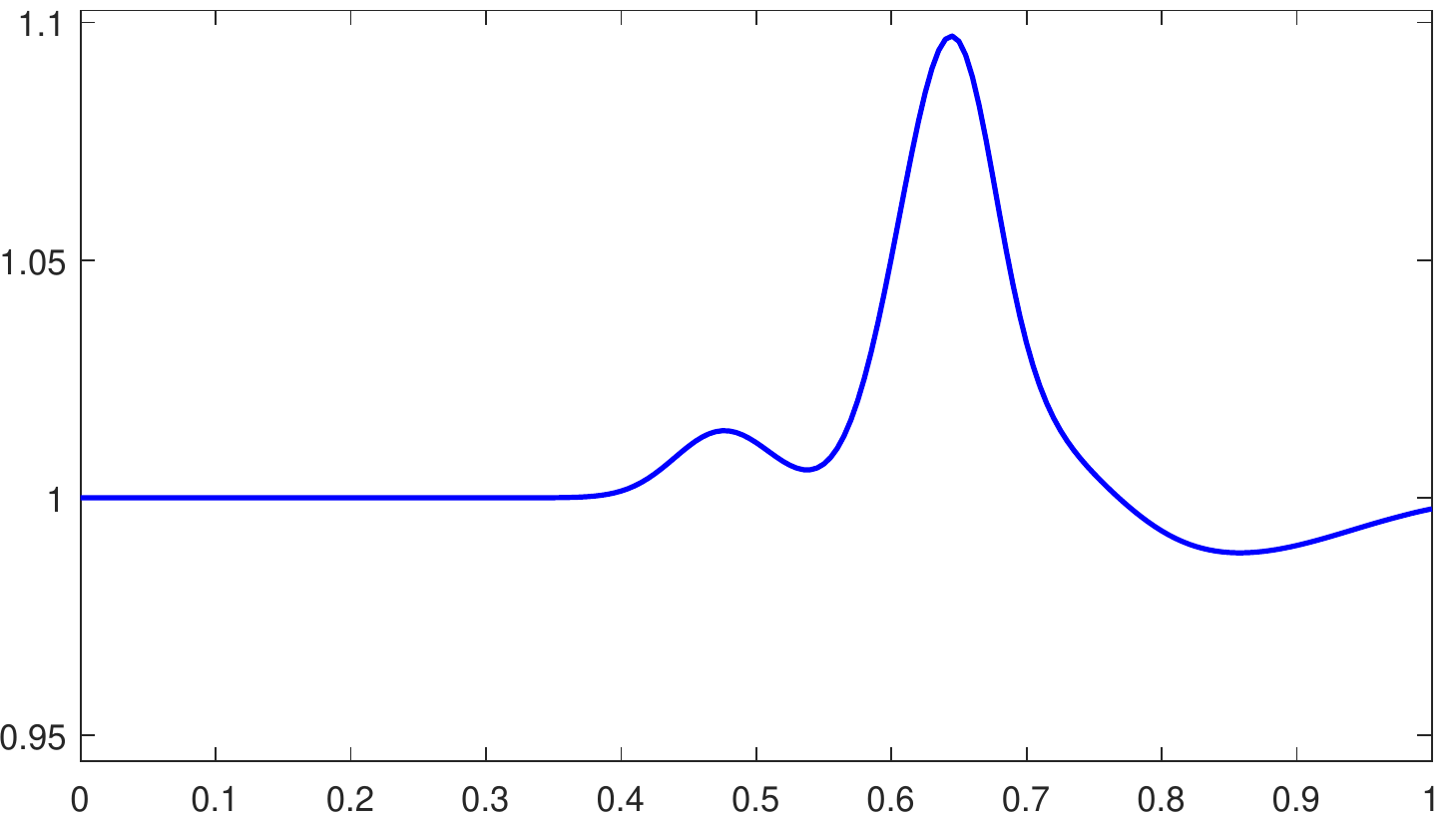}
\subcaption{$\eta$ for $t=0.057$}
\end{subfigure}\hspace{.01\textwidth}
\begin{subfigure}[b]{0.49\textwidth}
\includegraphics[width=\textwidth]{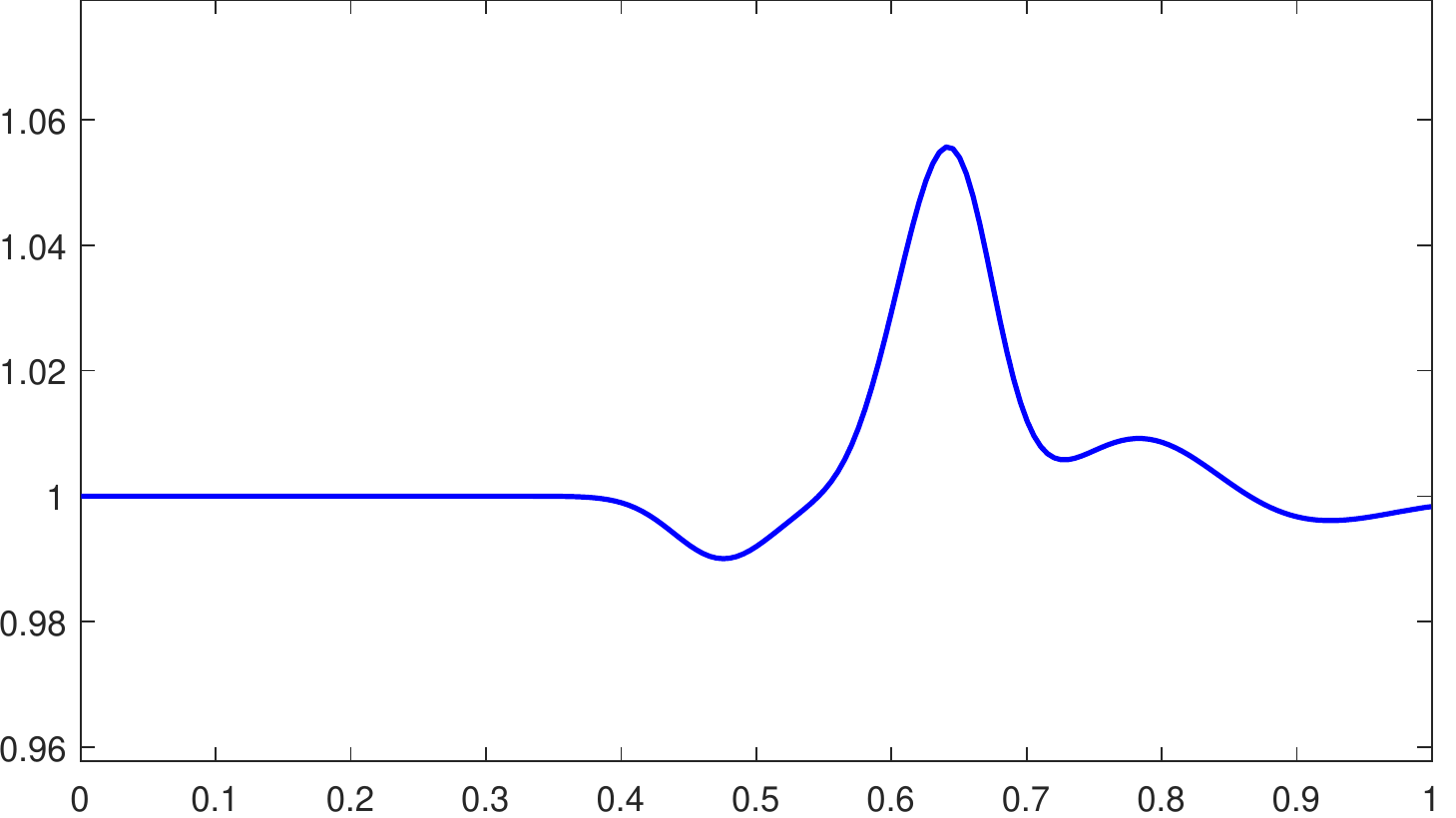}
\subcaption*{$u$ for $t=0.057$}
\end{subfigure}\\[1ex]
\begin{subfigure}[b]{0.49\textwidth}
\includegraphics[width=\textwidth]{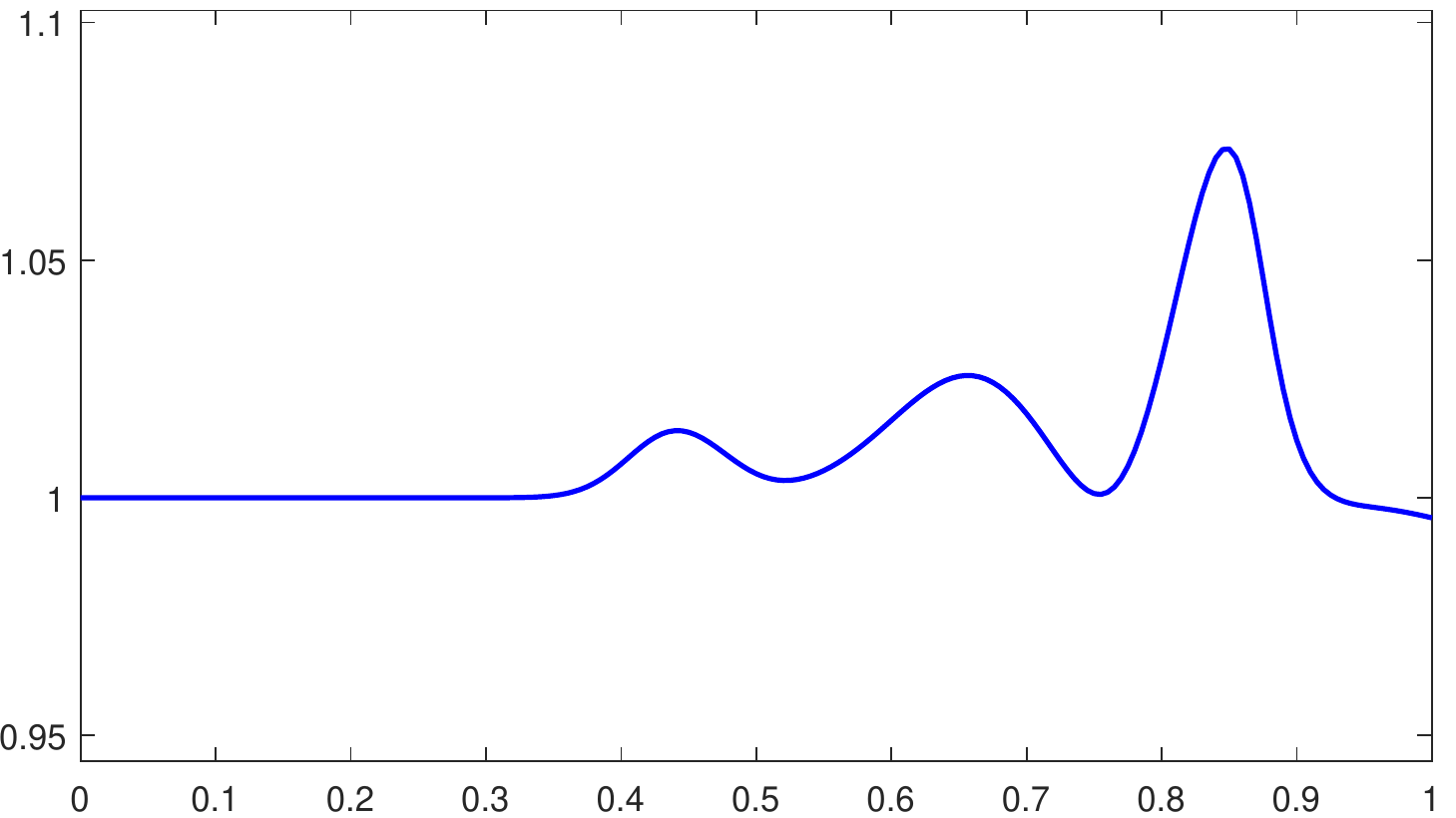}
\subcaption{$\eta$ for $t=0.138$}
\end{subfigure}\hspace{.01\textwidth}
\begin{subfigure}[b]{0.49\textwidth}
\includegraphics[width=\textwidth]{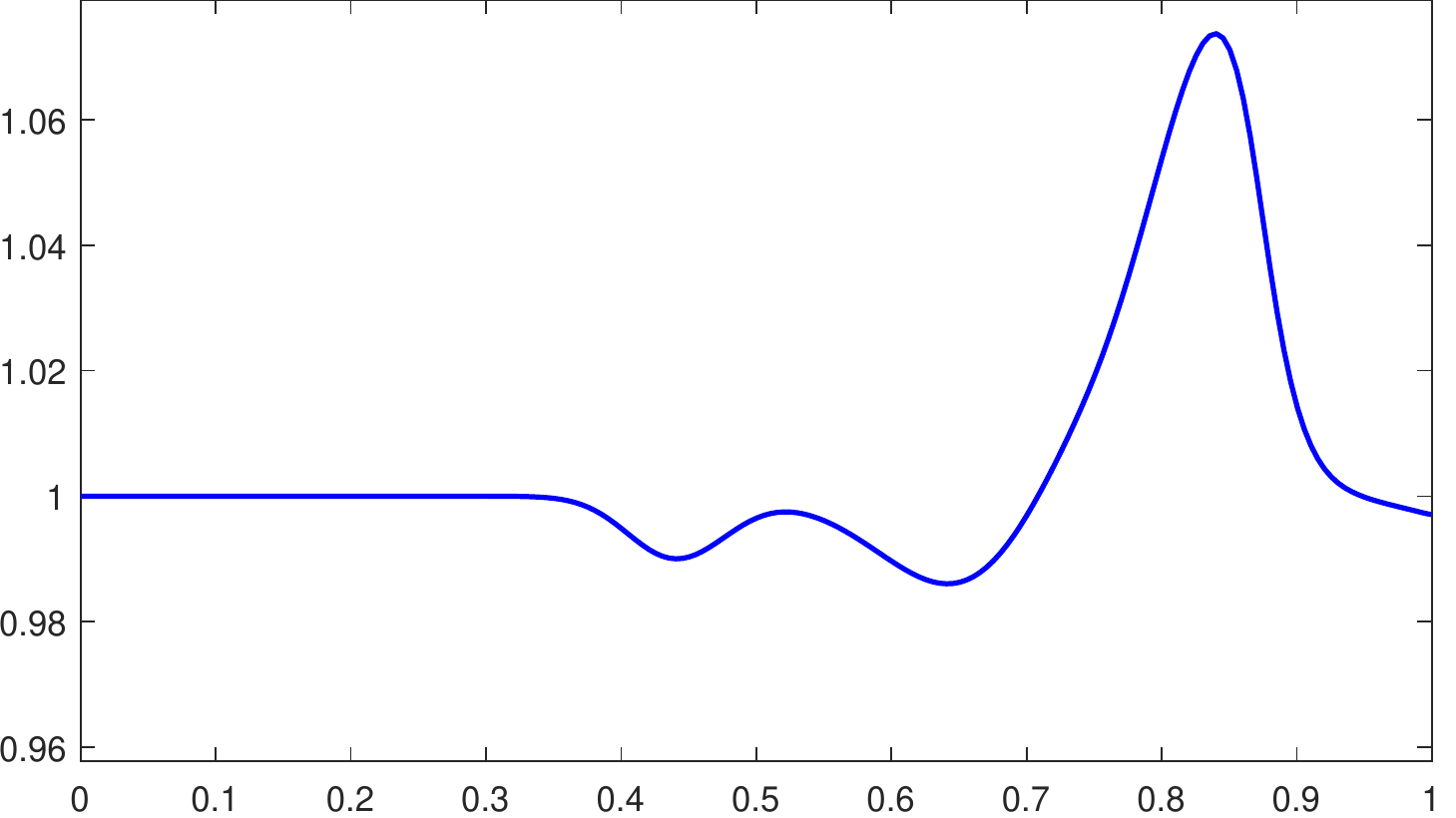}
\subcaption*{$u$ for $t=0.138$}
\end{subfigure}\\[1ex]
\begin{subfigure}[b]{0.49\textwidth}
\includegraphics[width=\textwidth]{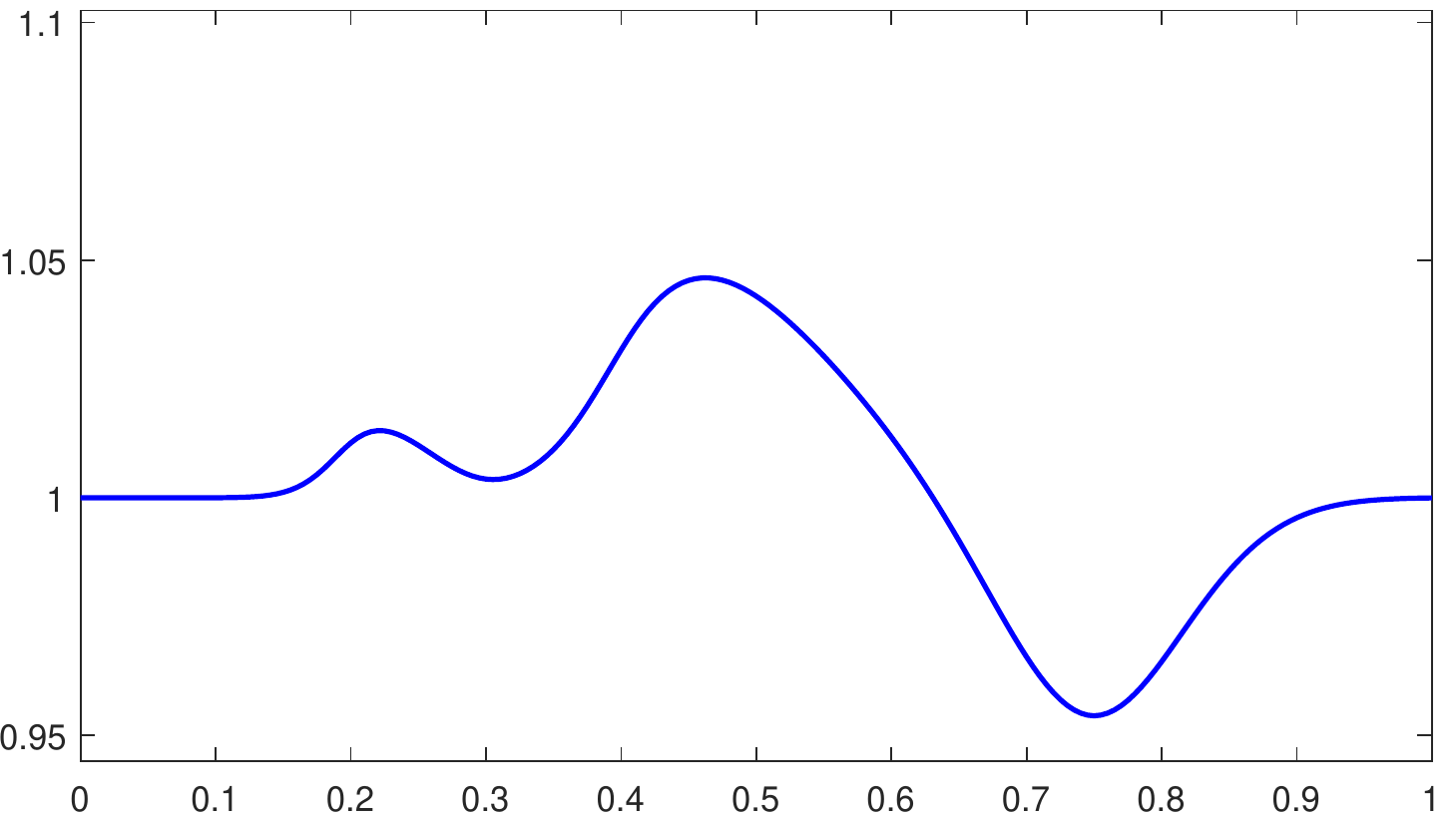}
\subcaption{$\eta$ for $t=0.651$}
\end{subfigure}\hspace{.01\textwidth}
\begin{subfigure}[b]{0.49\textwidth}
\includegraphics[width=\textwidth]{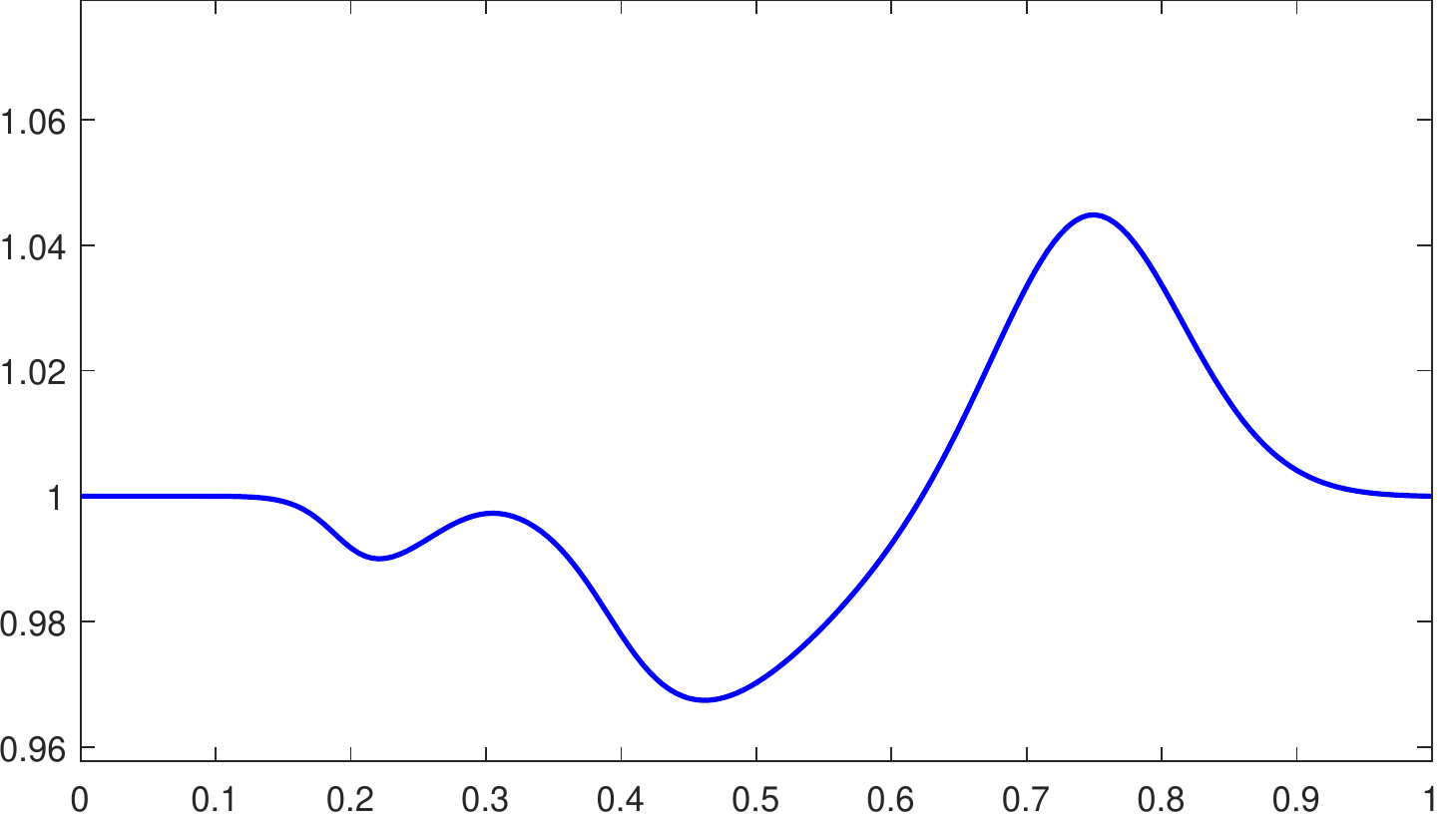}
\subcaption*{$u$ for $t=0.651$}
\end{subfigure}\\[1ex]
\begin{subfigure}[b]{0.49\textwidth}
\includegraphics[width=\textwidth]{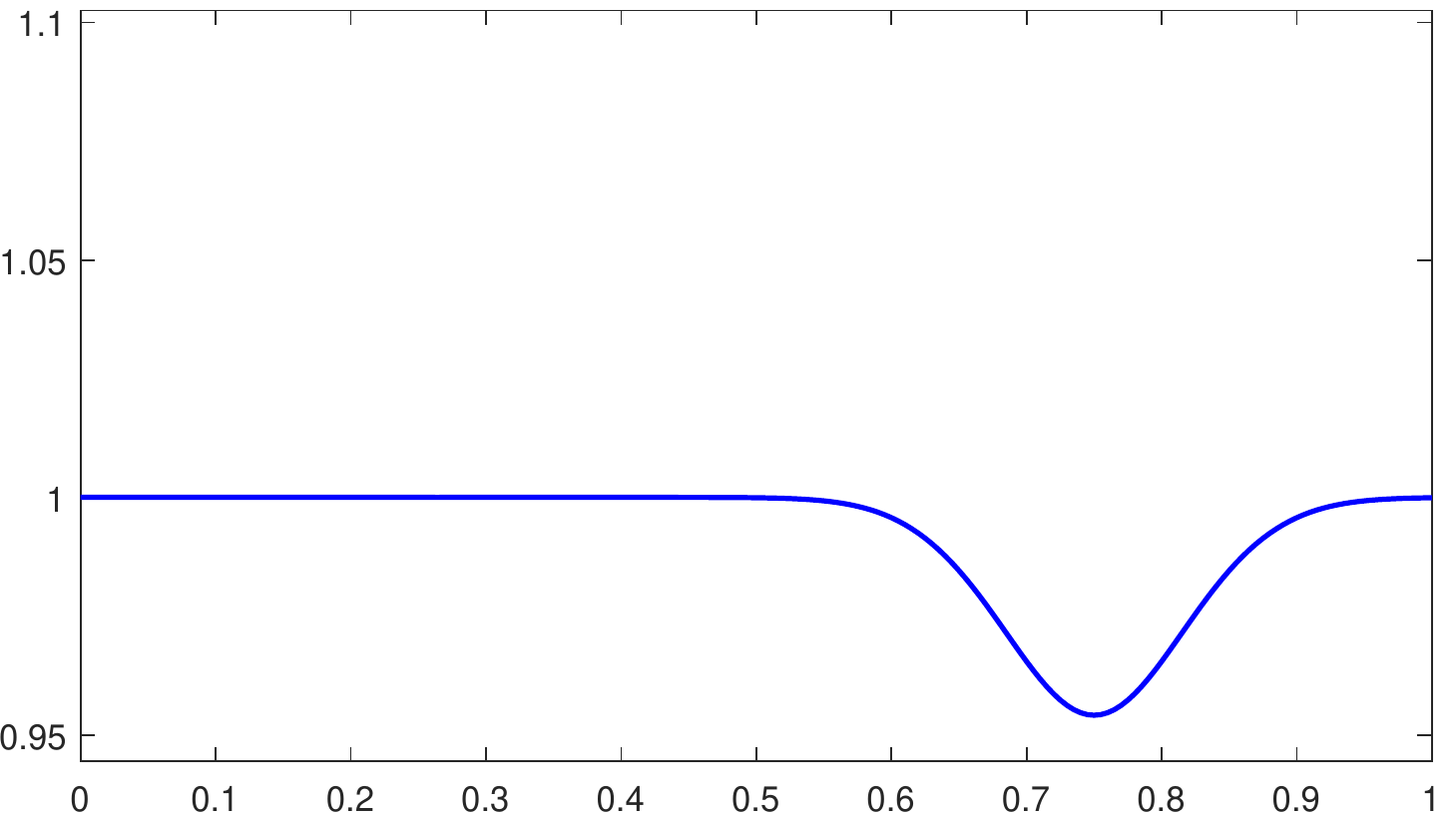}
\subcaption{$\eta$ for $t=3$ (steady state)}
\end{subfigure}\hspace{.01\textwidth}
\begin{subfigure}[b]{0.49\textwidth}
\includegraphics[width=\textwidth]{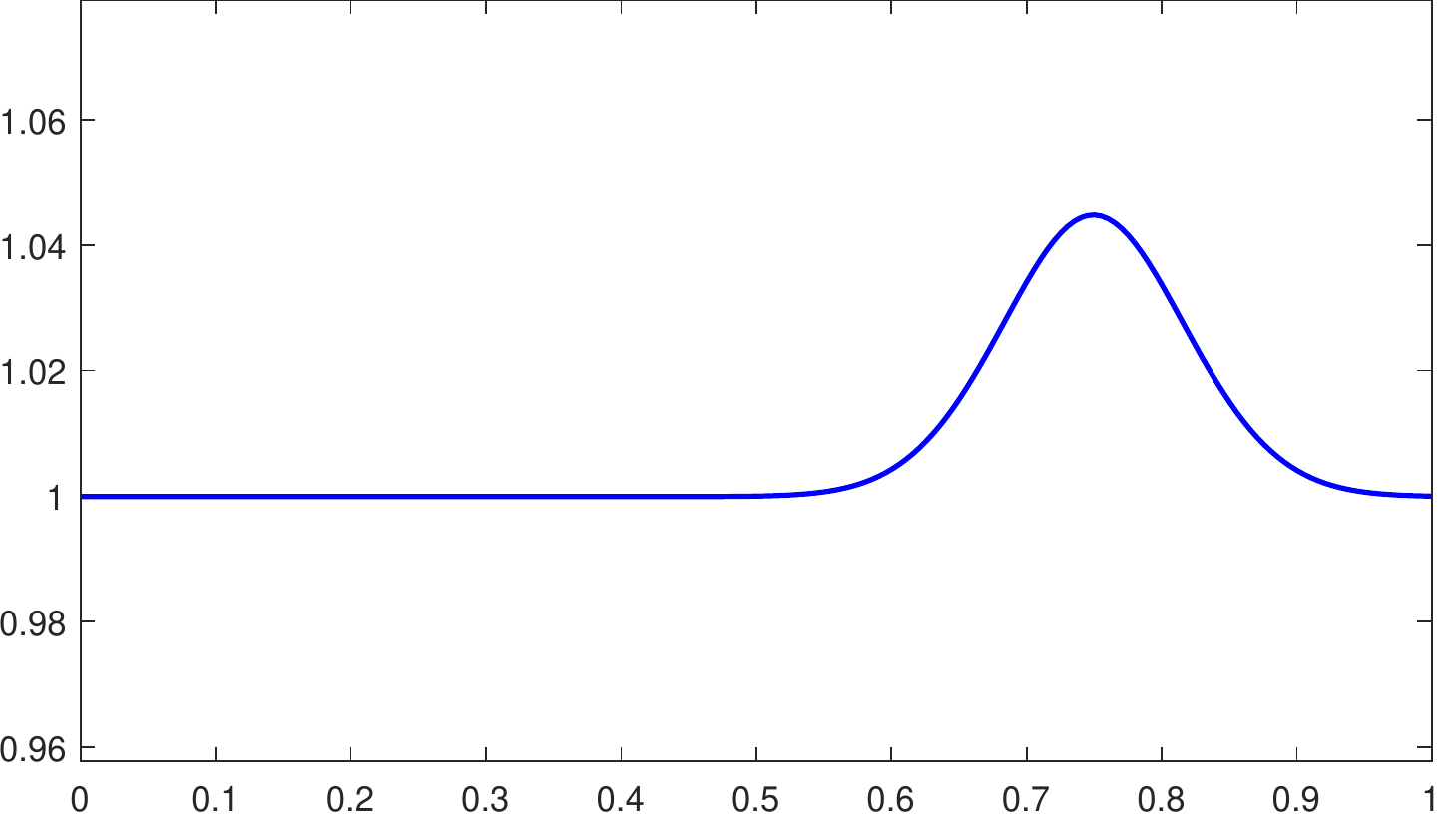}
\subcaption*{$u$ for $t=3$ (steady state)}
\end{subfigure}\\[1ex]
\caption{Evolution with data \eqref{eq:3p8}, subcritical case, $r=2$, 
$h=1/2000$, 
$k=h/10$ \label{fig:4}} \vspace{-1ex} 
\end{figure}

An example of subcritical flow with variable initial conditions is shown in 
Figure \ref{fig:4}, where we took $\eta_0=u_0=1$, and 
\begin{equation} \label{eq:3p8} 
\begin{aligned} 
& \beta(x) = 1-0.04\exp(-100(x-0.75)^2), \\ 
& \eta^0(x) = 0.05\exp(-400(x-0.5)^2)+\eta_0, \\ 
& u^0(x) = 0.1\exp(-400(x-0.5)^2)+u_0 
\end{aligned} 
\end{equation} 
and integrated with $h=1/2000$, $k=h/10$. A two-way wavetrain emerges and 
attains steady-state by $t=3$. 

We also tested the code in a few examples of the shallow water equations with 
absorbing (characteristic) boundary conditions, written in \emph{dimensional} 
form, i.e.\ as 
\begin{equation} \label{eq:3p9} 
\begin{aligned} 
& \eta_t + \left((\beta+\eta)u\right)_x = 0, \\ 
& u_t + g\eta_x + uu_x = 0, 
\end{aligned}\quad 
0\leq x\leq L,\ \ 0\leq t\leq T, 
\end{equation} 
with initial conditions $\eta(x,0)=\eta^0(x)$, $u(x,0)=u^0(x)$, $0\leq x\leq L$, 
and analogous characteristic boundary conditions in the super- and subcritical 
cases. (The Riemann invariants are now $u\pm \sqrt{g(\beta+\eta)}$, $g$ is the 
acceleration of gravity taken as $9.812\, \mathrm m/\mathrm s^2$, and the bottom 
is at $z=-\beta(x)$. If the bottom is horizontal it is located at $z=-h_0$; in 
the general case $h_0$ will be a typical depth.) 

As an example of \emph{supercritical} flow we considered a numerical experiment 
similar to the one described in Section 8.2 of \cite{Shiue}. Let $\widetilde 
\beta$ be the trapezoidal profile given by 
\begin{equation} \label{eq:3p10} 
\widetilde \beta(x) = 
\begin{cases} 
\displaystyle \frac{\delta_0}{c\kappa - \kappa/2}\left(x-\frac L 2 + 
c\kappa\right), & \text{if}\ \ -c\kappa \leq x-L/2\leq -\kappa/2, \\ 
\displaystyle \delta_0, & \text{if}\ \ -\kappa/2\leq x-L/2\leq \kappa/2, \\ 
\displaystyle -\frac{\delta_0}{c\kappa - \kappa/2}\left(x-\frac L 2 
-c\kappa\right), & \text{if}\qquad \kappa/2\leq x-L/2\leq c\kappa, \\ 
\displaystyle 0, & \text{otherwise}, 
\end{cases} 
\end{equation} 
where $L = 10^6\,\mathrm m$, $\delta_0 = 500\,\mathrm m$, $k=L/10$. The bottom 
was located at $z=-\beta(x)$, where $\beta(x) = h_0-\widetilde \beta(x)$, $h_0 = 
1000\,\mathrm m$, and the problem \eqref{eq:3p9} was solved with characteristic 
boundary conditions and initial conditions $\eta^0(x) = \eta_0 = 0$ and 
$u^0(x)=u_0$, where the constant $u_0$ was varied in order to give flows with 
different Froude numbers $Fr = u_0/\sqrt{gh_0}$. We solved 
\eqref{eq:3p9}--\eqref{eq:3p10} numerically with piecewise linear elements and 
RK4 on a uniform mesh with $h=1000\,\mathrm m$, $k = 1\,\mathrm s$. Some 
profiles of the steady state of the free surface $\eta$ and the associated 
bottom function $\beta(x)$ for various Froude numbers and values of the 
parameter $c$ are shown in Figure \ref{fig:5}. As expected the eventual 
maximum value of $\eta$ decreases as $Fr$ increases; the results are consistent 
with those of \cite{Shiue}. 

\begin{figure}[htb]
\centering 

\begin{subfigure}[b]{0.49\textwidth}
\includegraphics[width=\textwidth]{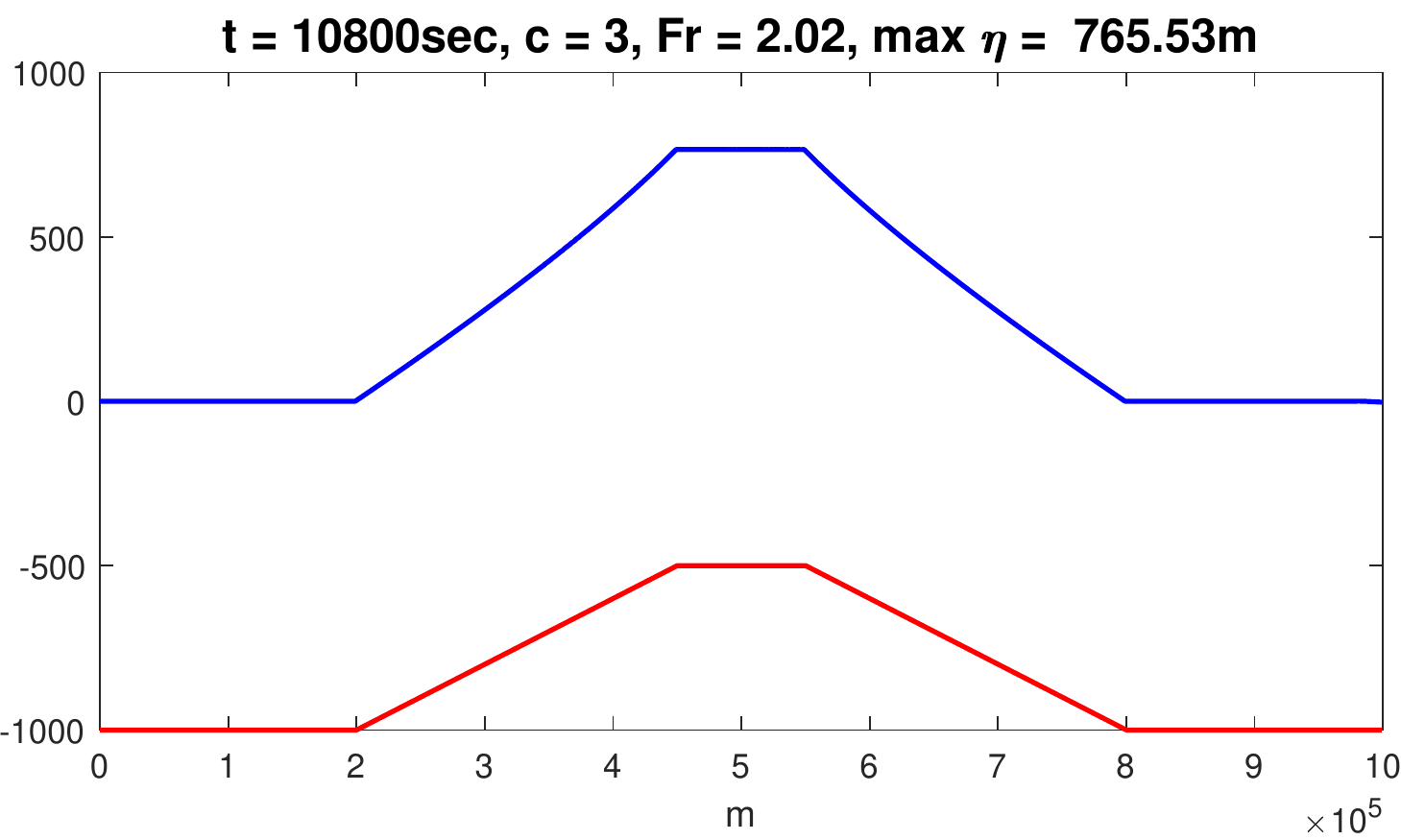}
\end{subfigure}\hspace{.01\textwidth}
\begin{subfigure}[b]{0.49\textwidth}
\includegraphics[width=\textwidth]{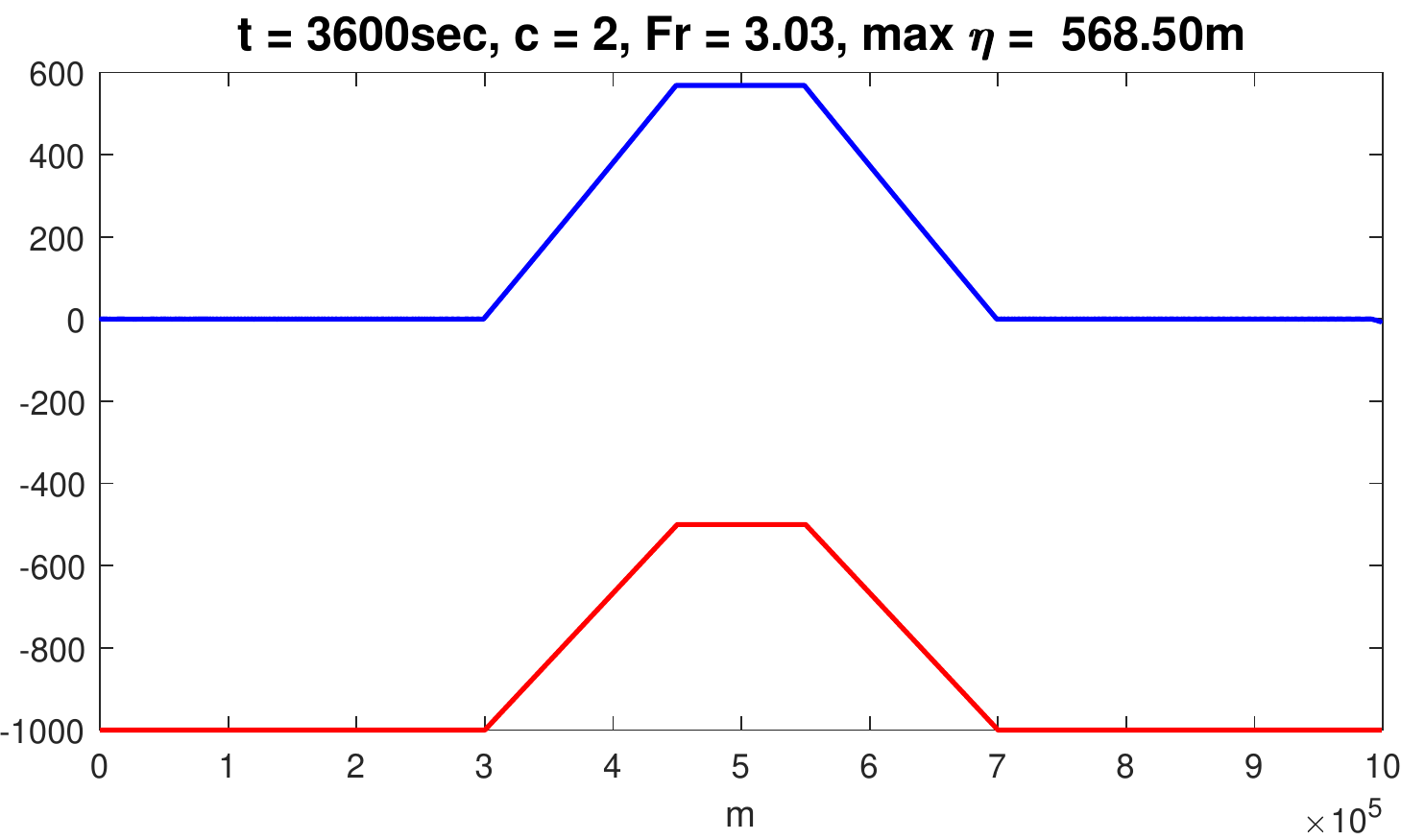}
\end{subfigure}\\[1ex]
\begin{subfigure}[b]{0.49\textwidth}
\includegraphics[width=\textwidth]{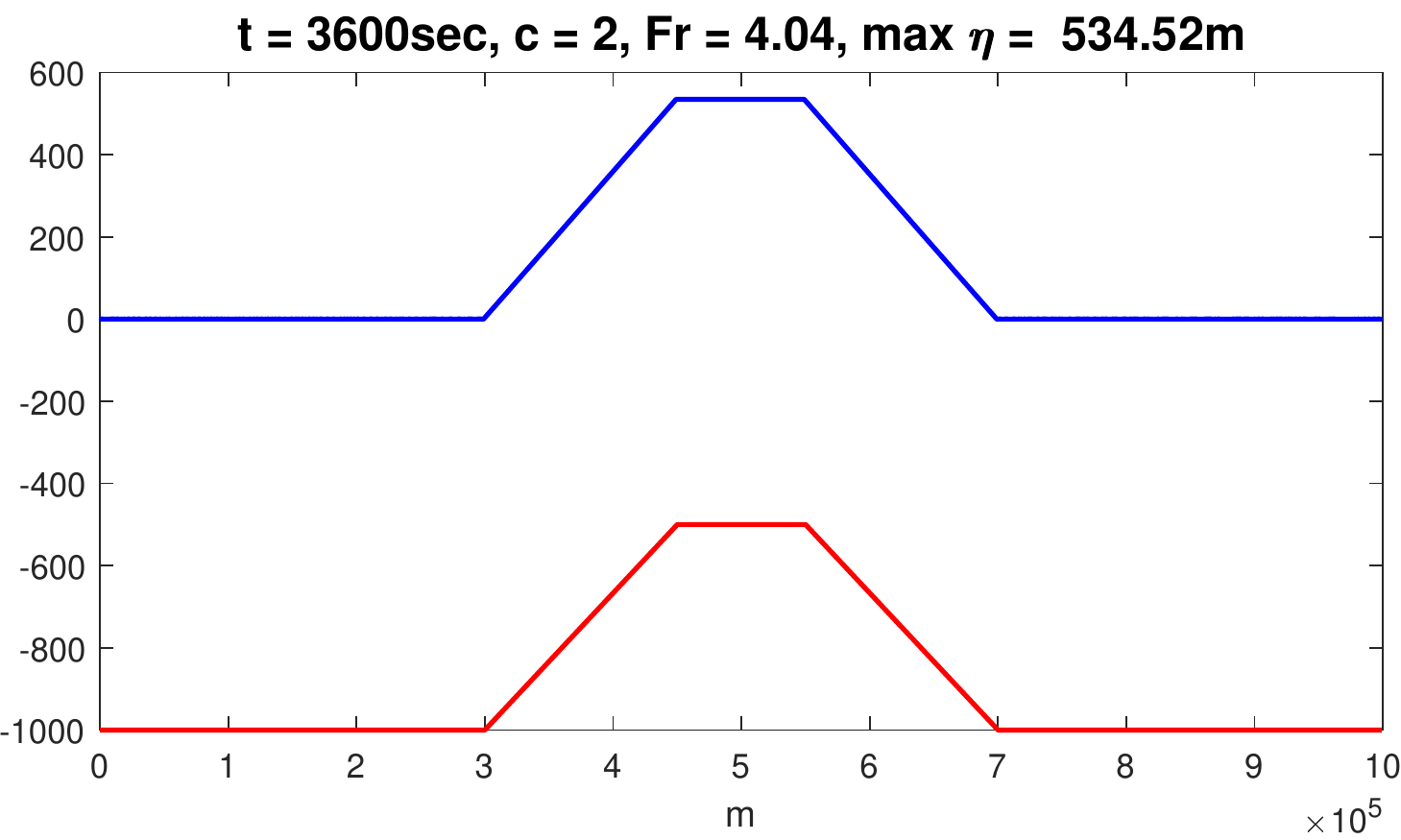}
\end{subfigure}\hspace{.01\textwidth}
\begin{subfigure}[b]{0.49\textwidth}
\includegraphics[width=\textwidth]{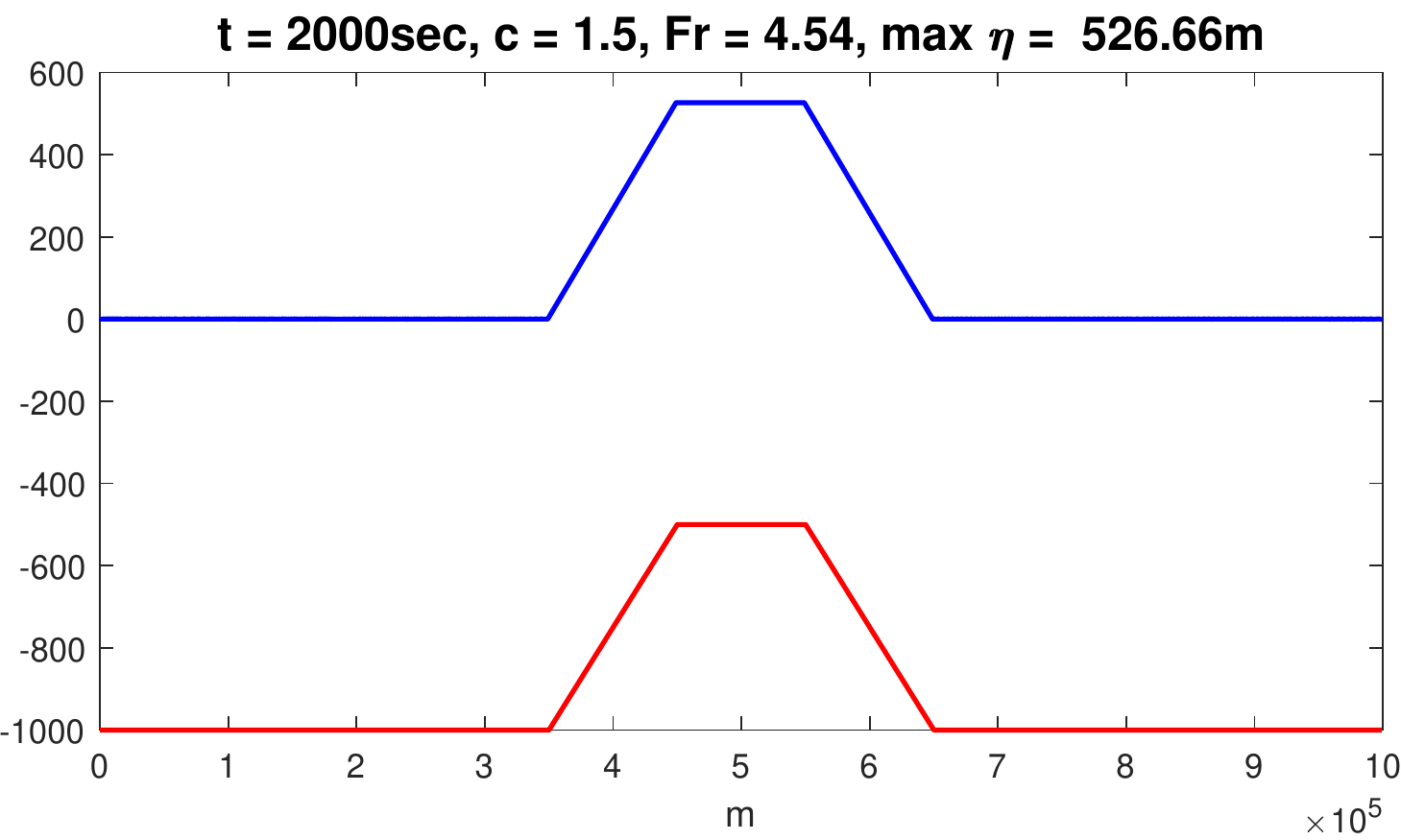}
\end{subfigure}\\[1ex]
\caption{Supercritical flows over a trapezoidal bottom, 
\eqref{eq:3p9}--\eqref{eq:3p10}, $r=2$, $h=1000\,\mathrm m$, $k = 1\,\mathrm s$. 
(Upper figures: steady-state $\eta(x)$; lower figures: $\beta(x)$.)\label{fig:5}} 
\end{figure}

In an example of a dimensional \emph{subcritical} flow we modified the profile 
given in \S 5.1 of \cite{Shiue} in order to avoid discontinuity formation. Thus, 
the initial $\eta$-profile was rounded and its amplitude decreased. Let 
$\widetilde\beta$ be defined by 
\begin{equation} \label{eq:3p11} 
\widetilde\beta(x) = 
\begin{cases} 
\displaystyle \frac \delta 2 + \frac \delta 2 \cos 
\left[\frac{\pi(x-L/2)}{\kappa}\right], & \displaystyle \text{if}\ \ \left|x - 
\frac L 2 \right|<\kappa, \\ 
0, & \text{otherwise}, 
\end{cases} 
\end{equation} 
where $L=10^6\,\mathrm m$, $\delta = 5000\,\mathrm m$, $k=L/10$. The bottom was 
taken at $z=-\beta(x)$, where $\beta(x) = h_0-\widetilde\beta(x)$, $h_0 = 
10^4\,\mathrm m$, and the problem \eqref{eq:3p9} was solved with characteristic 
boundary conditions with $\eta_0 = u_0 = 0$ and $\eta^0(x) = 
0.2\,\varepsilon\,h_0\exp \left[-5{\cdot}10^{-8}\left(x-3L/20)/10\right)^2\right]$, 
$0\leq x\leq L$, where $\varepsilon = 0.2$. 

\begin{figure}[htbp]
\centering 

\begin{subfigure}[b]{0.49\textwidth}
\includegraphics[width=\textwidth]{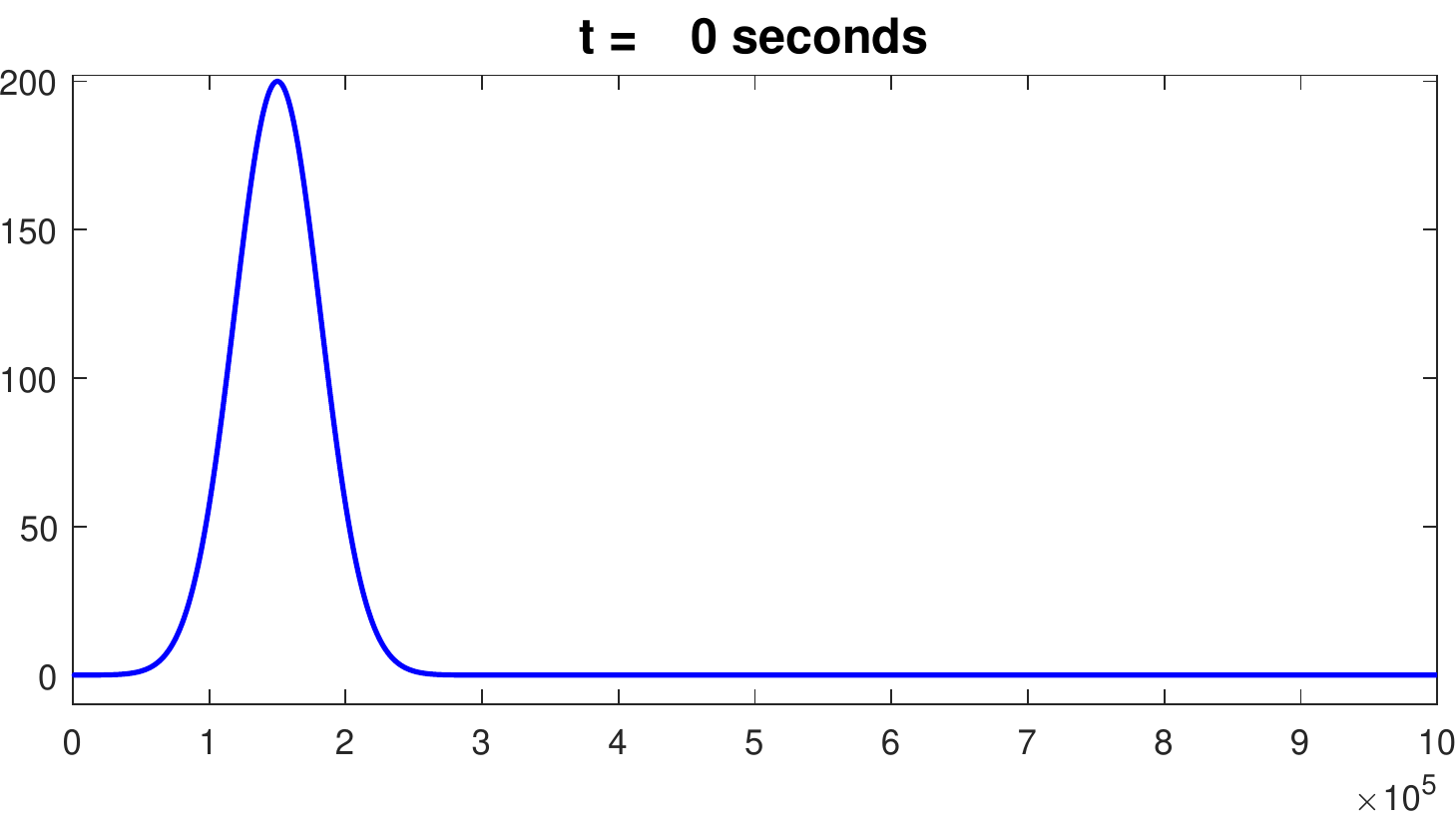}
\end{subfigure}\hspace{.01\textwidth}
\begin{subfigure}[b]{0.49\textwidth}
\includegraphics[width=\textwidth]{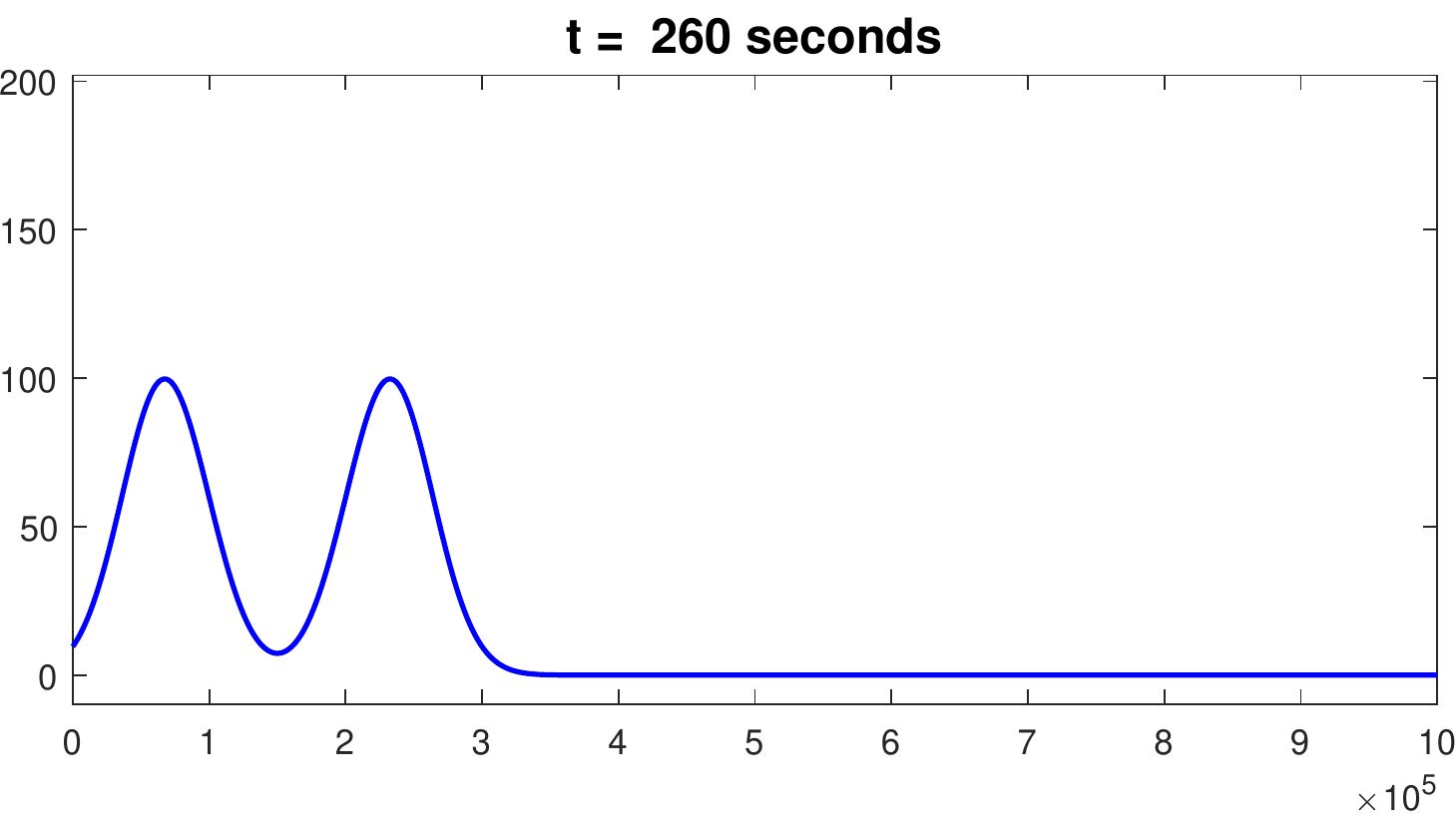}
\end{subfigure}\\[1ex]
\begin{subfigure}[b]{0.49\textwidth}
\includegraphics[width=\textwidth]{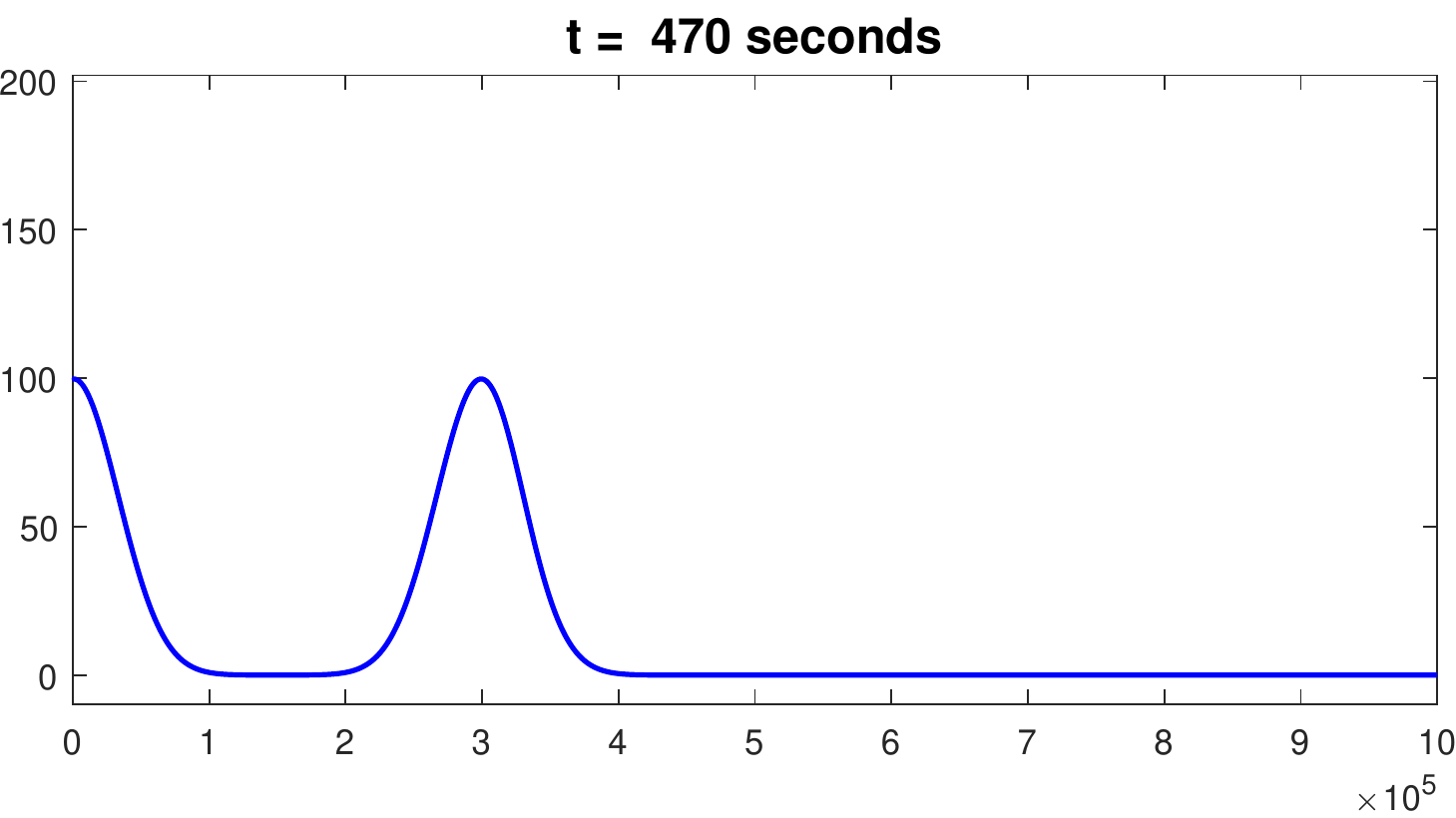}
\end{subfigure}\hspace{.01\textwidth}
\begin{subfigure}[b]{0.49\textwidth}
\includegraphics[width=\textwidth]{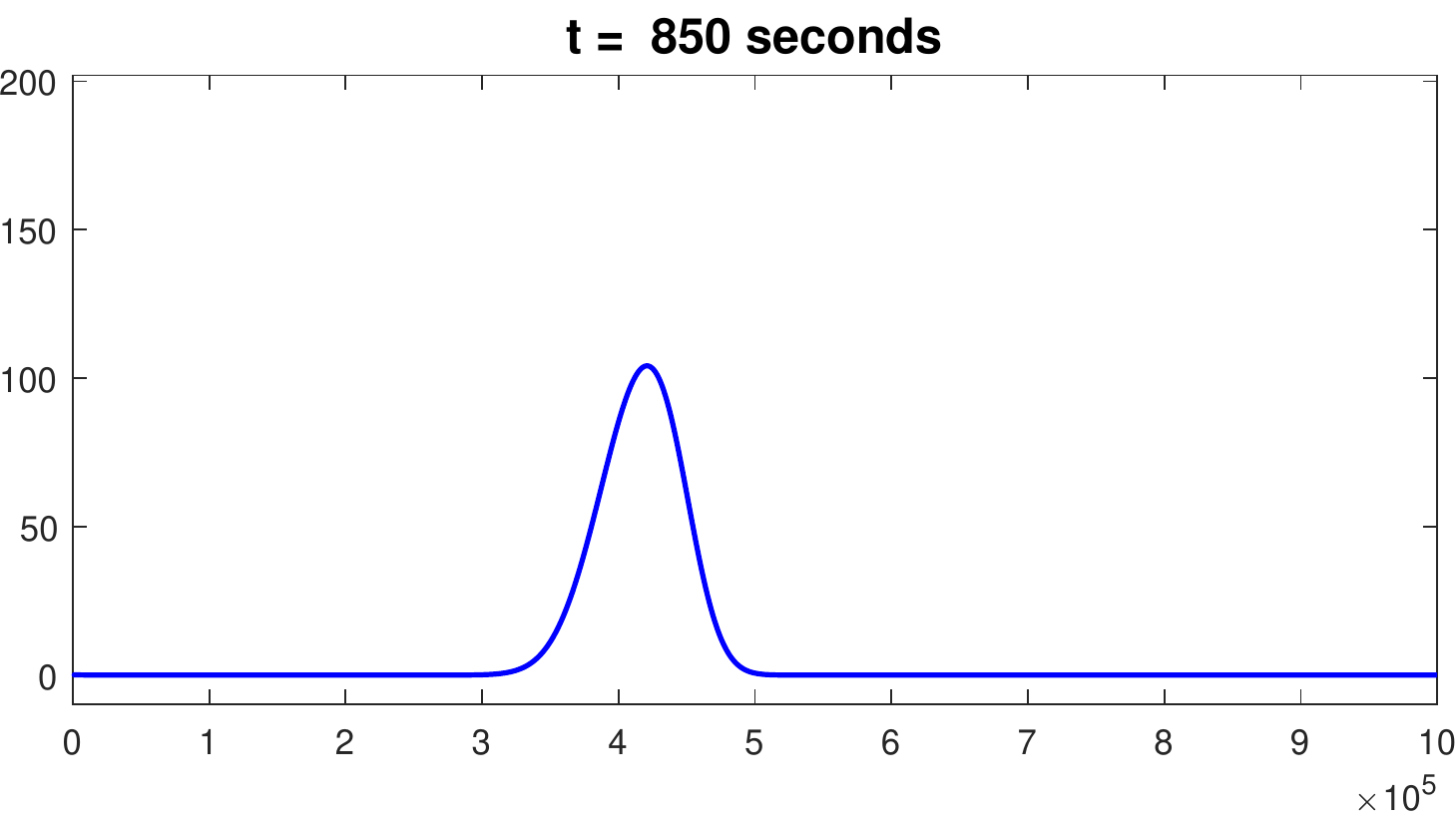}
\end{subfigure}\\[1ex]
\begin{subfigure}[b]{0.49\textwidth}
\includegraphics[width=\textwidth]{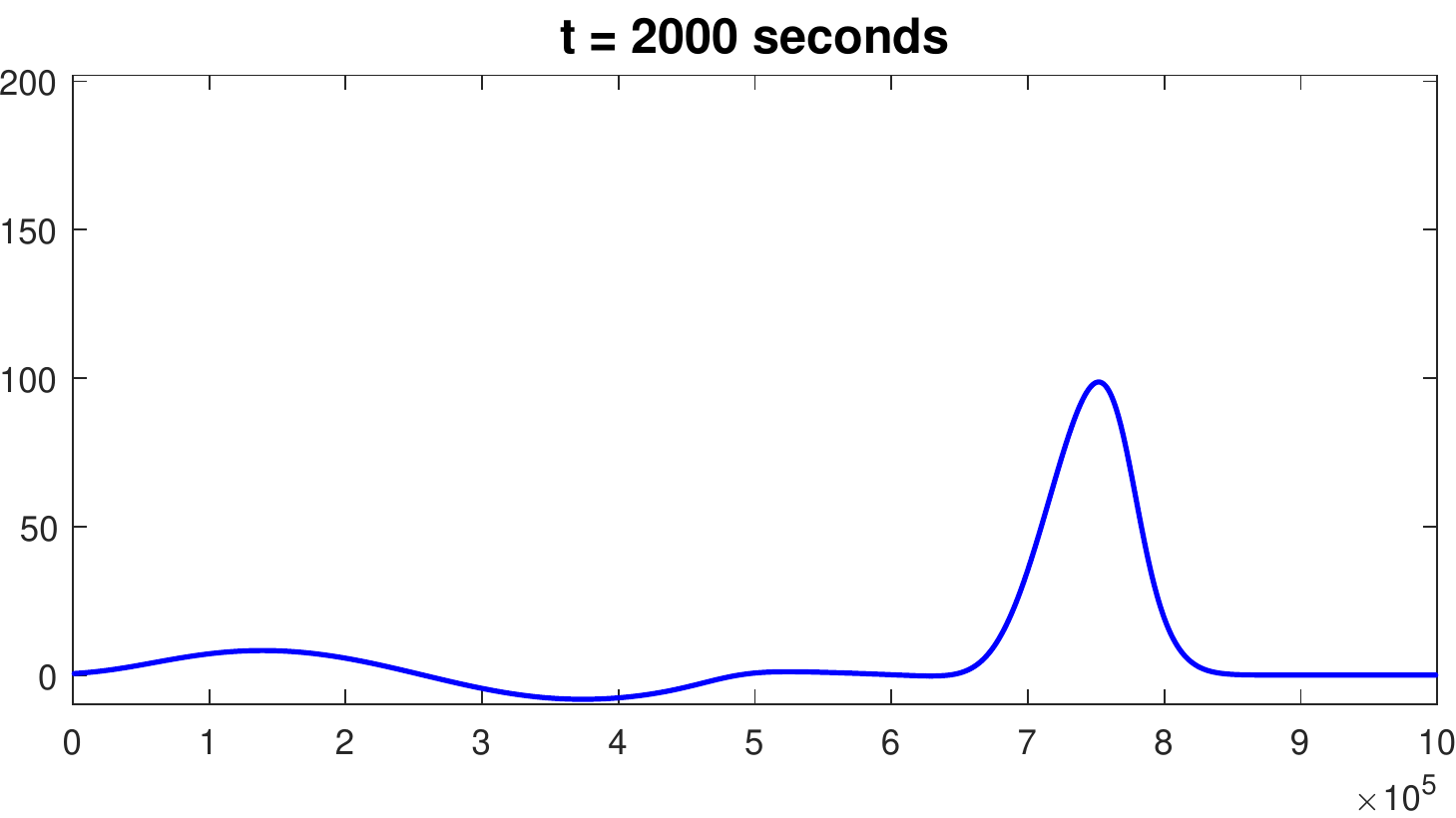}
\end{subfigure}\hspace{.01\textwidth}
\begin{subfigure}[b]{0.49\textwidth}
\includegraphics[width=\textwidth]{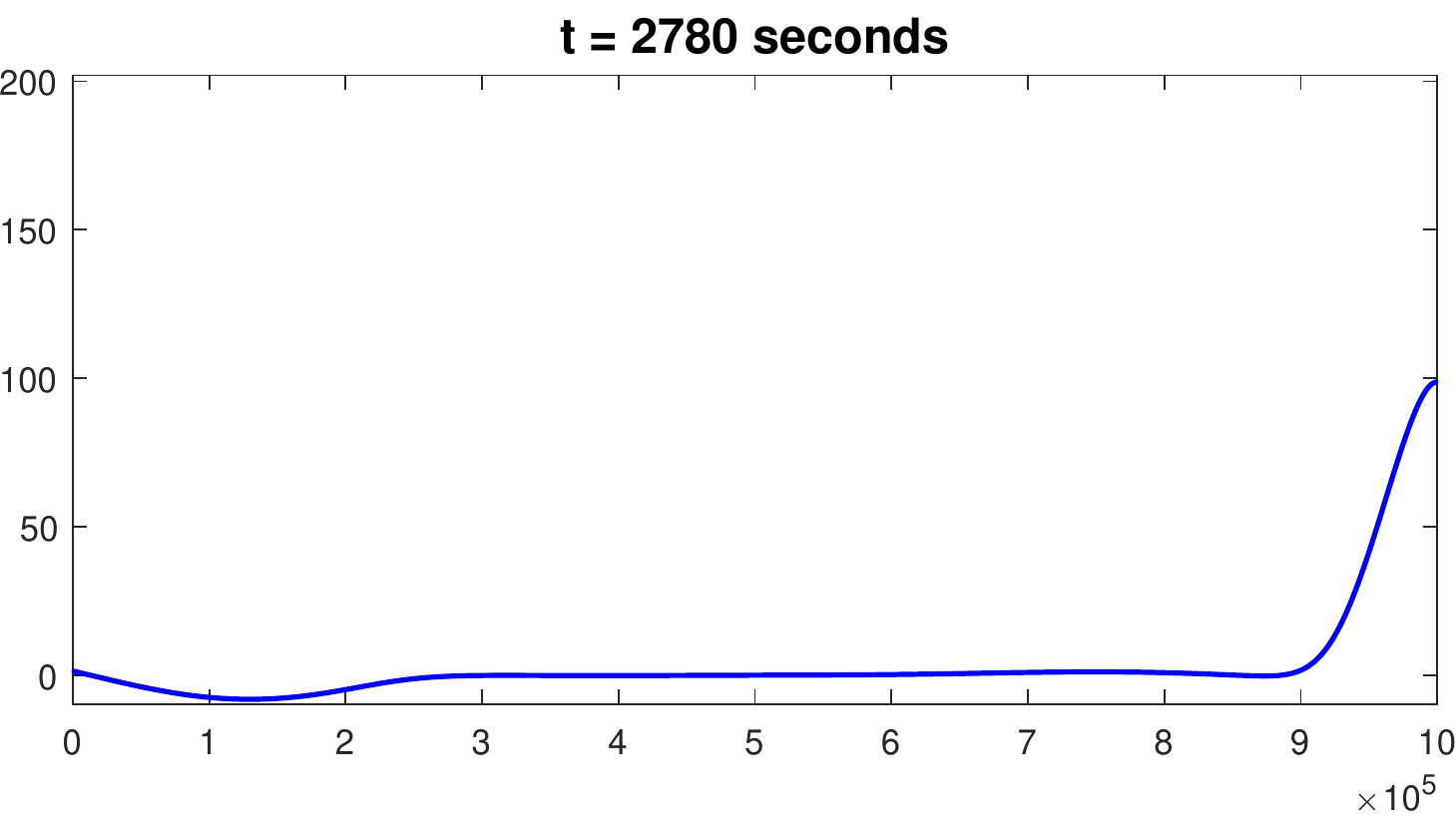}
\end{subfigure}\\[1ex]
\begin{subfigure}[b]{0.49\textwidth}
\includegraphics[width=\textwidth]{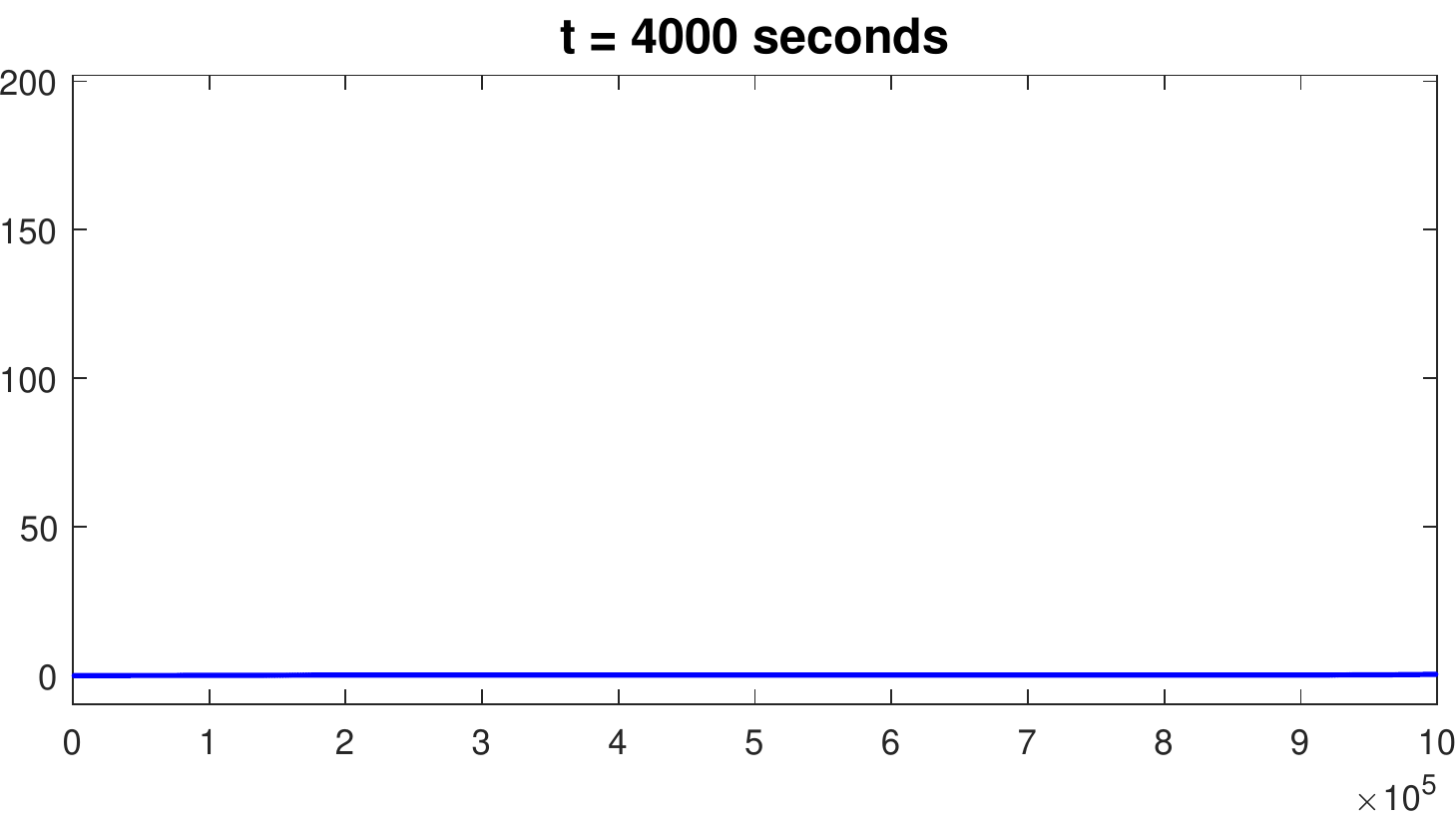}
\end{subfigure}\\[1ex]
\caption{Subcritical flow over a hump, \eqref{eq:3p9}, \eqref{eq:3p11}, $r=2$, 
$h=1000\,\mathrm m$, $k = 1\,\mathrm s$, $\eta$ profiles\label{fig:6}} 
\end{figure}

The evolution of the $\eta$-profiles is shown in Figure \ref{fig:6} up to 
the attainment of the steady state $\eta = u = 0$. The results resemble 
qualitatively those of \cite{Shiue}.

\subsection{Shallow water equations in balance-law form} 
\label{sec:3p2} 

In this section we consider the numerical solution by the standard Galerkin 
finite element method of the shallow water equations written in 
\emph{balance-law form} (i.e.\ in conservation-law form with a source term), as 
\begin{equation} \label{eq:3p12} 
\begin{aligned} 
& d_t + (du)_x = 0, \\ 
& (du)_t + \left(du^2 + \tfrac 1 2 d^2\right)_x = \beta^\prime(x)d, 
\end{aligned} 
\end{equation} 
where $d=\eta+\beta$ is the water depth assumed as always to be positive; the 
variables in \eqref{eq:3p12} are nondimensional. Is is straightforward to see 
that the system \eqref{eq:3p12} is equivalent to \eqref{eq:SW} since $d\neq 0$. 
In the sequel we will consider the periodic initial-value problem for 
\eqref{eq:3p12} on the spatial interval $[0,1]$ and assume that it has 
sufficiently smooth solutions for $t\in [0,T]$, provided that $\beta$ is smooth 
and $1$-periodic. We will discretize the problem in space on a uniform or 
quasiuniform mesh $\{x_i\}$ in $[0,1]$ and seek approximations $d_h$, $u_h$ of 
$d$, $u$, respectively, in the finite element space $S_{h,p} = \{\phi\in C_p^k: 
\phi\big|_{[x_j,x_{j+1}]}\in\mathbb P_{r-1}, \text{all}\ j\}$, where as usual 
$r$, $k$ are integers such that $r\geq 2$, $0\leq k\le r-2$, and $C_p^k$ are the 
$k$ times continuously differentiable, periodic functions on $[0,1]$. The 
semidiscrete approximations satisfy 
\begin{align} 
& \begin{aligned} 
& (d_{ht},\phi) + ((d_hu_h,\phi) = 0, \\ 
& ((d_hu_h)_t,\phi) + ((d_hu_h^2 + \tfrac 1 2 d_h^2)_x,\phi) = 
(\beta^\prime d_h,\phi), 
\end{aligned}\ \ \forall \phi\in S_{h,p},\ 0\leq t\leq T, \label{eq:3p13} \\ 
& d_h(0) = \opP d^0,\quad u_h(0) = \opP u^0, \label{eq:3p14} 
\end{align} 
where $d^0$, $u^0$ are the initial conditions of $d$ and $u$ and $\opP$ is now 
the $L^2$ projection operator onto $S_{h,p}$. (The second equation in 
\eqref{eq:3p13} is advanced in time for the variable $v_h=d_hu_h$ and $u_h$ is 
recovered as $v_h/d_h$.) In the case of a uniform mesh it is expected that the 
$L^2$ errors of the semidiscrete solution will be of $\mathcal O(h^r)$ while, 
for a quasiuniform mesh, of $\mathcal O(h^{r-1})$, cf.\ \cite{AD1}. We 
verified these rates of accuracy in numerical experiments using $C^0$ linear, 
$C^2$ cubic and $C^4$ quintic splines (i.e.\ spaces $S_{h,p}$ with $r=2$, $4$, 
and $6$, respectively) on uniform and nonuniform spatial meshes, coupled with 
explicit Runge Kutta schemes of third, fourth, and sixth order of accuracy,  
respectively. The fully discrete methods were stable under Courant number 
restrictions. We note that in order to preserve the optimal order of accuracy, 
say in the case of a uniform mesh, one has to compute the integrals that occur in 
the finite element equations using, on each subinterval $[x_i,x_{i+1}]$, an 
$s$-point Gauss quadrature rule with $s\geq r-1$. For example, in the case of a 
cubic spline spatial discretization, a 3-point Gauss rule is sufficient. 

It is interesting to examine whether the method \eqref{eq:3p13}--\eqref{eq:3p14} 
preserves the still water solution $\eta=0$, $u=0$, e.g.\ of the periodic 
i.v.p.\ for the shallow water equations in the form \eqref{eq:3p12}. 
Discretizations that approximate accurately this solution are called `well 
balanced', cf.\ e.g.\ \cite{BV}, and \cite{XZS} and its references. (It is easy 
to check that the standard Galerkin semidiscretization of e.g.\ the periodic ivp 
for \eqref{eq:SW}, i.e.\ for the shallow water equations in their 
`nonconservative' form, is trivially well-balanced, since it satisfies 
$\eta_h(x,t)=\alpha$, $\alpha$ constant, $u_h(x,t)=0$ for all $t\geq 0$ and 
$x\in [0,1]$, provided $\eta_h(x,0)=\alpha$, $u_h(x,0)=0$. So, our attention is 
turned to the periodic ivp for \eqref{eq:3p12} and its standard Galerkin 
semidiscretization \eqref{eq:3p13}--\eqref{eq:3p14}.)

For this purpose, since $d=\eta+\beta$, assume that (suppressing the 
$x$-dependence in the variables), $d_h(0) = \opP\beta$, $u_h(0) = 0$ in 
\eqref{eq:3p14}, and ask whether there exist time-independent solutions of 
\eqref{eq:3p13}--\eqref{eq:3p14} that approximate well the steady state solution 
$d=\beta$, $u=0$ of the continuous problem. Taking $u_h=0$ in \eqref{eq:3p13} we 
see that a steady-state solution $d_h$ must satisfy $(d_hd_{hx},\phi) 
=(d_h\beta^\prime,\phi)$, for all $\phi$ in $S_{h,p}$, from which it is evident 
that the source term $\beta$ should be replaced by some approximations 
$\beta_h\in S_{h,p}$ thereof. Moreover for the equation $(d_hd_{h,x},\phi) = 
(d_h\beta_h^\prime,\phi)$ to hold for $\phi\in S_{h,p}$, (this will imply that 
$d_h=\beta_h$, i.e.\ good balance), it is necessary that the integrals on each 
subinterval $[x_i,x_{i+1}]$ that contribute to these $L^2$ inner products should 
be evaluated \emph{exactly}. Since both integrands are polynomials of degree at 
most $3r-4$ on each subinterval, if an $s$-point Gauss quadrature rule is used 
(recall that such a rule is exact for polynomials of degree at most $2s-1$), 
then it should hold that $s\geq \tfrac 3 2 (r-1)$. For example, in the case of 
cubic splines $(r=4)$, a 5-point Gauss rule must be used. Therefore, although a 
3-point Gauss is enough to preserve the optimal-order $\mathcal O(h^4)$ 
$L^2$-error estimate, good balance of the solution with cubic splines requires 
that a 5-point Gauss rule be used. This is confirmed by the results of the 
following experiment. We solve the periodic ivp for \eqref{eq:3p12} on $[0,1]$ 
by \eqref{eq:3p13}--\eqref{eq:3p14} using cubic splines for the spatial 
discretization on a uniform mesh and taking $\beta(x) = 
1-0.3\exp(-1000(x-0.5)^2)$, $h=0.02$, $k=0.01$, $u_h(0)=0$, $d_h(0)=\opP\beta$. 
Table \ref{tab:3} shows the error $d_h(1)-d_h(0)$ (where 
$d_h(1)=d_h\big|_{T=1}$) in the $L^2$ and $L^\infty$ norms when the analytical 
formula of $\beta$ or $\beta_h=\opP\beta$ is taken in the source term, and a 3- 
or a 5-point Gauss rule is used. It is evident that when $b_h=\opP\beta$ and 
\begin{table}[htbp] \tt 
\centering
{\small 
\begin{tabular}[b]{|*{5}{c|}}
\hline
$d_h(0)$ & \rm $\beta$ in source term & 
\parbox{5em}{\rm\setlength{\baselineskip}{0pt}$s$ (-point\\ Gauss rule)} 
& $\|d_h(1)-d_h(0)\|$ & $\|d_h(1)-d_h(0)\|_\infty$ \\ \hline
$\opP\beta$ &  \rm analytical formula & 3 & 1.8191e-4\phantom{0} & 
8.3845e-4\phantom{0} \\ \hline 
$\opP\beta$ & $\beta_h=\opP\beta$ & 3 & 1.2204e-6\phantom{0} & 
4.7085e-6\phantom{0} \\ \hline 
$\opP\beta$ & $\beta_h=\opP\beta$ & 5 & 3.7458e-15 & 1.0214e-14 \\ \hline 
\end{tabular}
}
\caption{Treatment of source terms and effect of quadrature in 
\eqref{eq:3p13}--\eqref{eq:3p14}, cubic splines and RK4, $h=0.02$, $k=0.01$, 
$T=1$ \label{tab:3}} 
\end{table}
a 5-point Gauss quadrature rule is used, the scheme is well balanced to roundoff 
and there is no influence of the time-stepping error. It should be noted that 
similar results were found when $d_h(0)$ and $\beta_h$ were taken as the cubic 
spline interpolant of $\beta$ at the nodes, and when piecewise smooth bottom 
profiles, e.g.\ like a parabolic perturbation of $\beta=1$ supported in the interval of 
$[0,1]$, were considered.

\section*{Acknowledgement}
\label{sec:ack} 
This work was partially supported by the project ``Innovative Actions in 
Environmental Research and Development (PErAn)'' (MIS 5002358), implemented 
under the ``Action for the strategic development of the Research and 
Technological sector'' funded by the Operational Program ``Competitiveness, and 
Innovation'' (NSRF 2014-2020) and cofinanced by Greece and the EU (European 
Regional Development Fund).

\bibliographystyle{elsarticle-num} 
\bibliography{ref}

\end{document}